\batchmode
\documentclass[11pt]{article}

\usepackage{color}
\usepackage{mathtools}

\newcommand{\be}{\begin{equation}}
\newcommand{\ee}{\end{equation}}
\newcommand{\bea}{\begin{eqnarray}}
\newcommand{\eea}{\end{eqnarray}}
\newcommand{\beas}{\begin{eqnarray*}}
\newcommand{\eeas}{\end{eqnarray*}}
\newcommand{\ba}{\begin{array}}
\newcommand{\ea}{\end{array}}

\definecolor{armygreen}{rgb}{0.29, 0.33, 0.13}

\newcommand{\real}{\mbox{$\mathbb{R}$}}

\newcommand{\Grad}{\ensuremath{\nabla}}

\newcommand{\bfe}{\ensuremath{\mathbf{e}}}

\def\XXint#1#2#3{{\setbox0=\hbox{$#1{#2#3}{\int}$}
     \vcenter{\hbox{$#2#3$}}\kern-.5\wd0}}

\newcommand{\mcT}{\ensuremath{\mathcal{T}}}

\newcommand{\brOmega}{\ensuremath{\breve{\Omega}}}

\newcommand*{\sigtens}{\underline{\bm{\sigma}}}
\newcommand*{\tautens}{\underline{\bm{\tau}}}
\newcommand*{\rhotens}{\underline{\bm{\rho}}}
\newcommand*{\xitens}{\underline{\bm{\xi}}}

\newcommand*{\Ltens}{\underline{\mbf{L}}^{2}}

\newcommand*{\Htens}{\underline{\mbf{H}}}

\newcommand*{\regdiv}{\nabla \cdot}

\newcommand*{\axidiv}{\nabla_{\text{axi}}\cdot}
\newcommand*{\taxidiv}{\text{div}_{\text{axi}}}
\newcommand*{\axicurl}{\nabla_{\text{ac}}}

\newcommand*{\Tr}[1]{\text{tr}({#1})}

\newcommand*{\epstens}{\underline{\boldsymbol{\epsilon}}}
\newcommand*{\bSigma}{\boldsymbol{\Sigma}}
\newcommand*{\bS}{\boldsymbol{S}}
\newcommand*{\bPi}{\boldsymbol{\Pi}}

\newcommand{\mbf}[1]{\mathbf{#1}}
\newcommand*{\ub}[1]{\underline{\mathbf{#1}}}

\usepackage{braket,amsfonts}
\usepackage{blkarray}
\usepackage{amsmath, amssymb, amsbsy, amsthm}
\usepackage{bm, bbm}
\usepackage{array,tabu,multirow}
\usepackage{graphicx,epstopdf}
\usepackage{caption}
\usepackage{tikz}
\usepackage{tikz-qtree}
\usepackage{enumitem}
\usepackage{pdflscape}
\usepackage{geometry}
\usepackage{afterpage}
\usepackage{multirow}
\usepackage{afterpage}
\usepackage{siunitx}
\usepackage{subfig}
\usepackage{cleveref}

\newtheorem{lemma}{Lemma}
\newtheorem{corollary}{Corollary}
\newtheorem{theorem}{Theorem}

\renewcommand{\theequation}{\thesection.\arabic{equation}}

\def\qed{\hbox{\vrule width 6pt height 6pt depth 0pt}}

\title{Approximation of the Axisymmetric Elasticity Equations with Weak Symmetry} 
\author{
	Alistair Bentley \thanks{CarMax, 12800 Tuckahoe Creek Pkwy, 
	Richmond, VA 23238, USA. email: {\tt alistairbntl@gmail.com}.} 
	\and
	V.J.~Ervin\thanks{School of Mathematical and Statistical Sciences,
	  Clemson University, Clemson, South Carolina 29634-0975, USA.
	  email: {\tt vjervin@clemson.edu}. }
 }

\date{\today}

\begin{document}
\maketitle

\begin{abstract}
In this article we consider the linear elasticity problem in an axisymmetric three dimensional domain, 
with data which are axisymmetric and have zero angular component. The weak formulation of the
the three dimensional problem reduces to a two dimensional problem on the meridian domain,
involving weighted integrals. The problem is formulated in a mixed method framework with both
the stress and displacement treated as unknowns. The symmetry condition for the stress tensor
is weakly imposed. Well posedness of the continuous weak formulation and its discretization
are shown. Two approximation spaces are discussed and corresponding numerical computations
presented.

\end{abstract}



\setcounter{equation}{0}
\setcounter{figure}{0}
\setcounter{table}{0}
\setcounter{theorem}{0}
\setcounter{lemma}{0}
\setcounter{corollary}{0}
\setcounter{definition}{0}
%
\section{Introduction}
\label{sec:introduction}

During the past twenty years, a number of papers have emerged in the numerical analysis literature investigating 
three-dimensional axisymmetric problems.  This class of problem has attracted attention because a 
three-dimensional axisymmetric problem can be reduced to a two-dimensional problem when cylindrical 
coordinates are used (see Figure \ref{fig:axisymmetric_domain}).  It is well recognized that the computational effort 
required to solve a two-dimensional problem is significantly less that the computational effort needed to solve a three-dimensional problem.  

\begin{figure}[t]
\center
\includegraphics[scale=0.4]{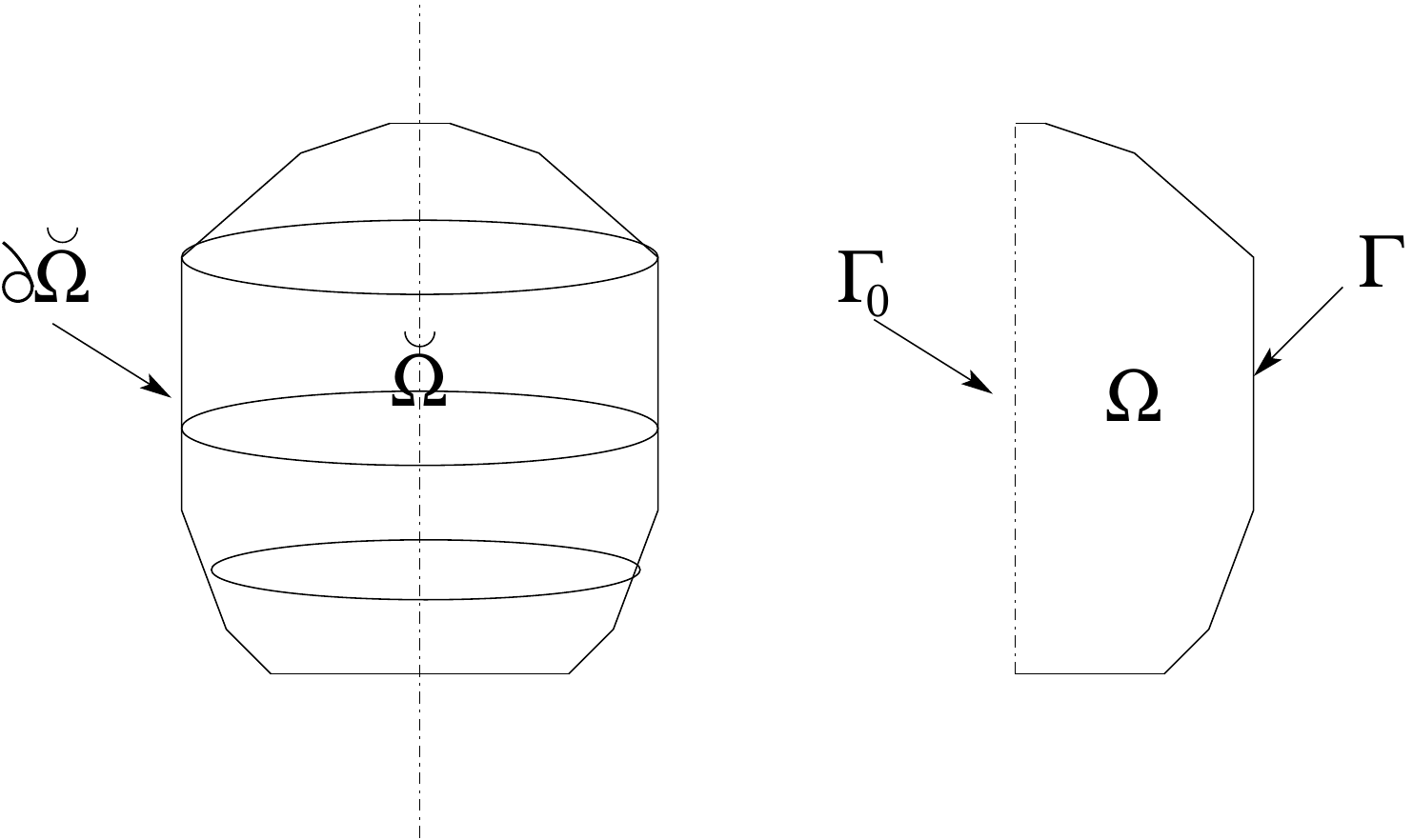}
\caption{Axisymmetric Domain}
\label{fig:axisymmetric_domain}
\end{figure}

For axisymmetric problems, Mercier and Raugel \cite{MercierAndRaugel} undertook one of the first finite element analyses 
of these problems.  In \cite{BernardiDaugeMaday}, Bernardi, Dauge and Maday 
studied the axisymmetric formulation of a number of standard problems (including Laplace, Stokes and Maxwell equations),
 and introduced tools for analyzing axisymmetric spectral methods. 
Assous, Ciarlet, et al. 
investigated the numerical approximation of the axisymmetric solution 
of the static and time dependent Maxwell equations in \cite{AssousAndCiarlet1, AssousAndCiarlet2}.  
Following these papers, a number of studies analyzing different axisymmmetric problems appeared.  
Notably, a computational framework for the axisymmetric Poisson equation was developed by Ciarlet, Jung et al. in \cite{CiarletEtAl}, 
and a computational framework for div-curl systems was presented by Copeland, Gopalakrishnan, and Pasciak
 in \cite{CopelandEtAl}.  More recently,  \cite{Oh} Oh used finite element exterior calculus techniques to study the 
 axisymmetric Hodge Laplacian problem.
 For axisymmetric fluid dynamics problems,  in \cite{BermudezEtAl}, 
 Berm\'{u}dez, Reales, et al. used 
 axisymmetry to reduce the dimension of an eddy current model, and in \cite{AnayaEtAl} Anaya, Mora et al. developed
  a computational framework for axisymmetric Brinkman flows.  
 The axisymmetric Stokes and Darcy problems have been studied in 
 \cite{BelhachmiEtAl, Ervin1, ErvinAndJenkins, LeeAndLi, Ying}.  
 A coupled axisymmetric Stokes-Darcy problem was investigated by Ervin in \cite{Ervin2}.

The finite element approximation of the linear elasticity problem has been extensively studied 
(see \cite{BBF1} for a detailed discussion). 
For many years, the only known stable finite elements for the mixed method formulation, involving the stress and displacement,
 used macro-elements in which the stress tensor was approximated on a finer mesh than the displacement vector \cite{ADG, JM, V3}.  
 In \cite{AW} Arnold and Winther developed a stable pair of piecewise polynomials with respect to a single triangulation. 
 These elements, however, carry a significant computational cost since the lowest order representation uses 
 24 degrees of freedom per triangle.

The major difficulty to creating a stable finite element scheme for the mixed formulation of the linear elasticity problem
 is in enforcing the symmetry of the stress tensor, which represents the law of conservation of angular momentum.  
 To avoid enforcing symmetry in the stress tensor strongly, a Lagrangian multiplier can be used to weakly enforces symmetry in the stress tensor \cite{ABD,AFW,FF,M1,S,S1}.

The form of differential operators expressed in cylindrical coordinates (e.g. the addition of a $\frac{1}{r}$ term) is an important 
reason why the numerical analysis for the finite element approximation to the axisymmetric 
linear elasticity problem is challenging.  A consequence of this radial scaling
 is that the gradient and divergence operators do not map polynomial spaces to polynomial spaces.  This feature 
 makes the construction of suitable inf-sup stable finite element approximation spaces more difficult than in the Cartesian 
 setting.

Following, in Sections \ref{sec:notation}-\ref{sec:axisymmetric_definitions} notation and needed preliminary results
are introduced. 
A continuous weak formulation for the axisymmetric 
linear elasticity problem is presented in Section \ref{sec:axisymmetric_form} and shown to be well posed. 
Then, in Section \ref{sec:discrete_axisymmetric_problem},
the corresponding discrete weak formulation is analyzed, and sufficient conditions for its well posedness established
in terms of the existence of a suitable bounded projection operator.
Shown in Section \ref{sec:disspace} is the existence of projection operators for two well known approximation spaces which,
together with the assumption of boundedness of the projection, establishes the approximation spaces are inf-sup stable.
An error analysis is given in Section \ref{sec:convergence_analysis}. 
The numerical computations presented in Section \ref{sec:computational_results} support the derived
theoretical results. Some concluding remarks are given in Section \ref{sec:elasticity_discussion}

 \setcounter{equation}{0}
\setcounter{figure}{0}
\setcounter{table}{0}
\setcounter{theorem}{0}
\setcounter{lemma}{0}
\setcounter{corollary}{0}
\setcounter{definition}{0}
\section{Notation}
\label{sec:notation}

In this section we introduce the notation used below.  Bold Greek letters (e.g. $\bm{\sigma}$) represent vectors, while bold Greek letters with an underline (e.g. $\sigtens$) denote tensors.  For English letters, bold lowercase letters (e.g. $\mbf{p}$) denote vectors, while bold uppercase letters (e.g. $\mbf{P}$) denote tensors.  Matrices are represented with capital, non-bold letters (e.g. $A$).
Additionally, $\mathbb{M}^{n}$ denote the space of $n\times n$ dimensional real matrices, $\mathbb{S}^{n}$ denote the space of $n \times n$ dimensional real symmetric matrices and $\mathbb{K}^{n}$ denote the space of $n \times n$ dimensional real skew-symmetric matrices.

The space of piecewise polynomials of degree less than or equal to $k$ on a partition, $\mathcal{T}_{h}$, of a domain is denoted as $P_{k}(\mathcal{T}_{h})$.  The polynomials of degree less than or equal to $k$ on a specific domain $T$, or on an element $T \in \mathcal{T}_{h}$, are notated by $P_{k}(T)$.  When referencing a vector or tensor space of polynomials, the notation $(P_{k}(T))^{n}$ and $(P_{k}(T))^{n \times n}$ is used, respectively.

The symmetric gradient operator, $\epstens$ applied to a vector $\mbf{u}$, is given by
\begin{align*}
\epstens(\mbf{u})_{ij} = \frac{1}{2} \left ( \dfrac{\partial \mbf{u}_{i}}{\partial {x}_{j}} + \dfrac{\partial \mbf{u}_{j}}{\partial x_{i}} \right).
\end{align*}
For $\sigtens_{i}$ denoting the  $i$ row of $\sigtens$, the vector $\nabla \cdot \sigtens$ is given by
\begin{align*}
\big( \nabla \cdot \sigtens \big)_{i} =  \nabla \cdot \sigtens_{i} \, .
\end{align*}

The trace operator, $\text{tr}$, is defined as
\begin{align*}
\text{tr} (\sigtens) = \sum_{i=1}^{n} \sigma_{ii}.
\end{align*}

The skew-symmetric part of a tensor $\sigtens$ is defined as
\begin{align*}
as(\sigtens) = \dfrac{1}{2} ( \sigtens - \sigtens^{t})
\end{align*}
where $\sigtens^{t}$ is the transpose of $\sigtens$.  

For $q \in \mathbb{R}$, $\mathcal{S}^{2}(q)$ is defined as 
$\mathcal{S}^{2}(q) = \begin{pmatrix} 0 & q \\ 
-q & 0\end{pmatrix}$, and in two dimensions 
$as(\sigtens)$ can be identified as $as(\sigtens)=
\mathcal{S}^{2}(q)  \text{ where } q =\dfrac{1}{2}( \sigma_{12} - \sigma_{21})$.

For vectors $\mbf{a} = (a_{1}, a_{2})^{t}$ and $\mbf{b} = (b_{1}, b_{2})^{t}$,
\begin{align*}
\nabla_{\text{curl}} \mbf{a} \ = \ 
\begin{pmatrix} - \partial_{y} a_{1}  & \partial_{x} a_{1} \\ - \partial_{y} a_{2}  & \partial_{x} a_{2} \end{pmatrix} \, , \ \ \ \
\mbf{a} \otimes \mbf{b} = \begin{pmatrix} a_{1} b_{1} & a_{1} b_{2} \\ a_{2} b_{1} & a_{2} b_{2} \end{pmatrix}.
\end{align*}
If $\mbf{w} = (w_{1},w_{2})^{t}$ and $\mbf{v} = (v_{1}, v_{2})^{t}$ are vectors, then the two-dimensional wedge product is
\begin{align*}
\mbf{w} \wedge \mbf{v} = w_{1} v_{2} - w_{2} v_{1}.
\end{align*}
For a tensor $\tautens$ and vector $\mbf{v}$, the wedge product is
\begin{align*}
(\tautens \wedge \mbf{v}) = \begin{pmatrix}    \tau_{11}v_{2} -  \tau_{21}v_{1}  \\ 
 \tau_{12}v_{2} -  \tau_{22}v_{1}  
\end{pmatrix}.
\end{align*}

For $\mbf{x} = (x_{1}, x_{2})^{t}$, $\mbf{x}^{\perp}$ is defined as, $\mbf{x}^{\perp} = (x_{2}, -x_{1})^{t}$.

To distinguish between inner product and bilinear forms defined in Cartesian coordinates from those defined in cylindrical coordinates, a $c$ subscript is attached to all Cartesian inner products and bilinear forms.

\setcounter{equation}{0}
\setcounter{figure}{0}
\setcounter{table}{0}
\setcounter{theorem}{0}
\setcounter{lemma}{0}
\setcounter{corollary}{0}
\setcounter{definition}{0}
\section{Variational Formulation}
\label{sec:vform}

As a starting point for the derivation of our weak formulation for the axisymmetric problem, we begin with the
weak (mixed) formulation for the elasticity problem, subject to a weakly enforced symmetry condition for the stress.

For $\sigtens$ denoting the stress tensor, $\mbf{u}$ is the displacement, $\brOmega \subset \mathbb{R}^{3}$
a convex (axisymmetric) domain with Lipschitz continuous boundary,
$\mu$ and $\lambda$ Lam\'{e} constants, the modeling equations of linear elasticity, subject to a fixed boundary,
are given by
\begin{align}
\mathcal{A} \sigtens = \epstens ( \mbf{u} ), & \quad \nabla \cdot \sigtens = \bm{f} \mbox{ in } \brOmega \, ,  \label{eq:elasticity_strong} \\
\mbox{subject to} \ \   \mbf{u} &= \ \mbf{0} \ \ \mbox{ on } \partial \brOmega \, .  \label{eq:elasticity_strong2}
\end{align}
In \eqref{eq:elasticity_strong} the compliance tensor $\mathcal{A}: \mathbb{S}^{n \times n} \rightarrow \mathbb{S}^{n \times n}$ is a bounded, symmetric positive definite operator that, for isotropic materials, takes the form
\begin{align}
\label{eq:complianceTensor}
\mathcal{A} \sigtens & = \frac{1}{2 \mu} \left( \sigtens - \dfrac{\lambda}{2\mu + m \lambda} \Tr{\sigtens} I \right) \, .
\end{align}

In order to describe the weak formulation we introduce the following function spaces.
\begin{align*}
\begin{split}
&L^{2}(\brOmega) = \{ v : \int_{\brOmega} v^{2} \; d \brOmega < \infty \}, \\
&\mbf{L}^{2}(\brOmega) = \{ \mbf{v} : v_{i} \in L^{2}(\brOmega) \text{ for } i = 1, \cdots, n \}, \\
&\mbf{H}^{1}(\brOmega) = \{ \mbf{v} : v_{i} \in L^{2}(\brOmega) , \ \nabla v_{i} \in \mbf{L}^{2}(\brOmega) \, ,  \text{ for } i = 1, \cdots, n \}, \\
&\Ltens(\brOmega; \mathbb{M}^{n}) = \{\sigtens \in \mathbb{M}^{n} : \sigma_{ij} \in L^{2}(\brOmega) \, , \text{ for } i,j = 1, \cdots , n \}, \\
&\Ltens(\brOmega; \mathbb{K}^{n}) = \{\sigtens \in \mathbb{K}^{n} : \sigma_{ij} \in L^{2}(\brOmega) \, , \text{ for } i,j = 1, \cdots , n \}, \\
&\Htens^{1}(\brOmega, \mathbb{M}^{n}) = \{ \sigtens \in \Ltens(\brOmega, \mathbb{M}^{n}) : \nabla \sigma_{ij} 
 \in \mbf{L}^{2}(\brOmega ) \, , \text{ for } i,j = 1, \cdots , n \} , \ \ \text{ and } \\
&\Htens(\text{div}, \brOmega, \mathbb{M}^{n}) = \{ \sigtens \in \Ltens(\brOmega, \mathbb{M}^{n}) : \nabla \cdot \sigtens \in \mbf{L}^{2}(\brOmega ) \} \, .
\end{split}
\end{align*}

Letting $X = \Htens(\text{div}, \brOmega, \mathbb{M}^{n})$, $Q = \mbf{L}^{2}(\brOmega)$, and $W = \Ltens(\brOmega; \mathbb{K}^{n})$.
Then, the weak formulation is given by \cite{ABD, FF, M1, S, S1}: \textit{ Given, $\mbf{f} \in Q$,
determine  $(\sigtens,\mbf{u}, \rhotens) \in X \times Q \times W$ such that, for all $(\tautens, \mbf{v}, \xitens) \in X \times Q \times W$}
\begin{align}
\label{eq:weak_symmetry_elasticity_1}
\int_{\brOmega} (\mathcal{A} \sigtens : \tautens + \nabla \cdot \tautens \cdot \mbf{u} + \tautens : \rhotens) \; d \brOmega &= 0 \\
\label{eq:weak_symmetry_elasticity_2}
\int_{\brOmega} \nabla \cdot \sigtens \cdot \mbf{v} \; d \brOmega &=\int_{\brOmega} \mbf{f} \cdot \mbf{v} \; d \brOmega\\
\label{eq:weak_symmetry_elasticity_3}
\int_{\brOmega} \sigtens : \xitens\; d \brOmega & = 0.
\end{align}

With the inner products,
\begin{align}
\label{eq:weak_symmetry_elasticity}
\begin{split}
a_{c}(\cdot,\cdot)&: X \times X \rightarrow \mathbb{R}, \quad a_{c}(\sigtens, \tautens) := \int_{\brOmega} \mathcal{A} \sigtens : \tautens \; d \brOmega, \\
b_{c}(\cdot,\cdot)&: Q \times X \rightarrow \mathbb{R}, \quad
b_{c} (\mbf{u}, \tautens )   := \int_{\brOmega} (\nabla \cdot \tautens) \cdot \mbf{u} \; d \brOmega, \\
c_{c}(\cdot,\cdot)&: W \times X \rightarrow \mathbb{R}, \quad
c_{c} (\rhotens, \tautens )  := \int_{\brOmega} \rhotens : \tautens \; d \brOmega
\end{split}
\end{align}
and taking $A_{c}(\sigtens, \tautens) = a_{c} (\sigtens, \tautens)$ and $B_{c}(\tautens, (\mbf{u}, \rhotens)) = b_{c}(\mbf{u},\tautens) + c_{c} (\rhotens,\tautens)$, $(\ref{eq:weak_symmetry_elasticity_1})-(\ref{eq:weak_symmetry_elasticity_3})$ can 
be rewritten in the familiar saddle-point formulation:
\textit{Given, $\mbf{f} \in Q$,
determine  $(\sigtens, (\mbf{u},\rhotens) ) \in X \times (Q \times W)$ such that, 
for all $(\tautens,\mbf{v},\xitens) \in X \times (Q \times W)$}
\begin{align}
\label{eq:simplified_weak_symmetry_elasticity}
\begin{split}
A(\sigtens,\tautens) + B(\tautens,(\mbf{u},\rhotens)) & = 0 \\
B(\sigtens,(\mbf{v},\xitens)) &= (\mbf{f}, \mbf{v}) \, .
\end{split}
\end{align}

For a detailed analysis of $(\ref{eq:weak_symmetry_elasticity_1})-(\ref{eq:weak_symmetry_elasticity_3})$, see \cite{BBF1}.

 \setcounter{equation}{0}
\setcounter{figure}{0}
\setcounter{table}{0}
\setcounter{theorem}{0}
\setcounter{lemma}{0}
\setcounter{corollary}{0}
\setcounter{definition}{0}
\section{Axisymmetric Function Spaces}
\label{sec:axisymmetric_definitions}

When the three dimensional axisymmetric linear elasticity problem is expressed in cylindrical coordinates, it can be expressed 
as a decoupled meridian and azimuthal problem.    Changing the coordinate system from Cartesian to cylindrical, however, 
alters the algebraic form of differential operators and requires a new set of function spaces and notation.  In this section, 
we introduce the key changes needed to present and discuss the meridian axisymmetric linear elasticity problem.  
Appendix \ref{sec:axisymmetric_elasticity} provides additional details on cylindrical coordinates and the procedure for decoupling the 
axisymmetric problem.

For axisymmetric vectors $\mbf{u} =  u_{r} \mbf{e}_{r} \, + \, u_{z} \mbf{e}_{z} = (u_{r}, u_{z})^{t}$, we define the gradient operators $\nabla$ and $\nabla_{\text{axi}}$ as
\begin{align}
\label{eq:gradient_operators}
\nabla \; \mbf{u} = \begin{pmatrix} \dfrac{\partial u_{r}}{\partial r} & \dfrac{\partial u_{r}}{\partial z} \\[3mm] \dfrac{\partial u_{z}}{\partial r} & \dfrac{\partial u_{z}}{\partial z} \end{pmatrix}, \text{ and } \; \nabla_{\text{axi}} \; \mbf{u} = \begin{pmatrix} \dfrac{\partial u_{r}}{\partial r} & 0 & \dfrac{\partial u_{r}}{\partial z} \\
0 & \frac{1}{r} u_{r} & 0 \\
\dfrac{\partial u_{z}}{\partial r} & 0 & \dfrac{\partial u_{z}}{\partial z} \end{pmatrix}.
\end{align}

Note that it is necessary to represent the gradient and axisymmetric gradient as tensors with different sizes because the 
non-constant nature of the cylindrical coordinate unit vectors creates additional terms in axisymmetric derivatives.  
However, in order to express the meridian problem using a two-dimensional formulation, we represent the 
tensor $\nabla_{\text{axi}} \mbf{u}$ as an ordered pair made up of a tensor and a scalar function.  That is
\begin{align}
\label{eq:axi_div_gradient_11}
\nabla_{\text{axi}} \; \mbf{u} = (\nabla \; \mbf{u}, \frac{1}{r} u_{r}).
\end{align}  

Next, for the axisymmetric vector $\mbf{u} = (u_{r}, u_{z})^{t}$, the divergence operators $\regdiv$ and $\axidiv$ are defined as
\begin{align}
\label{eq:del_rz}
\regdiv \mbf{u} = \dfrac{\partial u_{r}}{\partial r} + \dfrac{\partial u_{z}}{\partial z}, \text{ and } 
\nabla_{\text{axi}} \cdot \mbf{u} = \dfrac{1}{r} \dfrac{\partial (r \; u_{r})}{\partial r} + \dfrac{\partial u_{z}}{\partial z} 
=  \nabla \cdot \mbf{u} \, + \, \dfrac{1}{r} u_{r}  \ .
\end{align}

As alluded to in \eqref{eq:axi_div_gradient_11}, the stress tensor that appears in the meridian problem can be 
represented as $(\sigtens, \sigma)$, where $\sigtens$ denotes an $\mathbb{M}^{2}$ tensor function and $\sigma$ 
represents a scalar function.  The divergence of the meridian stress tensor is
\begin{align*}
\axidiv (\sigtens, \sigma) = \begin{pmatrix} \axidiv \sigtens_{1} - \dfrac{1}{r} \sigma \\ \axidiv \sigtens_{2} \end{pmatrix}.
\end{align*}

At times, the axisymmetric divergence operator will also be applied to an $\mathbb{M}^{2}$ tensor function $\sigtens$, in which case
\begin{align*}
\axidiv \sigtens = \axidiv (\sigtens, 0) = \begin{pmatrix} \axidiv \sigtens_{1} \\ \axidiv \sigtens_{2} \end{pmatrix}.
\end{align*}

Note that for the skew symmetric component of $(\sigtens, \sigma)$ we have
\begin{align*}
as((\sigtens, \sigma)) = as(\sigtens) = \mathcal{S}^{2}(q), \; \text{ where } \; q = \frac{1}{2} (\sigma_{12} - \sigma_{21}).
\end{align*}

The curl of an axisymmetric scalar function $p$ is denoted by $\axicurl$ and is defined as
\begin{align}
\label{eq:axicurl_def}
\axicurl \; p = \left( \dfrac{\partial p}{\partial z} , - \dfrac{1}{r} \dfrac{\partial(r \; p)}{\partial r} \right).
\end{align}
Note that $\axicurl$ returns a row-vector.  For a vector function $\mbf{p} = (p_{r}, p_{z})^{t}$ we have
\begin{align}
\axicurl \; \mbf{p} = \begin{pmatrix} \axicurl \; p_{r} \\[3mm] \axicurl \; p_{z} \end{pmatrix}. 
\end{align}

In addition to the divergence and curl, the cylindrical coordinate inner product also takes a different form from the 
Cartesian inner product.  
As illustrates in Figure \ref{fig:axisymmetric_domain}, $\Omega$ denotes the half cross section of the
axisymmetric domain $\brOmega$.
Consider the change of variables for a Cartesian function $\breve{p} \in L^{2}(\brOmega)$ into 
cylindrical coordinates
\begin{align}
\label{eq:cylindrical_inner_product}
\int_{\brOmega} \breve{p}^{2} \; d \brOmega = \iint_{\Omega} \int_{\theta = 0}^{2 \pi} \; p^{2} r \; d \theta \; dr \; dz  .
\end{align}
Notice the $r = r(\mathbf{x})$ scaling in the measure.  In the axisymmetric setting, $p \equiv p(r,z)$ and the $\theta$ integral can be computed to give a factor of $2 \pi$.  As this term is a constant factor in all such integrals arising, we omit it. To distinguish the cylindrical coordinate inner product from the Cartesian inner product, we use the following notation
\begin{align*}
(p, q) = \int_{\Omega} p \; q \; r \; dr \; d z .
\end{align*}

To account for this scaling in the inner product,  we introduce the following function spaces
\begin{align*}
\begin{split}
_{\alpha}L^{2}(\Omega) &= \{ v : \int_{\Omega} v^{2} r^{\alpha}  \; dr \; dz < \infty \}, \\
_{\alpha} \mbf{L}^{2} (\Omega) & = \{ \mbf{v} \in \mathbb{R}^{n} : v_{i} \in \; _{\alpha}L^{2}(\Omega) \text{ for } i = 1, ..., n \}, \\
_{\alpha} \Ltens (\Omega, \mathbb{M}^{n}) &= \{ \sigtens \in \mathbb{M}^{n} : \sigma_{ij} \in \; _{\alpha} L^{2} (\Omega) \text{ for } i=1, \cdots n \text{ and } j = 1, \cdots n \},  \ \ \mbox{ and } \\
_{\alpha} \Ltens (\Omega, \mathbb{K}^{n}) &= \{ \sigtens \in \mathbb{K}^{n} : \sigma_{ij} \in \; _{\alpha} L^{2} (\Omega) \text{ for } i=1, \cdots n \text{ and } j = 1, \cdots n \}.
\end{split}
\end{align*}

The norms associated with these $_{\alpha} L^{2}$ spaces are 
\begin{align*}
&\| v \|^{2}_{_{\alpha} L^{2}(\Omega) } = \int_{\Omega} v^{2} r^{\alpha}  \; dr \; dz ,  \quad
\| \mbf{v} \|^{2}_{_{\alpha} \mbf{L}^{2} (\Omega)} = \sum_{i=1}^{n} \| v_{i} \|^{2}_{_{\alpha} L^{2} (\Omega)} , \\
&\| \sigtens \|^{2}_{_{\alpha} \Ltens (\Omega, \mathbb{M}^{n})} = \sum_{i=1}^{n} \sum_{j=1}^{n} \| \sigma_{ij} \|^{2}_{_{\alpha} L^{2} (\Omega)} \; \; \text{ and } \; \; \| \sigtens \|^{2}_{_{\alpha} \Ltens (\Omega, \mathbb{K}^{n})} = \sum_{i=1}^{n} \sum_{j=1}^{n} \| \sigma_{ij} \|^{2}_{_{\alpha} L^{2} (\Omega)}.
\end{align*}

In addition to the $_{\alpha} L^{2}$ spaces, the elasticity problem requires divergence spaces for the stress tensors.  These spaces are
\begin{align*}
\begin{split}
_{\alpha} \mbf{H} (\taxidiv, \Omega) &= \{ \mbf{v} \in \; _{\alpha}\mbf{L}^{2}(\Omega) : \axidiv \mbf{v} \in \; _{\alpha}L^{2}(\Omega) \}, \\
_{\alpha} \ub{H} (\taxidiv, \Omega ;  \mathbb{M}^{n}) & = \{ \sigtens \in \; _{\alpha} \Ltens(\Omega; \mathbb{M}^{n}) : \axidiv \sigtens \in \; _{\alpha} \mbf{L}^{2} (\Omega) \},  \ \ \mbox{ and }  \\
_{\alpha} \ub{H} (\taxidiv, \Omega ;  \mathbb{K}^{n}) & = \{ \sigtens \in \; _{\alpha}\Ltens(\Omega; \mathbb{K}^{n}) : \axidiv  \sigtens \in \; _{\alpha} \mbf{L}^{2} (\Omega) \}.
\end{split}
\end{align*}
with norms
\begin{align*}
\| \mbf{v} \|^{2}_{_{\alpha} \mbf{H} (\taxidiv, \Omega )} &= \| \axidiv \mbf{v} \|^{2}_{_{\alpha}L^{2}(\Omega)} + \| \mbf{v} \|^{2}_{_{\alpha}\mbf{L}^{2}(\Omega)}, \\
\| \sigtens \|^{2}_{_{\alpha} \Htens(\taxidiv, \Omega; \; \mathbb{M}^{n})} &= \| \axidiv  \sigtens \|^{2}_{_{\alpha} \mbf{L}^{2}(\Omega)} + \| \sigtens \|^{2}_{_{\alpha} \Ltens(\Omega; \mathbb{M}^{n})}, \\ 
\| \sigtens \|^{2}_{_{\alpha} \Htens(\taxidiv, \Omega; \; \mathbb{K}^{n})} &= \| \axidiv  \sigtens \|^{2}_{_{\alpha} \mbf{L}^{2}(\Omega)} + \| \sigtens \|^{2}_{_{\alpha} \Ltens(\Omega; \mathbb{K}^{n})}.
\end{align*}

For $\zeta$ a nonnegative integer and $v$ a $\zeta$ times differentiable function, let
\begin{align*}
\nabla^{\zeta} v = \begin{bmatrix} \dfrac{\partial^{\zeta} v}{(\partial r)^{\zeta}  }, \dfrac{\partial^{\zeta} v}{(\partial r)^{\zeta-1} \partial z},  \cdots, \dfrac{\partial^{\zeta} v }{(\partial r) (\partial z)^{\zeta -1}},  \dfrac{\partial^{\zeta}v}{(\partial z)^{\zeta}} \end{bmatrix}.
\end{align*}
Then, 
\begin{align*}
_{\alpha} H^{k} (\Omega) &= \{ v \in {}_{\alpha} L^{2}(\Omega) : \nabla^{\zeta} v \in {}_{\alpha} \mbf{L}^{2} (\Omega) \text{ for all } \zeta \leq k \}, \\
_{\alpha} \mbf{H}^{k} (\Omega) &= \{ \mbf{v} \in {}_{\alpha} \mbf{L}^{2}(\Omega) : \nabla^{\zeta} v_{i} \in {}_{\alpha}\mbf{L}^{2}(\Omega) \text{ for all } \zeta \leq k \text{ and } i = 1,2, \cdots n \},  \ \ \mbox{ and } \\
_{\alpha} \Htens^{k} (\Omega; \mathbb{M}^{n}) &= \{ \sigtens \in {}_{\alpha} \mbf{L}^{2}(\Omega; \mathbb{M}^{n}) : \nabla^{\zeta} \sigma_{i,j} \in {}_{\alpha}\mbf{L}^{2}(\Omega) \text{ for all } \zeta \leq k \text{ and } i,j = 1,2, \cdots n \},
\end{align*}
with norms
\begin{align*}
\| v \|^{2}_{{}_{\alpha}H^{k}(\Omega)} &= \| v \|^{2}_{{}_{\alpha} L^{2}(\Omega)} + \sum_{\zeta=1}^{k} \| \nabla^{\zeta} v \|^{2}_{{}_{\alpha}\mbf{L}^{2}(\Omega)}, \\
\| \mbf{v} \|^{2}_{{}_{\alpha} \mbf{H}^{k}(\Omega)} &= \| \mbf{v} \|^{2}_{{}_{\alpha} \mbf{L}^{2}(\Omega)} + \sum_{i=1}^{n} \sum_{\zeta=1}^{k} \| \nabla^{\zeta} v_{i} \|^{2}_{{}_{\alpha}\mbf{L}^{2}(\Omega)}, \\
\| \sigtens \|^{2}_{{}_{\alpha} \Htens^{k}(\Omega; \mathbb{M}^{n})} &= \| \sigtens \|^{2}_{{}_{\alpha} \Ltens(\Omega; \mathbb{M}^{n})} + \sum_{i=1}^{n} \sum_{j=1}^{n} \sum_{\zeta=1}^{k} \| \nabla^{\zeta} \sigma_{ij} \|^{2}_{{}_{\alpha}\mbf{L}^{2}(\Omega)}.
\end{align*}

Next we consider some subtle details related to function spaces containing axisymmetric derivative terms.  To begin, 
using \eqref{eq:gradient_operators}, 
\begin{align*}
\| \nabla_{\text{axi}} \mbf{v} \|^{2}_{{}_{1}L^{2}(\Omega)} = \int_{\Omega} \nabla_{\text{axi}} \mbf{v} : \nabla_{\text{axi}} \mbf{v} \; r \; d \Omega = \int_{\Omega} \nabla \mbf{v} : \nabla \mbf{v}\;  r \; d \Omega + \int_{\Omega} \frac{1}{r} v_{r}^{2} \; d \Omega.
\end{align*}

Therefore, in order that $\| \nabla_{\text{axi}} \mbf{v} \|_{{}_{1}L^{2}(\Omega)} < \infty $, it is necessary 
for $v_{r} \in {}_{1}H^{1}(\Omega)$ and $v_{r} \in {}_{-1}L^{2}(\Omega)$.  To denote this important subspace 
of ${}_{1}H^{1}(\Omega)$, we define
\begin{align*}
{}_{1}V^{1}(\Omega)  = \{ v \in {}_{1}H^{1}(\Omega) : v \in {}_{-1}L^{2}(\Omega) \}  \ \ 
& \mbox{with associated norm } \ \  \| v \|_{{}_{1}V^{1}(\Omega)} \ = \ \left( \| v \|_{{}_{-1}L^{2}(\Omega)}^{2} \ + \ 
 \| v \|_{{}_{1}H^{1}(\Omega)}^{2} \right)^{1/2} . 
\end{align*} 
Also, we introduce
\begin{align*}
 {}_{1}\mbf{VH}^{1}(\Omega)  &=  \ \{ \mbf{v} \, = \, (v_{r} , v_{z})^{t} \, : \, 
 v_{r} \in {}_{1}V^{1}(\Omega)  \, , \  v_{z} \in {}_{1}H^{1}(\Omega)  \}   \\
 \mbox{with associated norm } \ \  \| \mbf{v} \|_{{}_{1}\mbf{VH}^{1}(\Omega)}
 &= \ \left( \| v_{r} \|_{{}_{1}V^{1}(\Omega)}^{2} \ + \ \| v_{z} \|_{{}_{1}H^{1}(\Omega)}^{2} \right)^{1/2} \, .
\end{align*}

It is also important to observe, that unlike in the Cartesian setting, 
${}_{1}\mbf{H}^{1}(\Omega) \not \subset {}_{1}\mbf{H}(\taxidiv, \Omega)$.

When referencing a function space whose functions have a vanishing trace along the boundary segment $\Gamma$, we use a zero subscript, i.e.,
\begin{align*}
{}_{1}H_{0}^{1}(\Omega) = \{ v \in {}_{1}H_{0}^{1}(\Omega) : v = 0 \text{ on } \Gamma \}.
\end{align*}

Note that $\Gamma$ here does not include the rotation axis portion of the boundary of $\Omega$ as illustrated in Figure \ref{fig:axisymmetric_domain}.

As eluded to above, in transforming from $\breve{\Omega}$ to $\Omega$, we have the following relationships.
\begin{lemma} \cite[Proposition 1]{BelhachmiEtAl}   \label{redvec}
The space of axisymmetric vector fields in $H^{1}(\breve{\Omega})^{3}$ with zero angular component is isomorphic to
${}_{1}\mbf{VH}^{1}(\Omega)$.
\end{lemma}

\begin{lemma}  \label{redten}
The space of axisymmetric tensors in $\Htens^{1}(\brOmega, \mathbb{M}^{3})$ with zero azimuthal components
is isomorphic to
\begin{equation}
\left\{ \tautens \, = \, \left( \begin{array}{ccc}
\tau_{r r}  &   0   &  \tau_{r z}   \\
0   &  \tau_{\theta  \theta}   &   0   \\
\tau_{z r}  &   0   &  \tau_{z z}   \end{array} \right) \, : \, \tau_{r r}, \, \tau_{\theta  \theta}, \, \tau_{z  z} \in  {}_{1}H^{1}(\Omega), \
 \tau_{r z}, \,  \tau_{z r} \in  {}_{1}V^{1}(\Omega), \ ( \tau_{r r} \, - \, \tau_{\theta  \theta}) \in   {}_{-1}L^{2}(\Omega) \right\}.
 \label{ref34}
\end{equation} 
\end{lemma}
\begin{proof}
The representation of an axisymmetric tensor $\tautens$, in cylindrical 
coordinates
with zero azimuthal components is given by the tensor in \eqref{ref34}. 
In terms of the unit coordinate vectors
$\bfe_{r}$, $\bfe_{\theta}$, $\bfe_{z}$, $\Grad \tautens$ may be written as
\begin{align*}
\Grad \tautens &= \ 
\partial_{r} \tau_{r r}  \, \bfe_{r} \otimes \bfe_{r} \otimes \bfe_{r} \ + \      0    \, \bfe_{r} \otimes \bfe_{r} \otimes \bfe_{\theta}   
\ + \    \partial_{z} \tau_{r r}  \,   \bfe_{r} \otimes \bfe_{r} \otimes \bfe_{z}  \\
& \ \  + \  0  \, \bfe_{r} \otimes \bfe_{\theta} \otimes \bfe_{r} 
\ + \   \frac{1}{r} ( \tau_{r r} \, - \, \tau_{\theta \theta} )   \, \bfe_{r} \otimes \bfe_{\theta} \otimes \bfe_{\theta}  
\ + \    0  \,   \bfe_{r} \otimes \bfe_{\theta} \otimes \bfe_{z}    \\
& \ \ + \  \partial_{r} \tau_{r z}   \, \bfe_{r} \otimes \bfe_{z} \otimes \bfe_{r} 
 \ + \     0       \, \bfe_{r} \otimes \bfe_{z} \otimes \bfe_{\theta}  
\ + \     \partial_{z} \tau_{r z} \, \bfe_{r} \otimes \bfe_{z} \otimes \bfe_{z}  \\
& \ \ + \  0  \, \bfe_{\theta} \otimes \bfe_{r} \otimes \bfe_{r} 
\ + \   \frac{1}{r} ( \tau_{r r} \, - \, \tau_{\theta \theta} )  \, \bfe_{\theta} \otimes \bfe_{r} \otimes \bfe_{\theta}   
\ + \    0 \,  \bfe_{\theta} \otimes \bfe_{r} \otimes \bfe_{z}  \\
& \ \ + \ \partial_{r} \tau_{\theta \theta}    \, \bfe_{\theta} \otimes \bfe_{\theta} \otimes \bfe_{r} 
\ + \    0  \, \bfe_{\theta} \otimes \bfe_{\theta} \otimes \bfe_{\theta}  
\ + \   \partial_{z} \tau_{\theta \theta} \,   \bfe_{\theta} \otimes \bfe_{\theta} \otimes \bfe_{z}    \\
& \ \ + \  0  \, \bfe_{\theta} \otimes \bfe_{z} \otimes \bfe_{r} 
 \ + \     \frac{1}{r} \,  \tau_{r z}    \, \bfe_{\theta} \otimes \bfe_{z} \otimes \bfe_{\theta}  
\ + \   0 \, \bfe_{\theta} \otimes \bfe_{z} \otimes \bfe_{z}   \\
& \ \ + \  \partial_{r} \tau_{z r}   \, \bfe_{z} \otimes \bfe_{r} \otimes \bfe_{r} 
\ + \   0   \, \bfe_{z} \otimes \bfe_{r} \otimes \bfe_{\theta}   
\ + \     \partial_{z} \tau_{z r}  \,  \bfe_{z} \otimes \bfe_{r} \otimes \bfe_{z}  \\
& \ \ + \ 0   \, \bfe_{z} \otimes \bfe_{\theta} \otimes \bfe_{r} 
\ + \   \frac{1}{r} \, \tau_{z r}   \, \bfe_{z} \otimes \bfe_{\theta} \otimes \bfe_{\theta}  
\ + \    0  \,   \bfe_{z} \otimes \bfe_{\theta} \otimes \bfe_{z}  \\
& \ \ + \  \partial_{r} \tau_{z z}   \, \bfe_{z} \otimes \bfe_{z} \otimes \bfe_{r} 
 \ + \     0   \, \bfe_{z} \otimes \bfe_{z} \otimes \bfe_{\theta}  
\ + \   \partial_{z} \tau_{z z}  \, \bfe_{z} \otimes \bfe_{z} \otimes \bfe_{z}  \, .
\end{align*}

Let, 
\[
D \tautens_{r} \ = \ \left( \begin{array}{ccc}
\partial_{r} \tau_{r r}   &     0        &   \partial_{z} \tau_{r r}   \\
0        &       \frac{1}{r} ( \tau_{r r} \, - \, \tau_{\theta \theta} )  &   0   \\
\partial_{r} \tau_{r z}   &     0        &   \partial_{z} \tau_{r z}  \end{array} \right) \, , \ \ \ 
D \tautens_{\theta} \ = \ \left( \begin{array}{ccc}
0        &       \frac{1}{r} ( \tau_{r r} \, - \, \tau_{\theta \theta} )  &   0   \\
\partial_{r} \tau_{\theta \theta}   &     0        &   \partial_{z} \tau_{\theta \theta}   \\
 0        &       \frac{1}{r} \,  \tau_{r z}   &   0  \end{array} \right) \, , 
\]
\[
\mbox{ and } \ \ \ 
D \tautens_{z} \ = \ \left( \begin{array}{ccc}
\partial_{r} \tau_{z r}   &     0        &   \partial_{z} \tau_{z r}   \\
0        &       \frac{1}{r} \, \tau_{z r}   &   0   \\
\partial_{r} \tau_{z z}   &     0        &   \partial_{z} \tau_{z z}  \end{array} \right) \, .
\]
Then, 
\begin{align*}
\| \tautens \|^{2}_{H^{1}(\breve{\Omega})} &= \ 
\int_{\breve{\Omega}} \| \tautens(x, y, z) \|^{2} \, d \breve{\Omega} \ + \ 
\int_{\breve{\Omega}} \| D \tautens(x, y, z) \|^{2} \, d \breve{\Omega} \\
&= \ 2 \pi \int_{\Omega} \left( \tau_{r r}^{2} \, + \, \tau_{r z}^{2} \, + \, \tau_{\theta \theta}^{2} \, + \, \tau_{z r}^{2} \, + \, 
\tau_{z z}^{2} \right) \, r \, dr \, dz \\
& \quad + \ 
2 \pi \int_{\Omega} \left( D \tautens_{r} : D \tautens_{r} \, + \, D \tautens_{\theta} : D \tautens_{\theta} \, + \, 
 D \tautens_{z} : D \tautens_{z} \right)  \, r \, dr \, dz \, . 
\end{align*}
Hence $\| \tautens \|^{2}_{H^{1}(\breve{\Omega})} < \infty$
implies, $\tau_{r r}, \, \tau_{\theta  \theta}, \, \tau_{z  z} \in  {}_{1}H^{1}(\Omega), \
 \tau_{r z}, \,  \tau_{z r} \in  {}_{1}V^{1}(\Omega), \ ( \tau_{r r} \, - \, \tau_{\theta  \theta}) \in   {}_{-1}L^{2}(\Omega)$.

Reversing the argument establishes the isomorphism between the spaces.
\end{proof}

In the discussions that follow, we take $U={}_{1}\mbf{L}^{2}(\Omega)$, $Q={}_{1}L^{2}(\Omega)$.  As the merdian stress 
tensor is made up of a tensor and scalar component, we introduce the space $\boldsymbol{\Sigma}(\Omega)$ defined by
\begin{align*}
\boldsymbol{\Sigma}(\Omega) = \{ (\sigtens, \sigma) \in {}_{1}\Ltens(\Omega, \mathbb{M}^{2}) \times {}_{1}L^{2}(\Omega) \; : 
\; \axidiv (\sigtens, \sigma) \in {}_{1}L^{2}(\Omega) \}.
\end{align*}

Associated with $\boldsymbol{\Sigma}(\Omega)$ we have the norm
\begin{align*}
\| (\sigtens, \sigma) \|_{\boldsymbol{\Sigma}(\Omega)} = \left( \|\axidiv (\sigtens, \sigma) \|^{2}_{{}_{1}\mbf{L}^{2}(\Omega)} + 
\| \sigtens \|^{2}_{{}_{1} \Ltens(\Omega ; \; \mathbb{M}^{2})} + \| \sigma \|^{2}_{{}_{1}L^{2}(\Omega)} \right)^{\frac{1}{2}}.
\end{align*}

Additionally, we define $\boldsymbol{S}(\Omega) \subset \boldsymbol{\Sigma}(\Omega)$ by
\begin{align*}
\boldsymbol{S}(\Omega) = \{ (\sigtens, 0) \in \boldsymbol{\Sigma}(\Omega) : 
\sigtens = \left( \begin{array}{c} 
       \mbf{w}^{t}  \\ \mbf{z}^{t}  \end{array} \right) \, ; \ \mbf{w}, \mbf{z} \in {}_{1}\mbf{VH}^{1}(\Omega)  \},
\end{align*}
with norm
\begin{align}
\| (\sigtens, 0) \|_{\boldsymbol{S}(\Omega)} = \left( \| \axidiv (\sigtens, 0) \|^{2}_{{}_{1}\mbf{L}^{2}(\Omega)} 
+ \| \mbf{w} \|^{2}_{{}_{1}\mbf{VH}^{1}(\Omega)} + \| \mbf{z} \|^{2}_{{}_{1}\mbf{VH}^{1}(\Omega)} \right)^{\frac{1}{2}} .
\end{align}

For convenience, when the context is clear, $\bSigma (\Omega)$ and $\bS(\Omega)$ will be denoted as $\bSigma$ and $\bS$.

 \setcounter{equation}{0}
\setcounter{figure}{0}
\setcounter{table}{0}
\setcounter{theorem}{0}
\setcounter{lemma}{0}
\setcounter{corollary}{0}
\setcounter{definition}{0}
\section{Axisymmetric Variational Formulation}
\label{sec:axisymmetric_form}
In this section we present the variational form of the axisymmetric meridian problem.  
This problem has many similarities with the elasticity problem in the Cartesian setting, however, new terms are 
introduced into the bilinear forms as a consequence of the change of variable from Cartesian to cylindrical coordinates.  
Details of the derivation can be found in Appendix \ref{sec:axisymmetric_elasticity}.

Analogous to \eqref{eq:weak_symmetry_elasticity}, define the bilinear forms
\begin{align}
\label{eq:axi_a_op_1}
& \tilde{a}_{}(.,.) : \boldsymbol{\Sigma} \times \boldsymbol{\Sigma} \rightarrow \mathbb{R} , \quad 
\tilde{a}_{}((\sigtens_{},\sigma),(\tautens_{},\tau)) = (\mathcal{A} \sigtens_{}, \tautens_{})_{} + (\mathcal{A} \sigma, \tau) - \frac{1}{2 \mu} \frac{\lambda}{2\mu + 3 \lambda} ( (\sigma, \Tr{\tautens})   + (\Tr{\sigtens_{}}, \tau)_{}),  \\
\label{eq:axi_b_op}
& \tilde{b}_{}(.,.) : \boldsymbol{\Sigma} \times U \rightarrow \mathbb{R} ,  \quad
\tilde{b}_{}((\tautens_{}, \tau), \mbf{u}) = (\nabla_{\text{axi}} \cdot \tautens_{}, \mbf{u}_{}) - (\dfrac{\tau}{r}, u_{r}) \, ,  \\
\label{eq:axi_c_op}
& \tilde{c}_{}(.,.): \boldsymbol{\Sigma} \times Q \rightarrow \mathbb{R} , \quad
\tilde{c}_{}((\sigtens_{},\sigma), p ) = (\sigtens_{}, \mathcal{S}^{2}(p)) \, , 
\end{align}
where the operator $\mathcal{A}$ applied to the scalar function $\sigma_{\theta \theta}$ is given by
(\ref{eq:complianceTensor}) for $m = 1$.

The axisymmetric meridian problem with weak symmetry is then: 
\textit{Given $\mbf{f} \in {}_{1}\mbf{L}^{2}(\Omega)$, find $((\sigtens_{}, \sigma), \mbf{u}_{}, p) \in \boldsymbol{\Sigma} \times U \times Q$ such that for all $((\tautens_{}, \tau), \mbf{v}_{}, q) \in \boldsymbol{\Sigma} \times U \times Q$}
\begin{align}
\label{eq:axi_saddle_point_1_weak_symmetry}
\tilde{a}_{}((\sigtens_{},\sigma),(\tautens_{},\tau)) + \tilde{b}_{}((\tautens_{},\tau), \mbf{u}_{} ) + \tilde{c}_{}((\tautens_{},\tau), p ) &= 0 \\
\label{eq:axi_saddle_point_2_weak_symmetry}
\tilde{b}_{}((\sigtens_{},\sigma), \mbf{v}_{} ) &= (\mbf{f},\mbf{v}_{})_{} \\
\label{eq:axi_saddle_point_3_weak_symmetry}
\tilde{c}_{}((\sigtens_{},\sigma), q ) & = 0 \, .
\end{align}

Of interest is to develop discrete inf-sup stable elements for the approximation of 
\eqref{eq:axi_saddle_point_1_weak_symmetry}-\eqref{eq:axi_saddle_point_3_weak_symmetry}.
In cylindrical coordinates, the divergence operator does not map polynomial spaces into polynomial spaces, so some of the standard techniques for verifying inf-sup stability cannot be used.  Thus, to help establish a variational formulation for which stable triples of finite elements may be verified to satisfy the discrete inf-sup condition, we make two modifications to \eqref{eq:axi_saddle_point_1_weak_symmetry}-\eqref{eq:axi_saddle_point_3_weak_symmetry}.

First, we add a grad-div stabilization term to $\tilde{a}_{}(\cdot,\cdot)$ and define a new bilinear form $a(\cdot,\cdot): \boldsymbol{\Sigma} \times \boldsymbol{\Sigma} : \rightarrow \mathbb{R}$
\begin{align}
\begin{split}
\label{eq:axi_a_grad_div}
a((\sigtens_{},\sigma),(\tautens_{},\tau)) &= \tilde{a}_{}((\sigtens_{},\sigma),(\tautens_{},\tau)) + \gamma (\nabla_{\text{axi}} \cdot (\sigtens_{},\sigma) , \nabla_{\text{axi}} \cdot (\tautens_{},\tau) )
\end{split}
\end{align}
where $\gamma$ is the grad-div stabilization term.  This stabilization term ensures that $a((\cdot,\cdot),(\cdot,\cdot))$ is coercive in the $\| \cdot \|_{\boldsymbol{\Sigma}}$ norm.  Unless specified otherwise, we take $\gamma = 1$.

Recall from \eqref{eq:elasticity_strong} that in cylindrical coordinates, $\axidiv (\sigtens, \sigma) = \mbf{f}$.  Therefore, to account for the grad-div stabilization term in the constituent equation, $(\mbf{f}, \axidiv (\tautens, \tau))$ must also be added to the right hand side of \eqref{eq:axi_saddle_point_1_weak_symmetry}.

For the second modification, recall that $\tilde{c}((\sigtens_{}, \sigma), q) = (\sigtens_{}, \mathcal{S}^{2}(q))$, and let $\mbf{x} = (r,z)^{t}$.  As described in Lemma \ref{lem:c_operator_result} below,
\begin{align}
\begin{split}
\label{eq:c_op_mod1}
\int_{\Omega} \sigtens_{} : \mathcal{S}^{2}(q) \;r \; d\; \Omega &= - \int_{\Omega} ( \axidiv (\sigtens, \sigma) \wedge \mbf{x}) \; q \; r \; d \Omega \\ 
& \quad  + \int_{\partial \Omega} (\sigtens \cdot \mbf{n}) \cdot \mbf{x}^{\perp} q \; r \; ds - \int_{\Omega} \sigtens : (\mbf{x}^{\perp} \otimes \nabla q) \; r \; d \Omega  
 - \int_{\Omega} \frac{1}{r} \sigma \; z \; q \; r\; d \Omega,
\end{split}
\end{align}
or equivalently
\begin{align}
\label{eq:c_op_mod2}
\begin{split}
\int_{\Omega} \sigtens_{} : \mathcal{S}^{2}(q) \;r \; d \Omega &+ \int_{\Omega} ( \axidiv (\sigtens, \sigma) \wedge \mbf{x}) \; q \; r \; d \Omega \\  
& = \int_{\partial \Omega} (\sigtens \cdot \mbf{n}) \cdot \mbf{x}^{\perp} q \; r \; ds - \int_{\Omega} \sigtens : (\mbf{x}^{\perp} \otimes \nabla q) \; r \; d \Omega 
 - \int_{\Omega} \sigma \; z \; q \; d \Omega.
\end{split}
\end{align}

In terms of establishing stable approximation elements via the construction of a suitable projection (see Theorem \ref{lem:thm_41}) it is
 more convenient to use equation \eqref{eq:c_op_mod2} than \eqref{eq:c_op_mod1}.  To introduce \eqref{eq:c_op_mod2} into the 
 weak form, we add $\int_{\Omega} (\axidiv (\sigtens_{}, \sigma) \wedge \mbf{x}) \; q \; r \; d \Omega $ to both sides of
  \eqref{eq:axi_saddle_point_3_weak_symmetry} giving
\begin{align}
\label{eq:new_c}
\tilde{c}((\sigtens_{}, \sigma), q) + \int_{\Omega} ( \axidiv (\sigtens_{}, \sigma) \wedge \mbf{x}) \; q \; r \; d \Omega = \int_{\Omega} (\mbf{f} \wedge \mbf{x}) \; q \; r \; d \Omega,
\end{align}
where we have used the relationship $\axidiv (\sigtens, \sigma) = \mbf{f}$ on the right hand side.  To represent the left hand side of \eqref{eq:new_c}, we define a new bilinear form $c(\cdot,\cdot): \boldsymbol{\Sigma} \times Q \rightarrow \mathbb{R}$ as
\begin{align}
\label{eq:modified_c}
\begin{split}
c((\sigtens_{}, \sigma), q) &:= \tilde{c}((\sigtens_{}, \sigma), q) + (\nabla_{\text{axi}} \cdot (\sigtens_{}, \sigma) \wedge \mbf{x}, q) \\
& \; =(\sigtens_{}, \mathcal{S}^{2}(q)) + (\nabla_{\text{axi}} \cdot (\sigtens_{}, \sigma) \wedge \mbf{x}, q).
\end{split}
\end{align}

Therefore, \eqref{eq:axi_saddle_point_3_weak_symmetry} becomes
\begin{align*}
c((\sigtens, \sigma), q) = (\mbf{f} \wedge \mbf{x}, q).
\end{align*} 

For notational consistency in the new formulation, we let $b(\cdot,\cdot) = \tilde{b}(\cdot,\cdot)$.

To maintain the saddle point structure of the variational formulation with the bilinear form $c(\cdot,\cdot)$, we need to add and subtract $(\axidiv (\tautens_{}, \tau) \wedge \mbf{x}, p)$ to the left hand side of \eqref{eq:axi_saddle_point_1_weak_symmetry}.  To understand the affect of this modification on the  formulation, first observe that 
\begin{align*}
\begin{split}
\axidiv (\tautens, \tau) \wedge \mbf{x} &= \begin{pmatrix} \dfrac{\partial \tau_{11}}{\partial r} + \dfrac{\partial \tau_{12}}{\partial z} + \dfrac{1}{r} (\tau_{11} - \tau) \\[3mm] \dfrac{\partial \tau_{21}}{\partial r} + \dfrac{\partial \tau_{22}}{\partial z} + \dfrac{1}{r} \tau_{21}\end{pmatrix} \wedge \begin{pmatrix} r \\ z \end{pmatrix} \\
&= z (\dfrac{\partial \tau_{11}}{\partial r} + \dfrac{\partial \tau_{12}}{\partial z} + \dfrac{1}{r} (\tau_{11}-\tau) ) - r (\dfrac{\partial \tau_{21}}{\partial r} + \dfrac{\partial \tau_{22}}{\partial z} + \dfrac{1}{r} \tau_{21}) 
\ =  \ (\axidiv(\tautens, \tau)) \cdot \mbf{x}^{\perp}.
\end{split}
\end{align*}

Therefore,
\begin{align}
\label{eq:expanded_c_star}
\begin{split}
( (\axidiv (\tautens,\tau)) \wedge \mbf{x}, p ) &= \int_{\Omega} \axidiv (\tautens, \tau) \wedge \mbf{x} \; p \; r \; d\Omega 
\ = \ \int_{\Omega} (\axidiv (\tautens,\tau)) \cdot \mbf{x}^{\perp} \; p \; r \; d \Omega \\
& = b((\tautens,\tau), \mbf{x}^{\perp} \; p).
\end{split}
\end{align}

This shows that $((\axidiv(\tautens, \tau) \wedge \mbf{x}, p)$ can be expressed as $b((\tautens,\tau), \mbf{x}^{\perp} p)$.  As a result, the negative part of $((\axidiv(\tautens, \tau)\wedge \mbf{x}, p)$ that is used to balance the constituent equation enters into the expression as part of the bilinear form $b(\cdot, \cdot)$.  That is,
\begin{align}
\label{eq:mod_b_term_with_w}
b((\tautens, \tau), \mbf{u} ) - b((\tautens,\tau),\mbf{x}^{\perp} p) = b((\tautens,\tau),\mbf{u} - \mbf{x}^{\perp} p).
\end{align}

To reflect the fact that the expression within the bilinear form $b(\cdot,\cdot)$ no longer depends only on the displacement $\mbf{u}$, we define a new variable $\mbf{w} = \mbf{u} - \mbf{x}^{\perp} p$.  As we discuss further in Sections \ref{sec:convergence_analysis} and \ref{sec:computational_results}, once the solution has be found, the displacement $\mbf{u}=\mbf{w}+\mbf{x}^{\perp}p$ can be accurately recovered during a post-processing step. 

Therefore, an equivalent but modified version of the axisymmetric linear elasticity problem \eqref{eq:axi_saddle_point_1_weak_symmetry}-\eqref{eq:axi_saddle_point_3_weak_symmetry} can be expressed as: 
\textit{Given $\mbf{f} \in {}_{1}\mbf{L}^{2}(\Omega)$ find $((\sigtens_{}, \sigma), \mbf{w}, p) \in \boldsymbol{\Sigma} \times U \times Q$ such that for all $((\tautens_{}, \tau), \mbf{v}, q) \in \boldsymbol{\Sigma} \times U \times Q$}
\begin{align}
\label{eq:axi_saddle_point_1_mod}
a((\sigtens_{},\sigma),(\tautens_{},\tau)) + b((\tautens_{},\tau), \mbf{w} ) + c((\tautens_{},\tau), p ) &= (\mbf{f}, \nabla_{\text{axi}} \cdot (\tautens_{}, \tau) ) \\
\label{eq:axi_saddle_point_2_mod}
b((\sigtens_{},\sigma), \mbf{v} )   &= (\mbf{f},\mbf{v})_{}  \\
\label{eq:axi_saddle_point_3_mod}
c((\sigtens_{},\sigma), q ) & = (\mbf{f} \wedge \mbf{x}, q)_{} \, .
\end{align}

\subsection{Well posedness of the variational formulation \eqref{eq:axi_saddle_point_1_mod}-\eqref{eq:axi_saddle_point_3_mod}}
\label{ssec:wellposed}
To establish the well posedness of the saddle point formulation, 
\eqref{eq:axi_saddle_point_1_mod}-\eqref{eq:axi_saddle_point_3_mod}, we show that $a(\cdot , \cdot)$ is bounded and coercive
on $\boldsymbol{\Sigma}  \times \boldsymbol{\Sigma}$, and that $( (\tautens, \tau), \mbf{v}, p)$ satisfy the
inf-sup condition
\begin{align}
\label{eq:main_inf_sup}
\inf_{\mbf{v} \in U, p \in Q} \sup_{ (\tautens, \tau) \in \bSigma} \dfrac{b( (\tautens, \tau),\mbf{v}) 
+ c( (\tautens, \tau),p)}{\|  (\tautens, \tau) \|_{\bSigma} ( \| \mbf{v} \|_{U} + \| p \|_{Q} )} \geq C.
\end{align}

\begin{lemma}
\label{lem:a_bounded}
The operator $a(.,.)$ defined in (\ref{eq:axi_a_grad_div}) is bounded.  That is,
\begin{align}
\label{eq:a_bounded}
a((\sigtens, \sigma), (\tautens, \tau)) \leq C \| (\sigtens, \sigma) \|_{\boldsymbol{\Sigma}} \| (\tautens, \tau) \|_{\boldsymbol{\Sigma}}
\end{align}
for some $C > 0$ and all $(\sigtens, \sigma), (\tautens, \tau) \in \boldsymbol{\Sigma}$.
\end{lemma}
\begin{proof}
Using the Cauchy-Schwarz inequality, 
\begin{align}
\label{eq:a_bounded_1}
\begin{split}
a((\sigtens, \sigma), (\tautens, \tau)) & = \dfrac{1}{2 \mu} ( \sigtens, \tautens) + \dfrac{1}{2 \mu} (\sigma, \tau) - \dfrac{1}{2 \mu} \dfrac{\lambda}{2 \mu + 3 \lambda} (\text{tr} (\sigtens) + \sigma, \text{tr} (\tautens) + \tau) \\
& \quad \quad + ( \nabla_{\text{axi}} \cdot (\sigtens, \sigma), \nabla_{\text{axi}} \cdot (\tautens, \tau) ) \\
&\leq \dfrac{1}{2 \mu} \left( \| \sigtens \|_{{}_{1}\Ltens(\Omega)} \| \tautens \|_{{}_{1}\Ltens(\Omega)} + \| \sigma \|_{{}_{1}L^{2}(\Omega)} \| \tau \|_{{}_{1}L^{2}(\Omega)} \right) \\
&\quad \quad + \| \axidiv (\sigtens, \sigma) \|_{{}_{1}\mbf{L}^{2}(\Omega)} \| \axidiv (\tautens, \tau) \|_{{}_{1}\mbf{L}^{2}(\Omega)} \\
& \quad \quad + \dfrac{1}{2 \mu}\dfrac{\lambda}{2 \mu + 3 \lambda} \| \text{tr} (\sigtens) + \sigma\|_{{}_{1}L^{2}(\Omega)} \| \text{tr} (\tautens) + \tau\|_{{}_{1}L^{2}(\Omega)}  \\
& \leq C \left( \|(\sigtens, \sigma)\|_{\boldsymbol{\Sigma}} \|(\tautens, \tau)\|_{\boldsymbol{\Sigma}} +  \| \text{tr} (\sigtens) + \sigma\|_{{}_{1}L^{2}(\Omega)} \| \text{tr} (\tautens) + \tau\|_{{}_{1}L^{2}(\Omega)} \right). 
\end{split}
\end{align}
Further, for any $(\sigtens, \sigma) \in \boldsymbol{\Sigma}$,
\begin{align}
\label{eq:a_bounded_2}
\begin{split}
\| \text{tr}(\sigtens) + \sigma \|_{{}_{1} L^{2}(\Omega)} &= \int_{\Omega} (\text{tr} (\sigtens) + \sigma)^{2} r \; d \Omega \leq C \left( (\sigtens, \sigtens) + (\sigma, \sigma) \right) \\
&\leq C \| (\sigtens, \sigma) \|_{{}_{1}\Ltens(\Omega;\mathbb{M}^{2} \times \mathbb{R}^{1} )} \leq C \| (\sigtens, \sigma ) \|_{\bSigma}.
\end{split}
\end{align}
Combining \eqref{eq:a_bounded_1} and \eqref{eq:a_bounded_2} yields \eqref{eq:a_bounded}.
\end{proof}

\begin{lemma}
\label{lem:a_coercive}
The operator $a(.,.)$ defined in (\ref{eq:axi_a_grad_div}) is coercive.  That is,
\begin{align}
a((\sigtens, \sigma),(\sigtens, \sigma)) \geq c \|(\sigtens, \sigma)\|^{2}_{\boldsymbol{\Sigma}} \text{ where } c = \min \{ \dfrac{1}{2 \mu} \dfrac{1}{2 \mu + 3 \lambda}, 1 \}.
\end{align}
\end{lemma}
\begin{proof}
We begin with the observation that
\begin{align}
\label{eq:a_coercive_1}
\begin{split}
a((\sigtens, \sigma), (\sigtens, \sigma)) &= (\mathcal{A} \sigtens, \sigtens) + (\mathcal{A}\sigma, \sigma ) - \dfrac{1}{2 \mu} \dfrac{\lambda}{2  + 3 \lambda} [(\sigma, \text{tr}(\sigtens)) + (\text{tr}(\sigtens), \sigma)] \\
& \quad \quad + ( \nabla_{\text{axi}} \cdot (\sigtens, \sigma), \nabla_{\text{axi}} \cdot (\sigtens, \sigma) ) \\
& = \dfrac{1}{2 \mu} ( \sigtens, \sigtens) + \dfrac{1}{2 \mu} (\sigma, \sigma) - \dfrac{1}{2 \mu} \dfrac{\lambda}{2 \mu + 3 \lambda} (\text{tr} (\sigtens) + \sigma, \text{tr} (\sigtens) + \sigma) \\
& \quad \quad + ( \nabla_{\text{axi}} \cdot (\sigtens, \sigma), \nabla_{\text{axi}} \cdot (\sigtens, \sigma) ).
\end{split}
\end{align}

Next we must incorporate the $(\text{tr} (\sigtens) + \sigma, \text{tr} (\sigtens) + \sigma)$ term into \eqref{eq:a_coercive_1} in a way that will allow us to obtain the $\boldsymbol{\Sigma}$ norm.  To do so, we start by adding the inequalities
\begin{align*}
\begin{split}
\sigma_{11}^{2} + \sigma_{22}^{2} \geq 2 \sigma_{11} \sigma_{22}, \quad \sigma_{22}^{2} + \sigma^{2} \geq 2\sigma_{22} \sigma, \quad \sigma_{11}^{2} + \sigma^{2} \geq 2 \sigma_{11} \sigma
\end{split}
\end{align*}
to get that $2 (\sigma_{11}^{2} + \sigma_{22}^{2} + \sigma^{2}) \geq 2 (\sigma_{11} \sigma_{22} + \sigma_{22} \sigma + \sigma_{11} \sigma )$.  Adding additional positive terms to the left-hand side of this inequality gives
\begin{align*}
2(\sigma_{11}^{2} + \sigma_{22}^{2} + \sigma^{2} ) + 3 (\sigma_{12}^{2} + \sigma_{21}^{2} ) \geq 2 (\sigma_{11} \sigma_{22} + \sigma_{22} \sigma + \sigma_{11}\sigma).
\end{align*}

Since $(\text{tr}(\sigtens) + \sigma)^{2} = \sigma^{2}_{11} + \sigma_{22}^{2} + \sigma^{2} + 2 (\sigma_{11} \sigma_{22} + \sigma_{22} \sigma + \sigma_{11}\sigma )$ and $\mu, \lambda > 0$,
\begin{align}
\begin{split}
\label{eq:a_coercive_2}
\dfrac{1}{2 \mu}\dfrac{\lambda}{2 \mu + 3 \lambda} (\text{tr}(\sigtens) + \sigma,\text{tr}(\sigtens) + \sigma) &\leq \dfrac{1}{2 \mu}\dfrac{3 \lambda}{2 \mu + 3 \lambda} \int_{\Omega} (\sigma_{11}^{2} + \sigma_{12}^{2} + \sigma_{21}^{2} + \sigma_{22}^{2} + \sigma^{2}) \; r \; d \Omega  \\
& = \dfrac{1}{2 \mu}\dfrac{3 \lambda}{2 \mu + 3 \lambda} (\sigtens, \sigtens) + \dfrac{1}{2 \mu}\dfrac{3 \lambda}{2 \mu + 3 \lambda}  (\sigma, \sigma).
\end{split}
\end{align}

Combining \eqref{eq:a_coercive_1} and \eqref{eq:a_coercive_2}
\begin{align*}
\label{eq:a_coercive_3}
\begin{split}
a((\sigtens, \sigma), (\sigtens, \sigma)) 
& \geq \dfrac{1}{2 \mu}\dfrac{2 \mu}{2 \mu + 3 \lambda} \left( ( \sigtens, \sigtens) + (\sigma, \sigma) \right) + ( \nabla_{\text{axi}} \cdot (\sigtens, \sigma), \nabla_{\text{axi}} \cdot (\sigtens, \sigma) )\\
&= \dfrac{1}{2 \mu + 3 \lambda} \| (\sigtens, \sigma )\|^{2}_{_{1} \mbf{L}^{2} (\Omega; \mathbb{M}^{2} \times \mathbb{R}^{1})} + \| \nabla_{\text{axi}} \cdot ( \sigtens, \sigma) \|^{2}_{_{1} \Ltens(\Omega)} \\
& \geq c \| (\sigtens,\sigma) \|^{2}_{\boldsymbol{\Sigma}} .
\end{split}
\end{align*}
\end{proof}

\subsubsection{Satisfying the continuous inf-sup condition \eqref{eq:main_inf_sup}}
\label{sssec:continous_inf_sup}

To establish the inf-sup condition \eqref{eq:main_inf_sup}, we follow a similar two step argument as 
used in \cite{BBF1} for the planar elasticity problem. In Step 1 a $(\tautens_{1}, \tau_{1})$ is found such that,
for $\mbf{v}, p$ given, $b( (\tautens_{1}, \tau_{1}),\mbf{v})  =  \| \mbf{v} \|_{U}^{2}$. Then, in Step 2 $(\tautens_{2}, \tau_{2})$
is constructed to handle the $c( \cdot , \cdot)$ term, while satisfying $b( (\tautens_{2}, \tau_{2}),\mbf{v})  =   0$.
The following lemma is useful in the construction of $(\tautens_{2}, \tau_{2})$.
\begin{lemma}
\label{lem:consttau2}
Given $\beta \in {}_{1}L^{2}(\Omega)$ there exist $(\tautens, \tau) \in \bSigma$ such that
\begin{equation}
\nabla_{axi} \cdot (\tautens, \tau) = \mbf{0} , \ \ \ \  \text{ as}( \tautens ) = 
\begin{pmatrix} 0 & \beta \\ - \beta & 0\end{pmatrix} ,
\ \ \ \ \mbox{ and } \ \  \| (\tautens, \tau) \|_{\bSigma}  \ \le \ C \, \| \beta \|_{{}_{1}L^{2}(\Omega)} \, .
\label{condtau2}
\end{equation}
\end{lemma}
\begin{proof}
From \eqref{condtau2},
\begin{align}
\frac{1}{2} \big( \tau_{1 2} - \tau_{2 1} \big) = \beta \, , \ \ & \Rightarrow \ \ \tau_{2 1} = \tau_{1 2} - 2 \beta \, ,   \label{ytr1} \\
\partial_{r} \tau_{1 1} + \frac{1}{r} \tau_{1 1}  & + \partial_{z} \tau_{1 2} - \frac{1}{r} \tau \ = \ 0 \, ,  \label{ytr2} \\
\partial_{r} \tau_{2 1} + \frac{1}{r} \tau_{2 1}  & + \partial_{z} \tau_{2 2}  \ = \ 0 \, .  \label{ytr3} 
\end{align}
Substituting \eqref{ytr1} into \eqref{ytr3}, then multiplying \eqref{ytr2} and \eqref{ytr3} through by $r$ and simplifying we obtain
\begin{align}
 \partial_{r} (r \, \tau_{1 1}) +  \partial_{z} (r \, \tau_{1 2} ) -   \tau &= 0  \, ,  \label{ytr4} \\
  \partial_{r} (r \, \tau_{1 2}) +  \partial_{z} (r \, \tau_{2 2})    &=  \partial_{r} (2 r \, \beta)  \, .  \label{ytr5} 
\end{align}
Integrating \eqref{ytr4} with respect to $z$, and \eqref{ytr5} with respect to $r$, yields for arbitrary $f_{1}$ and $f_{2}$,
\begin{align}
\int^{z} \partial_{r} (r \, \tau_{1 1})  \, dz \ +  \  r \, \tau_{1 2} \  -  \int^{z}  \tau \, dz &= f_{1}(r)  \, ,  \label{ytr6} \\
  r \, \tau_{1 2} \ +  \  \int^{r} \partial_{z} (r \, \tau_{2 2}) \, dr   &=   2 r \, \beta \ + \ f_{2}(z)  \, .  \label{ytr7} 
\end{align}
Interchanging the order of integration and differentiation, and then subtracting \eqref{ytr6} from \eqref{ytr7}, yields
\[
\partial_{r}  \left( - \int^{z} (r \, \tau_{1 1})  \, dz  \right) \ + \
 \partial_{z} \left( \int^{r}  (r \, \tau_{2 2}) \, dr \right) \ + \   \int^{z}  \tau \, dz
 \ = \ 2 r \, \beta \ + \ f_{2}(z) \ - \ f_{1}(r) \, . 
 \]
 Dividing through by $r$, choosing $f_{1} = f_{2} = 0$, and rearranging we have
 \be
 \frac{1}{r} \partial_{r}  \left( r (- \int^{z} \tau_{1 1}  \, dz  ) \right) \ + \
  \partial_{z} \left(  \frac{1}{r} \int^{r}  (r \, \tau_{2 2}) \, dr \right) \ - \   \frac{1}{r} \left( - \int^{z}  \tau \, dz \right)
 \ = \ 2 \beta \, . 
\label{ytr11}
\ee
\be
\mbox{Let } \ \ \  \sigma_{r r} = (- \int^{z} \tau_{1 1}  \, dz  ) \, , \ \ \ \sigma_{r z} =  \frac{1}{r} \int^{r}  (r \, \tau_{2 2}) \, dr \, , \ \ 
\mbox{ and  } \sigma_{\theta \theta} = - \int^{z}  \tau \, dz \, .
\label{ytr115}
\ee
Then, \eqref{ytr11} can be embedded in the meridian problem
\be
\nabla_{\text{axi}} \cdot \left[ \begin{array}{ccc}
\sigma_{r r}  &  0   & \sigma_{r z}  \\
0     &   \sigma_{\theta \theta}   &  0   \\
\sigma_{z r}    &   0   &   \sigma_{z z}
\end{array}  \right]
\ = \ \left[  \begin{array}{c}
2 \, \beta \\   0   
\end{array}  \right] \ \ \ \mbox{ in } \  \Omega \, .
\label{ytr12}
\ee
Lifting \eqref{ytr12} from $\Omega$ to $\breve{\Omega}$ we obtain an axisymmetric elasticity problem in $\breve{\Omega}$,
$ \nabla \cdot \sigtens \, = \, \bm{f}$ (see \eqref{eq:elasticity_strong}), with $\bm{f} \in \mbf{L}^{2}(\brOmega)$.
From \cite{gri921}, we have that, for $\breve{\Omega}$ a bounded polyhedral domain, 
$\sigtens \in \Htens^{1}(\brOmega, \mathbb{M}^{3})$.
Additionally, $ \| \sigtens \|_{\Htens^{1}(\brOmega, \mathbb{M}^{3})} \, \le \, C \, \| \bm{f} \|_{\mbf{L}^{2}(\breve{\Omega})} \, 
\le \, C \| \beta \|_{{}_{1}L^{2}(\Omega)}$.

Then, using Lemma \ref{redten} we have that 
\be
\sigma_{r r} , \, \sigma_{\theta \theta}  \in {}_{1}H^{1}(\Omega) , \
\sigma_{r z} \in {}_{1}V^{1}(\Omega) , \ \mbox{ and } \ (\sigma_{r r} - \sigma_{\theta \theta}) \in {}_{-1}L^{2}(\Omega) \,. 
\label{ytr15}
\ee

From \eqref{ytr115} and \eqref{ytr15}, 
\begin{align*}
\tau &= \ - \partial_{z}  \sigma_{\theta \theta}.  \  \
\mbox{As } \ \sigma_{\theta \theta}  \in {}_{1}H^{1}(\Omega), \ \mbox{ then } \ \tau \in  {}_{1}L^{2}(\Omega) \, . \\
\tau_{1 1} &= \ - \partial_{z} \sigma_{r r}.  \  \
\mbox{As } \ \sigma_{r r}  \in {}_{1}H^{1}(\Omega), \ \mbox{ then } \ \tau_{1 1} \in  {}_{1}L^{2}(\Omega) \, . \\
\tau_{2 2} &= \ \frac{1}{r}  \sigma_{r z}  \, + \, \partial_{r} \sigma_{r z}  \ .  \  \
\mbox{As } \ \sigma_{r z}  \in {}_{1}V^{1}(\Omega), \ \mbox{ then } \ \tau_{2 2} \in  {}_{1}L^{2}(\Omega) \, . 
\end{align*}

Also, it follows that
\be
  \| \tau \|_{ {}_{1}L^{2}(\Omega)} \ + \  \| \tau_{1 1} \|_{ {}_{1}L^{2}(\Omega)} 
   \ + \  \| \tau_{2 2} \|_{ {}_{1}L^{2}(\Omega)}  \ \le \ C \, \| \sigtens \|_{\Htens^{1}(\brOmega, \mathbb{M}^{3})}
   \ \le \ C \| \beta \|_{ {}_{1}L^{2}(\Omega)} \, .
\label{ytr16}
\ee

Next, from \eqref{ytr7} and \eqref{ytr15},
\begin{align}
\tau_{1 2} &= \ 2 \beta \ - \ \frac{1}{r} \int^{r} \partial_{z} (r \, \tau_{2 2}) \, dr 
\ = \ 2 \beta \ - \  \partial_{z} \left( \frac{1}{r} \int^{r} r \, \tau_{2 2} \, dr  \right)   \nonumber  \\
&= \ 2 \beta \ - \ \partial_{z} \sigma_{r z} \, .   \nonumber \\ 
\Rightarrow \ \tau_{1 2} &\in {}_{1}L^{2}(\Omega) \, , \ \ \mbox{ with } 
 \| \tau_{1 2} \|_{ {}_{1}L^{2}(\Omega)}  \ \le \ C \| \beta \|_{ {}_{1}L^{2}(\Omega)} \, .  \label{ytr17}
\end{align}

Also, from \eqref{ytr1} and \eqref{ytr17},
\be
\tau_{2 1}  \ = \ \tau_{1 2} \, - \, 2 \beta \ \in  {}_{1}L^{2}(\Omega) \, , \ \ \mbox{ with } 
 \| \tau_{2 1} \|_{ {}_{1}L^{2}(\Omega)}  \ \le \ C \| \beta \|_{ {}_{1}L^{2}(\Omega)} \, .  \label{ytr18}
\ee

Finally, we confirm that \eqref{ytr2} and \eqref{ytr3} are satisfied.
\begin{align*}
\partial_{r} \tau_{1 1} + \frac{1}{r} \tau_{1 1}   + \partial_{z} \tau_{1 2} - \frac{1}{r} \tau 
&= \ \frac{1}{r} \partial_{r} \big( r \, \tau_{1 1} \big) \, +  \,  \partial_{z} \tau_{1 2} \, -  \, \frac{1}{r} \tau  \\
&= \ \frac{1}{r} \partial_{r} \big( r \, \partial_{z} ( - \sigma_{r r} )  \big) 
\, +  \,  \partial_{z} \big( 2 \beta \, - \, \partial_{z} \sigma_{r z} \big)  \, -  \, \frac{1}{r} \big( - \partial_{z} \sigma_{\theta \theta} \big) \\
&= \ - \partial_{z} \big( \frac{1}{r} \partial_{r} (r \, \sigma_{r r}) \ + \ \partial_{z} \sigma_{r z} \ - \ \frac{1}{r} \sigma_{\theta \theta}
\ - \ 2 \beta \big) \\
&= \ - \partial_{z} \big( 0 ) \ = \ 0 \, \ \ \mbox{(from \eqref{ytr12}) } \, .
\end{align*}

Also,
\begin{align*}
\partial_{r} \tau_{2 1} + \frac{1}{r} \tau_{2 1}   + \partial_{z} \tau_{2 2} 
&= \ \frac{1}{r} \partial_{r} \big( r \, \tau_{2 1} \big) \, +  \,  \partial_{z} \tau_{1 2}   \\
&= \ \frac{1}{r} \partial_{r} \big( r \, ( - \partial_{z}   \sigma_{r z} )  \big) 
\, +  \,  \partial_{z} \big(  \frac{1}{r}  \partial_{r} ( r \sigma_{r z} ) \big)  \\
&= \  - \partial_{z}  \big( \frac{1}{r} \partial_{r}  ( r \,  \sigma_{r z} )  \big) 
 \ + \  \partial_{z} \big(  \frac{1}{r}  \partial_{r} ( r \sigma_{r z} ) \big) \ = \ 0 \, .
\end{align*}
This completes the proof.
\end{proof}

The following lemma established the inf-sup condition \eqref{eq:main_inf_sup}.
\begin{lemma}
\label{lem:b_operator_axi}
For any $\mbf{v} \in U$ and $p \in Q$, there exists a $C > 0$ and a $(\tautens, \tau) \in \bSigma$ such that
\begin{align}
\label{eq:B_operator}
b((\tautens,\tau), \mbf{v}) + c((\tautens,\tau) , p) &= \|\mbf{v} \|^{2}_{U} + \|p \|^{2}_{Q}  \, ,  \\
\mbox{with } \ \ 
\label{eq:tau_norm}
\| (\tautens,\tau) \|_{\bSigma} & \leq C ( \|\mbf{v} \|_{U} + \| p \|_{Q} ) \, .
\end{align}
\end{lemma}
\begin{proof}
Let $\mbf{v} = (v_{1}, v_{2})^{t} \in U$ and $p \in Q$ be given.  
Then, there exist vectors $\mbf{w}, \mbf{z} \in {}_{1}\mbf{VH}^{1}(\Omega)$ such that
\begin{align}
\axidiv \mbf{w}  = v_{1} \text{ and } \axidiv \mbf{z} = v_{2} \, 
\end{align}
where $\| \axidiv \mbf{w}\|_{{}_{1}L^{2}(\Omega)} + \| \mbf{w} \|_{{}_{1}\mbf{VH}^{1}(\Omega)}
\leq C \;  \| v_{1} \|_{{}_{1}L^{2}(\Omega)}$ and 
 $\| \axidiv \mbf{z} \|_{{}_{1}L^{2}(\Omega)} 
+ \| \mbf{z} \|_{{}_{1}\mbf{VH}^{1}(\Omega)} \leq C \; \| v_{2} \|_{{}_{1}L^{2}(\Omega)}$.  
To compute the vectors $\mbf{w}$ and $\mbf{z}$, one can map the axisymmetric scalar functions 
$v_{1}$ and $v_{2}$ into 3D Cartesian space and solve scalar Laplace equations to obtain functions $t_{1}$ and $t_{2}$.  
The gradient functions $\breve{\mbf{w}} = \nabla_{(x,y,z)} t_{1}$ and $\breve{\mbf{z}} = \nabla_{(x,y,z)} t_{2}$ are 
then computed.  
Finally, using Lemma \ref{redvec}, mapping
$\breve{\mbf{w}}$ and $\breve{\mbf{z}}$ from $\breve{\Omega}$ 
to $\Omega$, we obtain $\mbf{w}$ and $\mbf{z}$.

Using $\mbf{w}$ and $\mbf{z}$, we then construct a matrix $\tautens^{1}$, where
\begin{align}
\tautens^{1} = \begin{pmatrix} \mbf{w}^{t} \\ \mbf{z}^{t} \end{pmatrix}. 
\end{align}
Thus, taking $(\tautens^{1}, \tau^{1}) = (\tautens^{1}, 0)$  one has that
\begin{align}
\label{eq:tau_bound_with_u_1}
\axidiv (\tautens^{1}, \tau^{1})  = \mbf{v}, \; \; & \text{  hence  } \; \; b((\tautens^{1}, \tau^{1}), \mbf{v}) = \| \mbf{v} \|^{2}_{U}, \\
\mbox{and  }
\label{eq:tau_one_norm}
\| (\tautens^{1}, \tau^{1}) \|_{\bSigma} \leq \| (\tautens^{1}, 0) \|_{\boldsymbol{S}} 
& \leq C \| \mbf{v} \|_{U} \leq C ( \| \mbf{v} \|_{U} + \| p \|_{Q} ).
\end{align}

To build $(\tautens^{2}, \tau^{2}) \in \bSigma$, we first choose $\theta, \gamma \in {}_{1}L^{2}(\Omega)$ such that
\begin{align}
\label{eq:theta_gamma_defs}
\mathcal{S}^{2} (\theta) = as(\tautens^{1}), \quad \text{and} \; \gamma = \dfrac{1}{2} (v_{1} z - v_{2} r) 
= \dfrac{1}{2} (\axidiv \tautens^{1} \wedge \mbf{x}).
\end{align}

Next, set $\beta = (\gamma + \theta - \frac{1}{2} p ) \in {}_{1}L^{2}(\Omega)$.  Note that
\begin{align*}
\begin{split}
\| \theta \|_{{}_{1}L^{2}(\Omega)}^{2} &= \int_{\Omega} (\tau_{12} - \tau_{21})^{2} \; r \; d \Omega \leq 2 \int_{\Omega} (\tau_{11}^{2} + \tau_{12}^{2} + \tau_{21}^{2} + \tau_{22}^{2}) \; r \; d \Omega \\
& \leq 2 \| (\tautens^{1}, \tau^{1}) \|_{\bSigma} \leq C ( \| \mbf{v} \|_{U}^{2} + \| p \|_{Q}^{2} )
\end{split}
\end{align*}

Also, 
\begin{align*}
\| \gamma \|^{2}_{{}_{1}L^{2}(\Omega)} &= \int_{\Omega} \dfrac{1}{4} ( v_{1} z - v_{2} r)^{2} \; r \; d \Omega \leq \frac{1}{2} \int_{\Omega} (v_{1}^{2} z^{2} + v_{2}^{2} r^{2} ) \; r \; d \Omega \\
& \leq C \int_{\Omega} \mbf{v} \cdot \mbf{v} \; r \; d \Omega \ = \ C \| \mbf{v} \|^{2}_{U}  \, .
\end{align*}

Therefore, 
\begin{align}
\label{bd4beta}
\begin{split}
\| \beta \|_{{}_{1}L^{2}(\Omega)} &\leq C ( \| p \|_{Q} +  \| \theta \|_{{}_{1}L^{2}(\Omega)} + \| \gamma \|_{{}_{1}L^{2}(\Omega)} ) 
\ \leq  \ C ( \| \mbf{v} \|_{U} + \| p \|_{Q} ).
\end{split}
\end{align}

Next, $(\tautens^{2}, \tau^{2}) \in \bSigma$ is constructed using Lemma \ref{lem:consttau2}, with 
$\beta \, \rightarrow \, - \beta$.  For such a $(\tautens^{2}, \tau^{2})$,
it follows that 
\begin{align}
\label{eq:tautens2_axidiv_zero}
\axidiv(\tautens^{2}, \tau^{2})  = \mbf{0} , \; \; & \text{  hence  } \; \; b((\tautens^{2}, \tau^{2}), \mbf{v}) = 0  \, , \\
\label{eq:as_tautens_2}
\text{as}(\tautens^{2})  \ = \ \begin{pmatrix} 0 & - \beta \\ \beta & 0 \end{pmatrix} \, , \ \ & \quad
\left( \text{as}(\tautens^{2}) \, , \, \mathcal{S}^{2}(p) \right) \ = \  \left( \mathcal{S}^{2}(- \beta) \, , \, \mathcal{S}^{2}(p) \right)
 \, , \\
\mbox{and,   } \ \ 
\label{eq:tau_two_norm}
\| (\tautens^{2}, \tau^{2}) \|_{\bSigma} 
& \leq C \| \beta \|_{{}_{1}L^{2}(\Omega)} \leq C ( \| \mbf{v} \|_{U} + \| p \|_{Q} ) \, . 
\end{align}

As a result, for $(\tautens, \tau) = (\tautens^{1}, \tau^{1}) + (\tautens^{2}, \tau^{2})$ we have using \eqref{eq:tau_bound_with_u_1} and \eqref{eq:tautens2_axidiv_zero} 
\begin{align}
b((\tautens, \tau), \mbf{v}) = b((\tautens^{1}, \tau^{1}), \mbf{v}) + b((\tautens^{2}, \tau^{2}), \mbf{v}) = \| \mbf{v} \|^{2}_{U},
\label{eqbdds1}
\end{align}
and
\begin{align}
c((\tautens, \tau), p) &= c((\tautens^{1}, \tau^{1}), p) + c((\tautens^{2}, \tau^{2}), p)  \nonumber \\
&= (\tautens^{1}, \mathcal{S}^{2}(p))_{} + (\axidiv (\tautens^{1}, \tau^{1}) \wedge \mbf{x}, p)_{}  
 + (\tautens^{2}, \mathcal{S}^{2}(p))_{} + (\axidiv (\tautens^{2}, \tau^{2}) \wedge \mbf{x}, p)_{}  \nonumber \\
&= (\tautens^{1}, \mathcal{S}^{2}(p))_{} + (\axidiv \tautens^{1} \wedge \mbf{x}, p)_{} +(\tautens^{2}, \mathcal{S}^{2}(p))_{}, \; \; (\text{using } \axidiv (\tautens^{2}, \tau^{2}) = \mbf{0}) \nonumber \\
&= (as(\tautens^{1}), \mathcal{S}^{2}(p)) + 2 (\gamma, p) + (as(\tautens^{2}), \mathcal{S}^{2}(p)), \; \; (\text{using \eqref{eq:theta_gamma_defs}})  \nonumber \\
&= (\mathcal{S}^{2}(\theta), \mathcal{S}^{2}(p)) + (\mathcal{S}^{2}(\gamma), \mathcal{S}^{2}(p))
 + (\mathcal{S}^{2} (-\beta), \mathcal{S}^{2}(p))_{}, \; \; (\text{using \eqref{eq:as_tautens_2}})  \nonumber \\
& = (\mathcal{S}^{2}(\frac{1}{2} p), \mathcal{S}^{2}(p))  \nonumber \\
& = \| p \|_{Q}^{2}.   \label{eq:c_operator_norm_p}
\end{align}
Thus, from \eqref{eqbdds1} and \eqref{eq:c_operator_norm_p}, 
$(\tautens, \tau)$ satisfies \eqref{eq:B_operator}, and using \eqref{eq:tau_one_norm} and \eqref{eq:tau_two_norm},
\begin{align}
\| (\tautens, \tau) \|_{\bSigma} &\leq \| (\tautens^{1}, \tau^{1}) \|_{\bSigma} + \| ( \tautens^{2}, \tau^{2} ) \|_{\bSigma} 
\ \leq \ C ( \| \mbf{v} \|_{Q} + \| p \|_{Q} ).
\end{align}
\end{proof}

 \setcounter{equation}{0}
\setcounter{figure}{0}
\setcounter{table}{0}
\setcounter{theorem}{0}
\setcounter{lemma}{0}
\setcounter{corollary}{0}
\setcounter{definition}{0}
\section{Discrete Axisymmetric Variational Formulation}
\label{sec:discrete_axisymmetric_problem}
In this section we present the setting for the approximation of \eqref{eq:axi_saddle_point_1_mod}-\eqref{eq:axi_saddle_point_3_mod}.
We begin by introducing the approximation spaces used:
\begin{align}
\label{eq:space_1}
\bSigma_{h} := \Sigma_{h,\sigtens} \times \Sigma_{h,\sigma} = \{ (\sigtens_{h}, \sigma_{h}) : \sigtens_{h} \in \Sigma_{h, \sigtens} , \sigma_{h} \in \Sigma_{h, \sigma} \} \subset \bSigma   \, , \ \ 
U_{h} \subset U  ,  \ \mbox{ and } \ Q_{h}  \subset Q  \, .
\end{align}


We assume that there exists a piecewise polynomial space $(\Theta_{h})^{2}$ such that $( (\Theta_{h})^{2}, Q_{h})$ is a stable axisymmetric Stokes pair. Additionally we assume that the solution, $\mbf{w}_{h} = (w_{h 1} \, , \, w_{h 2})^{t} \in (\Theta_{h})^{2}$
 to the modified discrete axisymmetric Stokes problem: 
\textit{Given $\beta \in {}_{1}L^{2}(\Omega)$, determine $(\mbf{w}_{h} , p_{h})  \in ( (\Theta_{h})^{2}, Q_{h})$, such that
for all $(\mbf{v}_{h} , q_{h})  \in ( (\Theta_{h})^{2}, Q_{h})$}
\begin{align}
\label{eq:stokes_requirement_0}
(\Grad \mbf{w}_{h} \, : \, \Grad \mbf{v}_{h} ) \ + \ ( \frac{1}{r} w_{h 1} \, , \,  \frac{1}{r} v_{h 1} ) \ + \ 
(p_{h} , \axidiv \mbf{v}_{h}) &+ \ 
\sum_{T \in \mathcal{T}_{h}}  \left( r \dfrac{\partial^{2} w_{h  2} }{\partial z^{2}} , r \dfrac{\partial^{2} v_{h  2} }{\partial z^{2}} \right)_{T}
  = 0 , \\
\label{eq:stokes_requirement_1}
(\axidiv \mbf{w}_{h}, q_{h}) &= (\beta, q_{h}),  
\end{align}
satisfies
\begin{align}
\label{eq:stokes_requirement_2}
 \| \mbf{w}_{h} \|_{{}_{1}\mbf{VH}^{1}(\Omega)} \ + \ 
 \left(\sum_{T \in \mathcal{T}_{h}} \left\| r \dfrac{\partial^{2} w_{h  2} }{\partial z^{2}} \right\|^{2}_{{}_{1}L^{2}(T)} \right)^{\frac{1}{2}} 
&\leq C \; \| \beta \|_{{}_{1}L^{2}(\Omega)} \, .
\end{align}

\textbf{Remark}: The discrete space $(\Theta_{h})^{2}$ is a subspace of ${}_{1}\mbf{VH}^{1}(\Omega)$.

The discrete axisymmetric meridan problem with weak symmetry is then: 
\textit{Given $\mbf{f} \in {}_{1}\mbf{L}^{2}(\Omega)$ find $((\sigtens_{h}, \sigma_{h}), \mbf{w}_{h}, p_{h}) \in \bSigma_{h} \times U_{h} \times Q_{h}$ such that for all $((\tautens_{h}, \tau_{h}), \mbf{v}_{h}, q_{h}) \in \bSigma_{h} \times U_{h} \times Q_{h}$}
\begin{align}
\label{eq:axi_saddle_point_1_mod_discrete}
a((\sigtens_{h},\sigma_{h}),(\tautens_{h},\tau_{h})) + b((\tautens_{h},\tau_{h}), \mbf{w}_{h} ) + c((\tautens_{h},\tau_{h}), p_{h} ) &= (\mbf{f}, \nabla_{\text{axi}} \cdot (\tautens_{h}, \tau_{h}) ) \\
\label{eq:axi_saddle_point_2_mod_discrete}
b((\sigtens_{h},\sigma_{h}), \mbf{v}_{h} )   &= (\mbf{f},\mbf{v}_{h})  \\
\label{eq:axi_saddle_point_3_mod_discrete}
c((\sigtens_{h},\sigma_{h}), q_{h} ) & = (\mbf{f} \wedge \mbf{x}, q_{h})  \, .
\end{align}

\subsection{Well posedness of the discrete variational formulation 
\eqref{eq:axi_saddle_point_1_mod_discrete}-\eqref{eq:axi_saddle_point_3_mod_discrete}}
\label{sec:discrete_axi_inf_sup}
Analogous to the continuous formulation, the well posedness of 
\eqref{eq:axi_saddle_point_1_mod_discrete}-\eqref{eq:axi_saddle_point_3_mod_discrete} relies on the 
boundedness and coercivity of $a(\cdot , \cdot)$ on $\bSigma_{h} \times \bSigma_{h}$ and 
that $\bSigma_{h} \times (U_{h} \times Q_{h})$ satisfy the
inf-sup condition
\begin{align}
\label{eq:discrete_inf_sup_condition}
\inf_{\mbf{v}_{h} \in U_{h}, p_{h} \in Q_{h}} \sup_{ (\tautens_{h}, \tau_{h}) \in \bSigma_{h}} 
\dfrac{b( (\tautens_{h}, \tau_{h}),\mbf{v}_{h}) + c( (\tautens_{h}, \tau_{h}),p_{h})}{\|  (\tautens, \tau) \|_{\bSigma} ( \| \mbf{v}_{h} \|_{U} + \| p_{h} \|_{Q} )} \geq C.
\end{align}

To establish \eqref{eq:discrete_inf_sup_condition} we use Fortin's Lemma \cite{BBF}.  
Given $\mbf{u}_{h} \in U_{h} \subset U$, $p_{h} \in Q_{h} \subset Q$ we determine, as in the proof of 
Lemma \ref{lem:b_operator_axi}, a $(\tautens, \tau) = (\tautens^{1}, \tau^{1}) + (\tautens^{2}, \tau^{2})$ such 
that the continuous inf-sup condition is satisfied.  Then, using a suitably defined projection  
(see \eqref{eq:rz_projection_1}-\eqref{eq:rz_projection_3}), we obtain $(\tautens_{h}, \tau_{h}) \in \bSigma_{h}$ 
satisfying \eqref{eq:discrete_inf_sup_condition}.

Helpful in this discussion is to define the restriction of the operators $b(\cdot, \cdot)$ and $c(\cdot,\cdot)$ to $T \in \mathcal{T}_{h}$:
\begin{align}
b((\tautens, \tau), \mbf{u})_{T} &= (\axidiv \tautens, \mbf{u})_{T} - (\dfrac{\tau}{r}, u_{r})_{T}   \label{bres2R} \\
c((\tautens, \tau), p)_{T} &= (as(\tautens), \mathcal{S}^{2}(p))_{T} + ((\axidiv (\tautens, \tau)) \wedge \mbf{x}, p)_{T}.
   \label{cres2T}
\end{align}

Next, we present the following identity for the operator $c(\cdot, \cdot)_{T}$.
\begin{lemma}
\label{lem:c_operator_result}
For $T \in \mathcal{T}_{h}$, 
\begin{align}
\label{eq:exterior_result_5}
c((\tautens, \tau),p)_{T} = \int_{\partial T} (\tautens \cdot \mbf{n}) \cdot \mbf{x}^{\perp} \; p \; r \;ds - \int_{T} \tautens : (\mbf{x}^{\perp} \otimes \nabla p) \; r \; dT - \int_{T} \tau \; z \; p \; dT.
\end{align}
\end{lemma}
\begin{proof}
See Section \ref{apxpLm1} in the appendix.
\end{proof}
\begin{theorem}
\label{lem:thm_41}
Assume $\bSigma_{h}, U_{h}, Q_{h}$ satisfy \eqref{eq:space_1}.  Let  \vspace{0.5em} \\
$\boldsymbol{S}_{\Theta_{h}} = 
\left\{ (\tautens , \tau) \, : \, \tautens \, = \, \begin{pmatrix} \dfrac{\partial w_{h1} }{\partial z} &
 -\dfrac{1}{r} \dfrac{\partial}{\partial r}{(r \; w_{h1})} - \dfrac{\partial w_{2h}}{\partial z} \\ 0 & 0 \end{pmatrix} \, , \
\tau \, = \, r \; \dfrac{\partial^{2} \;  w_{h2} }{\partial z^{2}} ; \ \mbf{w}_{h} = (w_{h 1} , w_{h 2})^{t} \in \left( \Theta_{h} \right)^{2} \right\}$~. 
\linebreak[4]
If there exists a mapping $\bPi_{h} = \Pi_{h} \times \pi_{h} : (\boldsymbol{S} + \boldsymbol{S}_{\Theta_{h}} ) \rightarrow \bSigma_{h}$ such that : 
\begin{align}
\label{eq:projection_bound_1}
 \| \Pi_{h} \times \pi_{h} (\tautens , \tau) \|_{\bSigma} &\leq \ C \, \| (\tautens , \tau) \|_{\boldsymbol{S}}, \quad 
  \forall (\tautens , \tau) \in \boldsymbol{S},   \\
\label{eq:projection_bound_2}
 \| \Pi_{h} \times \pi_{h} (\tautens , \tau) \|_{\bSigma} &\leq \ C \, \| \mbf{w}_{h} \|_{{}_{1}\mbf{VH}^{1}(\Omega)} 
 \ + \ 
 \left(\sum_{T \in \mathcal{T}_{h}} \left\| r \dfrac{\partial^{2} w_{h  2} }{\partial z^{2}} \right\|^{2}_{{}_{1}L^{2}(T)} \right)^{\frac{1}{2}} , \quad
 \forall (\tautens , \tau) \in \boldsymbol{S}_{\Theta_{h}} , 
\end{align}
and  for all $T \in \mathcal{T}_{h}$
\begin{align}
\label{eq:rz_projection_1}
& \int_{T} (\tautens - \Pi_{h} \tautens) : ( \nabla \mbf{u}_{h} + \mbf{x}^{\perp} \otimes \nabla q_{h} ) \; r \; dT \ = \ 0 , \quad \forall  \mbf{u}_{h} \in U_{h},  \; \forall \; q_{h} \in Q_{h}, \\
\label{eq:rz_projection_2}
& \int_{\ell}  ((\tautens - \Pi_{h} \tautens) \cdot \mbf{n}_{K}) \cdot (\mbf{u}_{h} + \mbf{x}^{\perp} q_{h} ) \; r \; ds \ =  \ 0 , \quad \forall \text{ edges } \ell, \; \forall \; \mbf{u}_{h} \in U_{h}, \; \forall \; q_{h} \in Q_{h}, \\
\label{eq:rz_projection_3}
& \int_{T} \frac{1}{r} (\tau - \pi_{h} \tau) \; \sigma \; r \; dT \ = \ 0 , \quad \forall \; \sigma \in \{ z \, q_{h} \, :  \, q_{h} \in Q_{h} \} \cup
\{ v_{h 1} \, : \, ( v_{h 1} \, , \, 0 ) \in U_{h} \} \, , 
\end{align}
then $\bSigma_{h} \times \left( U_{h} \times Q_{h} \right)$ are inf-sup stable.
\end{theorem}
\begin{proof}
The approach to this proof is similar to that used in \cite{BBF1}.  Let $\mbf{v}_{h} = (v_{h1}, v_{h2})^{t} \in U_{h} \subset {}_{1}\mbf{L}^{2}(\Omega, \mathbb{R}^{2})$ and $p_{h} \in Q_{h} \subset {}_{1}L^{2}(\Omega)$.   
 
Recall from Lemma \ref{lem:b_operator_axi} that for $\mbf{v}_{h}$ given, there exists 
$(\tautens^{1} , 0) \in \boldsymbol{S}(\Omega) \subset \bSigma(\Omega)$ such that
\begin{align}
\label{eq:tau_bound_with_u}
\axidiv (\tautens^{1}, 0)  = \mbf{v}_{h} \quad \text{ and } \quad \| (\tautens^{1},0) \|_{\bS} \leq C \left( \| \mbf{v}_{h} \|_{U} + \| p_{h} \|_{Q} \right).
\end{align}
Also, from \eqref{eq:tau_bound_with_u_1},
 $b((\tautens^{1}, 0),\mbf{v}_{h}) = (\mbf{v}_{h}, \mbf{v}_{h})$, hence
\begin{align}
\begin{split}
\label{eq:b_operator_thm41}
b((\tautens^{1}, 0) - \boldsymbol{\Pi}_{h}(\tautens^{1}, 0), \mbf{v}_{h}) &= 
b((\tautens^{1} - \Pi_{h} \tautens^{1}, 0) , \mbf{v}_{h} ) \\ &= \sum\limits_{T \in \mathcal{T}_{h} }(\axidiv (\tautens^{1} - \Pi_{h} \tautens^{1}), \mbf{v}_{h})_{T} + (\dfrac{0}{r},  v_{h 1})_{T} \\
 &= 0 \ \ \mbox{(using \eqref{eq:rz_projection_1}-\eqref{eq:rz_projection_2} with $q_{h} = 0$)}. 
\end{split}
\end{align}
Furthermore, from \eqref{eq:projection_bound_1} and \eqref{eq:tau_bound_with_u}, 
\begin{align}
\| \boldsymbol{\Pi_{h}} (\tautens^{1}, 0) \|_{\bSigma} \leq C \| (\tautens^{1}, 0) \|_{\bS}  \leq C (\|\mbf{v}_{h} \|_{U} + \| p_{h} \|_{Q}).
\end{align}
Next, combining \eqref{eq:exterior_result_5} with \eqref{eq:rz_projection_1}-\eqref{eq:rz_projection_2} (with $\mbf{u}_{h} = 0$),  
\begin{align}
\label{eq:c_operator_thm41}
\begin{split}
& c((\tautens^{1}, 0) - \boldsymbol{\mbf{\Pi}_{h}}(\tautens^{1}, 0), p_{h}) \ =  \ \sum_{T \in \mathcal{T}_{h}}
c((\tautens^{1} - \Pi_{h} \tautens^{1}, 0) , p_{h} )_{T}  \\
& \quad  = \sum_{T \in \mathcal{T}_{h}} \left( \int_{\partial T} ((\tautens^{1} - \Pi_{h} \tautens^{1}) \cdot \mbf{n}_{T} ) \cdot \mbf{x}^{\perp} \; p_{h} \; r \; ds 
\ - \int_{T} (\tautens^{1} - \Pi_{h} \tautens^{1} ) : (\mbf{x}^{\perp} \otimes \nabla \; p_{h} ) \; r \; dT  \right. \\
& \quad \quad \quad \quad \quad \quad \quad \quad \left. - \int_{T} \frac{1}{r} \; 0 \; z \; p_{h} \; r \; dT \right) \ =  \ 0 .
\end{split}
\end{align}

The $(\tautens^{2} , \tau^{2})$ used in establishing the continuous inf-sup condition is not sufficiently regular in order to 
construct a suitable projection. To circumvent this problem we use 
\eqref{eq:stokes_requirement_0},\eqref{eq:stokes_requirement_1} to determine
a suitable replacement for $(\tautens^{2} , \tau^{2})$, namely $(\tautens^{2}_{h} , \tau^{2}_{h})$, and then use
a projection of $(\tautens^{2}_{h} , \tau^{2}_{h})$ to help satisfy \eqref{eq:discrete_inf_sup_condition}.

Let $\mbf{w}_{h} \in (\Theta_{h})^{2}$ be determined by \eqref{eq:stokes_requirement_0},\eqref{eq:stokes_requirement_1},
and define
\begin{align*}
\tautens_{h}^{2} \, = \, 2 \begin{pmatrix} \dfrac{\partial w_{h1} }{\partial z} & -\dfrac{1}{r} \dfrac{\partial}{\partial r}{(r \; w_{h1})} - \dfrac{\partial w_{h2}}{\partial z} \\ 0 & 0 \end{pmatrix} \ \  \text{ and } \ \ 
 \tau_{h}^{2} = - 2 \; \pi_{h} \left( r \; \dfrac{\partial^{2} \;  w_{h2} }{\partial z^{2}} \right). 
\end{align*}

Then, (cf. \eqref{eq:tautens2_axidiv_zero}-\eqref{eq:tau_two_norm})
\begin{align}
& b((\tautens_{h}^{2}, \tau_{h}^{2}) , \mbf{v}_{h})  =
 \sum\limits_{T \in \mathcal{T}_{h} } \left(\axidiv \tautens_{h \, 1}^{2} \ - \ \frac{1}{r} (\tau_{h}^{2}) , v_{h 1} \right)_{T}  \nonumber \\
&
 \ \ \ = \sum\limits_{T \in \mathcal{T}_{h} } \left(\axidiv (\frac{\partial w_{h1} }{\partial z}  \, , \,  -\frac{1}{r} \dfrac{\partial}{\partial r}{(r \; w_{h1})} - \frac{\partial w_{h2}}{\partial z} ) \ - \ \frac{1}{r} ( - 2 \; \pi_{h} ( r \; \frac{\partial^{2} \;  w_{h2} }{\partial z^{2}} )) , 
 v_{h 1} \right)_{T}   \nonumber \\  
& \ \ \ = \sum\limits_{T \in \mathcal{T}_{h} } 2 \,  ( \frac{\partial^{2} w_{h1} }{\partial r \partial z} + \frac{1}{r} \frac{\partial w_{h1} }{\partial z}
-  \frac{1}{r} \frac{\partial w_{h1} }{\partial z} -  \frac{\partial^{2} w_{h1} }{\partial r \partial z} - 
 \frac{\partial^{2} w_{h2} }{\partial^{2} z} +  \frac{\partial^{2} w_{h2} }{\partial^{2} z}  , 
 v_{h 1})_{T} \ \ \ \mbox{ (using \eqref{eq:rz_projection_3})}   \nonumber \\
& \ \ \ = 0  \ = \ b((\tautens^{2}, \tau^{2}) , \mbf{v}_{h})  \, .   \label{bneweq1}
\end{align}

From \eqref{eq:stokes_requirement_1}),
\begin{align}
\label{eq:mm2p5}
\left(  \text{as} (\tautens_{h}^{2} ) \, , \, \mathcal{S}^{2}(p_{h}) \right) 
&= \  \left( \mathcal{S}^{2}(- \beta) \, , \, \mathcal{S}^{2}(p_{h}) \right) \, .
\end{align}

For $c( (\tautens_{h}^{2}, \tau_{h}^{2}) , p_{h})$ we have
\begin{align*}
c( (\tautens_{h}^{2}, \tau_{h}^{2}) , p_{h}) &= (\tautens_{h}^{2}, \mathcal{S}^{2}(p_{h}))_{} 
+ (\axidiv (\tautens_{h}^{2}, \tau_{h}^{2}) \wedge \mbf{x}, p_{h})_{}   \\
&= \ (as(\tautens_{h}^{2}, \mathcal{S}^{2}(p_{h})) \ + \ 
\sum\limits_{T \in \mathcal{T}_{h} } \left(\axidiv \tautens_{h \, 1}^{2} \ - \ \frac{1}{r} (\tau_{h}^{2}) , z \, p_{h} \right)_{T} \\
&= \ (\mathcal{S}^{2}(- \beta) , \mathcal{S}^{2}(p_{h})) \ + \  0 \ \ \mbox{ (as in \eqref{bneweq1}) } \\
&= \ c( (\tautens^{2}, \tau^{2}) , p_{h}) \, .
\end{align*}

Using \eqref{eq:projection_bound_2} and \eqref{eq:stokes_requirement_2}
\[
\| \tau_{h}^{2} \|_{{}_{1}L^{2}(\Omega)} \ = \ \| 2 \; \pi_{h} \left( r \; \dfrac{\partial^{2} \;  w_{h2} }{\partial z^{2}} \right)
 \|_{{}_{1}L^{2}(\Omega)} 
\ \le \ C \,  \left(\sum_{T \in \mathcal{T}_{h}} \left\| r \dfrac{\partial^{2} w_{h  2} }{\partial z^{2}} \right\|^{2}_{{}_{1}L^{2}(T)} \right)^{\frac{1}{2}} 
\ \le \ C \, \beta \, .
\]

Then, from \eqref{eq:stokes_requirement_2}
 and  $\| \beta \|_{{}_{1}L^{2}(\Omega)} \leq C \left( \| \mbf{v}_{h} \|_{U} + \| p_{h} \|_{Q} \right)$,  it follows that
\begin{align}
\label{eq:mm3}
\| (\tautens_{h}^{2}, \tau_{h}^{2}) \|_{\bSigma} \leq C \left( \| \mbf{v}_{h} \|_{U} + \| p_{h} \|_{Q} \right).
\end{align}

Now, for $(\Pi_{h} \tautens_{h}^{2}, \tau^{2}_{h})  \in \bSigma_{h} $, proceeding as in \eqref{eq:b_operator_thm41}; using \eqref{eq:rz_projection_1}-\eqref{eq:rz_projection_2} (with $q_{h} = 0$), and \eqref{eq:rz_projection_3} 
\begin{align}
\label{eq:mm4}
\begin{split}
& b((\tautens_{h}^{2}, \tau_{h}^{2}) - (\Pi_{h}  \tautens_{h}^{2}, \tau_{h}^{2}) \ , \ \mbf{v}_{h}) \ = \ b((\tautens_{h}^{2} - \Pi_{h} \tautens_{h}^{2}, 0) \ , \ \mbf{v}_{h})  \\
& \quad \quad = \sum_{T \in \mathcal{T}_{h}} \left( \axidiv (\tautens_{h}^{2} - \Pi_{h} \tautens_{h}^{2}), \mbf{v}_{h}\right)_{T} 
\  =  \ 0.
\end{split}
\end{align}

Also, as in \eqref{eq:c_operator_thm41}, and using \eqref{eq:rz_projection_3},
\begin{align}
\label{eq:mm5}
\begin{split}
& c( (\tautens_{h}^{2}, \tau_{h}^{2}) - (\Pi_{h} \tautens_{h}^{2}, \tau_{h}^{2}) , p_{h}) \ = \ \sum_{T \in \mathcal{T}_{h}} c( (\tautens_{h}^{2} - \Pi_{h} \tautens_{h}^{2}), 0 ), p_{h} )_{T}  \\
& = \sum_{T \in \mathcal{T}_{h}} \left( \int_{\partial T} ((\tautens_{h}^{2} - \Pi_{h} \tautens_{h}^{2}) \cdot \mbf{n}_{T} ) 
 \cdot \mbf{x}^{\perp} \;  p_{h} \; r \; ds 
\ -  \ \int_{T} (\tautens_{h}^{2} - \Pi_{h} \tautens_{h}^{2} ) : ( \mbf{x}^{\perp} \otimes \nabla p_{h} ) \; r \; d T  \right.  = 0.
\end{split}
\end{align}

Finally, with $(\tautens_{h}, \tau_{h}) = (\Pi_{h} \tautens^{1}, \pi_{h} \tau^{1}) + (\Pi_{h} \tautens_{h}^{2},  \tau_{h}^{2}) \in \bSigma_{h}$,
\begin{align*}
\begin{split}
&\sup_{(\sigtens_{h}, \sigma_{h}) \in \bSigma_{h}} \dfrac{b((\sigtens_{h}, \sigma_{h}), \mbf{v}_{h}) + c((\sigtens_{h}, \sigma_{h}), p_{h} )  }{\| (\sigtens_{h}, \sigma_{h}) \|_{\bSigma} \left( \| \mbf{v}_{h} \|_{U} + \| p_{h} \|_{Q} \right) } \geq \dfrac{b((\tautens_{h}, \tau_{h}), \mbf{v}_{h}) + c((\tautens_{h}, \tau_{h}), p_{h} )  }{\| (\tautens_{h}, \tau_{h}) \|_{\bSigma} \left( \| \mbf{v}_{h} \|_{U} + \| p_{h} \|_{Q} \right) }     \\
&\geq \dfrac{b((\Pi_{h} \tautens^{1}, \pi_{h} \tau^{1}), \mbf{v}_{h}) + c((\Pi_{h} \tautens^{1}, \pi_{h} \tau^{1}), p_{h} ) + b((\Pi_{h} \tautens^{2}_{h},  \tau^{2}_{h}), \mbf{v}_{h}) + c((\Pi_{h} \tautens^{2}_{h}, \tau^{2}_{h}), p_{h} )  }{\left( \| (\Pi_{h} \tautens^{1}, \pi_{h} \tau^{1}) \|_{\bSigma} + \| (\Pi_{h} \tautens^{2}_{h},  \tau^{2}_{h}) \|_{\bSigma}  \right) \left( \| \mbf{v}_{h} \|_{U} + \| p_{h} \|_{Q} \right) }  \\
%
& \geq C \; \dfrac{b((\tautens^{1}, \tau^{1}), \mbf{v}_{h}) + c((\tautens^{1}, \tau^{1}), p_{h} ) + 
b((\tautens_{h}^{2}, \tau_{h}^{2}) , \mbf{v}_{h}) + c(( \tautens^{2}_{h},  \tau^{2}_{h}), p_{h} )  }{\left( \| (\tautens^{1}, \tau^{1}) \|_{\bS} + \| (\tautens^{2}_{h}, \tau^{2}_{h}) \|_{\bSigma}  \right) \left( \| \mbf{v}_{h} \|_{U} + \| p_{h} \|_{Q} \right) } \\
& \geq C \; \dfrac{b((\tautens^{1}, \tau^{1}), \mbf{v}_{h}) + c((\tautens^{1}, \tau^{1}), p_{h} ) + 
b((\tautens^{2}, \tau^{2}) , \mbf{v}_{h}) + c(( \tautens^{2},  \tau^{2}), p_{h} )  }{\left( \| (\tautens^{1}, \tau^{1}) \|_{\bS} + \| (\tautens^{2}_{h}, \tau^{2}_{h}) \|_{\bSigma}  \right) \left( \| \mbf{v}_{h} \|_{U} + \| p_{h} \|_{Q} \right) } \\
%
&\geq C \, \dfrac{\|\mbf{v}_{h} \|^{2}_{U} + \|p_{h} \|^{2}_{Q}}{(\| \mbf{v}_{h} \|_{U} + \|p_{h} \|_{Q} + \| \mbf{v}_{h} \|_{U} + \| p_{h} \|_{Q} )(\| \mbf{v}_{h} \|_{U} + \| p_{h} \|_{Q})} \\
&\geq C .
\end{split}
\end{align*}
\end{proof}

 Throughout the remainder of this document, we will denote the space $(\bS + \bS_{\Theta_{h}} )$ as $\bSigma^{S}$.  Additionally, 
 we denote the tensor and scalar components of $\bSigma^{S}$ as $\Sigma^{S}_{\sigtens}$ and $\Sigma^{S}_{\sigma}$, i.e.,
  $\bSigma^{S} = \Sigma^{S}_{\sigtens} \times \Sigma^{S}_{\sigma}$.

 \setcounter{equation}{0}
\setcounter{figure}{0}
\setcounter{table}{0}
\setcounter{theorem}{0}
\setcounter{lemma}{0}
\setcounter{corollary}{0}
\setcounter{definition}{0}
\section{Approximation spaces satisfying \eqref{eq:rz_projection_1},\eqref{eq:rz_projection_2}}
\label{sec:disspace}
In this section we investigate approximation spaces for $\bSigma^{S}$, $U_{h}$, and $Q_{h}$ such that there
exists a projection operator $\boldsymbol{\mbf{\Pi}_{h}} (\tautens , \tau)$ satisfying \eqref{eq:rz_projection_1}-\eqref{eq:rz_projection_3}.

\subsection{$\bSigma_{h} \, = \, \left( \mbf{BDM}_{1}(\mcT_{h}) \right)^{2} \times P_{1}(\mcT_{h})$, 
$U_{h} \, = \, \left( P_{0}(\mcT_{h}) \right)^{2}$, and $Q_{h} \, = \, P_{0}(\mcT_{h})$}
\label{ssec:bdm_1}

In this section we show that for the choice of spaces 
$\Sigma_{h , \sigtens} \, = \, \left( \mbf{BDM}_{1}(\mcT_{h}) \right)^{2}$, $\Sigma_{h , \sigma} \, = \, P_{1}(\mcT_{h})$,
$U_{h} \, = \, \left( P_{0}(\mcT_{h}) \right)^{2}$, and $Q_{h} \, = \, P_{0}(\mcT_{h})$ there exists a 
projection operator, $\boldsymbol{\mbf{\Pi}_{h}} (\tautens , \tau)$,  satisfying \eqref{eq:rz_projection_1}-\eqref{eq:rz_projection_3}.
 
\begin{lemma}
\label{lem:bdm1_projection}
Let $T \in \mathcal{T}_{h}$.  The mappings $\Pi_{h} : \Sigma_{\sigtens}^{S}(T) \rightarrow (P_{1}(T))^{4}$ and 
$\pi_{h} : \Sigma_{\sigma}^{S}(T) \rightarrow P_{1}(T)$ given by
\begin{align}
\label{eq:BDM1_interpolation_operator_lem1}
\int_{\ell} (\tautens - \Pi_{h} \tautens) \cdot \mbf{n}_{k} \cdot \mbf{p}_{1} \; r \; ds &= 0 
\text{ for all edges } \ell \in \partial T \text{ and }  \mbf{p}_{1} \in (P_{1}(\ell))^{2}\\
\label{eq:pi_projection_operator_lem1}
\int_{T} \frac{1}{r} (\tau - \pi_{h} \tau) \; p_{1} \, r  \; dT &= 0 \text{ for all } p_{1} \in P_{1}(T)
\end{align}
are well defined.  Hence the spaces
\begin{align}
\label{eq:bdm1_discrete_function_spaces}
\Sigma_{h,\sigtens} = (\mbf{BDM}_{1}(\mcT_{h}))^{2} \quad \quad \Sigma_{h, \sigma} = P_{1}(\mcT_{h}) \quad \quad 
U_{h} = (P_{0}(\mcT_{h}))^{2} 
\quad \quad Q_{h} = P_{0}(\mcT_{h}) 
\end{align}
satisfy \eqref{eq:rz_projection_1}-\eqref{eq:rz_projection_3}.
\end{lemma}
\begin{proof}
Observe that $\pi_{h}$ is the well defined $L^{2}$ projection.   

Next we show that $\Pi_{h}$ is well defined.  Note that 
$\Pi_{h} \tautens \in ( P_{1}(T) )^{4} \, = \, (\mbf{BDM}_{1}(T))^{2}$ has 12 degrees of freedom, 
and $(P_{1}(\ell))^{2}$ has 4 degrees of freedom per edge.  Thus the number of unknowns in $\Pi_{h} \tautens$ 
is equal to the number of constraints in \eqref{eq:BDM1_interpolation_operator_lem1}.  
It follows that if $\tautens = \mbf{0}$ implies that $\Pi_{h} \tautens = \mbf{0}$, then the projection $\Pi_{h}$ is well defined.

Consider a single row of the tensor projection \eqref{eq:BDM1_interpolation_operator_lem1}.  
In this case, for $\tautens = (\boldsymbol{\tau}_{1}, \boldsymbol{\tau}_{2})^{t}$ the projection  
\eqref{eq:BDM1_interpolation_operator_lem1} takes the form
\begin{align}
\label{eq:BDM1_interpolation_operator_vector}
\int_{\ell} (\tautens_{s} - \Pi_{h} \tautens_{s} ) \cdot \mbf{n}_{k} \; p_{1} \; r \; ds = 0 \text{ for } p_{1} \in P_{1}(\ell) , \ s=1,2.
\end{align}

Next, observe that the function $\Pi_{h} \tautens_{s} \cdot \mbf{n}_{k} \; p_{1} \; r$ is a cubic polynomial.  
Recalling that a degree $n$ Gauss quadrature rule integrates polynomials of degree $2n-1$ exactly, we 
select two Gauss quadrature points $\{q^{\ell_{k}}_{i} \}_{i=1}^{2}$ on each edge $\ell_{k}$ for $k=1,2,3$.

For $\ell_{k} \in \partial K$, define a basis for $P_{1}(\ell_{k})$ so that
\begin{align} \label{eq:p1_basis}
p_{1}^{\ell_{k}} (x) =
\begin{cases}
1 \text{ if } x = q^{\ell_{k}}_{1}\\
0 \text{ if } x = q^{\ell_{k}}_{2}
\end{cases} \text{ and } \quad 
p_{2}^{\ell_{k}} (x) =
\begin{cases}
0 \text{ if } x = q^{\ell_{k}}_{1}\\
1 \text{ if } x = q^{\ell_{k}}_{2}
\end{cases}.
\end{align}

Let $\{ \boldsymbol{\phi}^{\ell_{k}}_{i} \}$ be a basis for $\mbf{BDM}_{1}(T)$ \cite{Ervin} such that 
\begin{align*}
(\boldsymbol{\phi}_{i}^{\ell_{m}} \cdot \mbf{n}) (q_{j}^{\ell_{n}}) = \delta_{(i,j),(\ell_{m},\ell_{n})} 
\text{ for } i, j = 1, 2 \text{ and } m, n = 1,2,3.
\end{align*}

Note that the normal component of the basis functions satisfy a Lagrangian property at the boundary quadrature points.  
Since $\Pi_{h}  \tautens_{s} \in \mbf{BDM}_{1}(T)$, it can be written as
\begin{align*}
\Pi_{h} \tautens_{s} = \sum_{k=1}^{3} \sum_{i=1}^{2} \alpha^{\ell_{k}}_{i} \boldsymbol{\phi}^{\ell_{k}}_{i}.
\end{align*}

With $\tautens_{s} = \mbf{0}$, taking the basis function $p_{1}^{\ell_{k}}$ for $\ell_{k} \in \partial K$ and using
 \eqref{eq:BDM1_interpolation_operator_lem1} and Gaussian quadrature gives  
\begin{align*}
0 = \int_{\ell_{k}} \Pi_{h} \tautens_{s} \cdot \mbf{n} \; p_{1} \; r \; ds &= \sum_{j=1}^{2} 
(\Pi_{h} \tautens_{s} \cdot \mbf{n}) (q^{\ell_{k}}_{j}) 
\cdot p^{\ell_{k}}_{1} (q^{\ell_{k}}_{j}) \; r(q^{\ell_{k}}_{j}) \; w(q_{j}^{\ell_{k}})  \\
&= \alpha^{\ell_{k}}_{1} \; p^{\ell_{k}}_{1} (q^{\ell_{k}}_{1}) \; r(q^{\ell_{k}}_{1}) w(q_{1}^{\ell_{k}}) 
+ \alpha^{\ell_{k}}_{2} \; p^{\ell_{k}}_{1} (q^{\ell_{k}}_{2}) \; r(q^{\ell_{k}}_{2}) w(q_{2}^{\ell_{k}}) \\
& = \alpha_{1}^{\ell_{k}} r(q_{1}^{\ell_{k}}) w(q_{1}^{\ell_{k}}).
\end{align*}

In the case where $r(q_{1}^{\ell_{k}}) \neq 0$, this implies $\alpha^{\ell_{k}}_{1} = 0$.  If, however, $r(q_{1}^{\ell_{k}}) = 0$, 
then $\alpha_{1}^{\ell_{k}}$ and $\beta_{1}^{\ell_{k}}$ must be zero, otherwise, the normal stress along the axis of symmetry 
will be non-zero implying that the solution is not axisymmetric.  A similar argument can be used to show that the other 
$\alpha$ terms are also zero.  
Hence the vector projection from \eqref{eq:BDM1_interpolation_operator_vector} is well defined.

To extend the vector projection from \eqref{eq:BDM1_interpolation_operator_vector} to the tensor projection given
in
 \eqref{eq:BDM1_interpolation_operator_lem1}, we extend the basis for $P_{1}(\ell_{k})$ from
  \eqref{eq:p1_basis} to $(P_{1}(\ell_{k}))^{2}$ by using
\begin{align*}
(P_{1}(\ell_{k}))^{2} \ = \ \mbox{span} 
\left \{ \begin{pmatrix} p_{1}^{\ell_{k}} \\ 0 \end{pmatrix}, \quad \begin{pmatrix} p_{2}^{\ell_{k}} \\ 0 \end{pmatrix}, 
\quad \begin{pmatrix} 0 \\ p_{1}^{\ell_{k}} \end{pmatrix}, \quad \begin{pmatrix} 0 \\ p_{2}^{\ell_{k}} \end{pmatrix}\right \}.
\end{align*} 

With this basis, the arguments presented above for the vector case can be applied to each row 
of \eqref{eq:BDM1_interpolation_operator_lem1} to show that $\Pi_{h}$ is well defined.

Lastly, we verify that the spaces given in \eqref{eq:bdm1_discrete_function_spaces} satisfy the conditions outlined in 
\eqref{eq:rz_projection_1}-\eqref{eq:rz_projection_3}.  Since gradients of the piecewise constant spaces 
$U_{h}$ and $Q_{h}$ are zero on each element $T$, \eqref{eq:rz_projection_1} is trivially satisfied. 
Next, observe that the test space of \eqref{eq:BDM1_interpolation_operator_lem1} includes all $\mbf{p}_{1} \in (P_{1}(\ell_{k}))^{2}$ 
for $k=1,2,3$, while \eqref{eq:rz_projection_2} only requires that the projection is satisfied on a subspace of $(P_{1}(\ell_{k}))^{2}$. 
 Finally, since $P_{0}(T) \subset P_{1}(T)$, \eqref{eq:pi_projection_operator_lem1} ensures that \eqref{eq:rz_projection_3} is satisfied.
\end{proof}

\subsection{$\bSigma_{h} \, = \, \left( \mbf{BDM}_{2}(\mcT_{h}) \right)^{2} \times P_{2}(\mcT_{h})$, 
$U_{h} \, = \, \left( P_{1}(\mcT_{h}) \right)^{2}$, and $Q_{h} \, = \, P_{1}(\mcT_{h})$}
\label{ssec:bdm_2}

In this section we show that for the choice of spaces 
$\Sigma_{h , \sigtens} \, = \, \left( \mbf{BDM}_{2}(\mcT_{h}) \right)^{2}$, $\Sigma_{h , \sigma} \, = \, P_{2}(\mcT_{h})$,
$U_{h} \, = \, \left( P_{1}(\mcT_{h}) \right)^{2}$, and $Q_{h} \, = \, P_{1}(\mcT_{h})$ there exists a 
projection operator, $\boldsymbol{\mbf{\Pi}_{h}} (\tautens , \tau)$,  satisfying \eqref{eq:rz_projection_1}-\eqref{eq:rz_projection_3}.
 
\begin{lemma}
\label{lem:BDM2_projection_stable}
Let $T \in \mathcal{T}_{h}$.  The projection operators 
$\Pi_{h}: \Sigma_{\sigtens}^{S}(T) \rightarrow (P_{2}(T))^{4}$ and $\pi_{h} : \Sigma_{\sigma}^{S}(T) \rightarrow P_{2}(T)$ given by
\begin{align}
\label{eq:BDM2_interpolation_operator_element_1}
\int_{T} (\tautens - \Pi_{h} \tautens) : (\underline{\mbf{p_{0}}} + \mbf{x}^{\perp} \otimes \mbf{p_{0}}) \; r \; d T &= 0 \quad 
\forall \; \underline{\mbf{p_{0}}} \in (P_{0}(T))^{2 \times 2} \quad \forall \; \mbf{p_{0}} \in (P_{0}(T))^{2} \quad  \\
\label{eq:BDM2_interpolation_operator_edge_1}
\int_{\ell} (\tautens - \Pi_{h} \tautens) \cdot \mbf{n}_{k} \cdot \mbf{p}_{2} \; r \; ds &= 0 \quad 
\forall \text{ edges } \ell \quad \forall \; \mbf{p}_{2} \in (P_{2}(\ell))^{2} \\
\label{eq:BDM2_interpolation_operator_p2}
\int_{T}  \frac{1}{r} (\tau - \pi_{h} \tau) \; p_{2} \; r \; d T &= 0 \text{ for all } p_{2} \in P_{2}(T)
\end{align}
are well defined. Hence the spaces
\begin{align}
\label{eq:bdm2_discrete_function_spaces}
\Sigma_{h,\sigtens} = (\mbf{BDM}_{2}(\mcT_{h}))^{2} \quad \quad \Sigma_{h, \sigma} = P_{2}(\mcT_{h})
 \quad \quad U_{h} = (P_{1})^{2}(\mcT_{h}) 
\quad \quad Q_{h} = P_{1}(\mcT_{h}) 
\end{align}
satisfy \eqref{eq:rz_projection_1}-\eqref{eq:rz_projection_3}.
\end{lemma}
\begin{proof}
Observe that $\pi_{h}$ is the well defined $L^{2}$ projection.

Next we show that $\Pi_{h}$ is well defined.  First observe that the number of constraints defined by $\Pi_{h}$, 24, 
is the same as number of degrees of freedom in $( P_{2}(T) )^{4} \, = \, (\mbf{BDM}_{2}(T))^{2}$. 
We verify that the projection is injective by showing that
\begin{align}
\label{eq:tau_zero_1}
\int_{T} \Pi_{h} \tautens : (\underline{\mbf{p_{0}}} + \mbf{x}^{\perp} \otimes \mbf{p_{0}}) \; r \; d T &= 0 \quad 
\forall \; \underline{\mbf{p_{0}}} \in (P_{0}(T))^{2 \times 2} \quad \forall \; \mbf{p_{0}} \in (P_{0}(T))^{2} \\
\label{eq:tau_zero_2}
\int_{\ell} \Pi_{h} \tautens \cdot \mbf{n}_{k} \cdot \mbf{p}_{2} \; r \; ds &= 0 \quad 
\forall  \text{ edges } \ell \quad \forall \; \mbf{p}_{2} \in (P_{2}(\ell))^{2}
\end{align}
has the unique solution $\Pi_{h} \tautens = \mbf{0}$.

We can represent $\Pi_{h} \tautens$ in terms of the basis for $(\mbf{BDM}_{2}(\widehat{T}))^{2}$, where 
$\mbf{BDM}_{2}(\widehat{T})$ is the reference element representation presented in \cite[Section 4.2]{Ervin}.  
This $\mbf{BDM}_{2}(\widehat{T})$ basis is expressed in terms of edge and interior element functions.  Using 
equation \eqref{eq:tau_zero_2} with three Gauss quadrature points
and an argument analogous to that used in the proof of Lemma \ref{lem:bdm1_projection}, it follows that all 18 of the $\mbf{BDM}_{2}(\widehat{T})$ edge basis functions must equal zero.

Therefore, the only possible non-zero basis functions on $\widehat{T}$ are the interior element functions
\begin{align}
\begin{split}
\label{eq:bdm2_interior_funcs}
\underline{\phi}_{1} = \dfrac{\sqrt{2}}{(g_{2} - g_{1})} (1 - \xi - \eta)  &\begin{pmatrix} g_{2} \xi \\ (g_{2} - 1) \eta \end{pmatrix} \quad 
\underline{\phi}_{2} = \dfrac{1}{(g_{2}-g_{1})} \xi \begin{pmatrix} g_{2} \xi + \eta - g_{2} \\ (g_{2} - 1) \eta \end{pmatrix} \\[8pt]
\underline{\phi}_{3} &= \dfrac{1}{(g_{2} - g_{1})} \eta \begin{pmatrix} (g_{2} - 1) \xi \\ \xi + g_{2} \eta - g_{2} \end{pmatrix}
\end{split}
\end{align}
where $g_{1} = 1 /2 - \sqrt{3} / 6$ and $g_{2}= 1 / 2 + \sqrt{3} / 6$ are the Gaussian quadrature points on $[0,1]$.  
Thus, $\widehat{\Pi_{h} \tautens}$, the representation of $\Pi_{h} \tautens$ on $\widehat{T}$, must have the 
form $\widehat{\Pi_{h} \tautens} = \begin{pmatrix} \underline{\phi_{\alpha}}^{t} \\ \underline{\phi_{\beta}}^{t} \end{pmatrix} $
where
\begin{align*}
\underline{\phi_{\alpha}}^{t} = \alpha_{1} \underline{\phi_{1}}^{t} + \alpha_{2} \underline{\phi_{2}}^{t} + \alpha_{3} \underline{\phi_{3}}^{t}
\quad \text{ and } \quad 
\underline{\phi_{\beta}}^{t} = \beta_{1} \underline{\phi_{1}}^{t} + \beta_{2} \underline{\phi_{2}}^{t} + \beta_{3} \underline{\phi_{3}}^{t}.
\end{align*}

It remains to show that $\alpha_{i} = \beta_{i}=0$ for $i=1,2,3$.  To do so, we consider the matrix representation of equation 
\eqref{eq:BDM2_interpolation_operator_element_1}.  The functions in \eqref{eq:bdm2_interior_funcs} can be used as
 the six trial basis functions of \eqref{eq:BDM2_interpolation_operator_element_1}, while the test space of
 \eqref{eq:BDM2_interpolation_operator_element_1} has dimension 6, and is spanned by the functions
\begin{align}
\label{eq:interior_test_functions}
\underline{\psi}_{i} = \begin{pmatrix} \delta_{i1} & \delta_{i2} \\ \delta_{i3} & \delta_{i4} \end{pmatrix} + 
\begin{pmatrix} \eta \; \delta_{i5} & \eta \; \delta_{i6} \\ -\xi \delta_{i5} \; & -\xi \delta_{i6}  \end{pmatrix} \text{ for } i = 1, \cdots, 6 
\text{ and } \delta_{ij} \in \mathbb{R} \text{ for } i, j = 1, 2, \cdots 6.
\end{align}

Taking $\psi_{i}$ as the test function for row $i$, the resulting matrix representation of equation
 \eqref{eq:BDM2_interpolation_operator_element_1} is presented in \eqref{eq:interior_projection_mat} where 
 $I(\cdot)$ is defined in \eqref{eq:integral_short_hand}.
 
To illustrate how the elements of \eqref{eq:interior_projection_mat} are calculated, we consider the first row 
of \eqref{eq:interior_projection_mat}.  From \eqref{eq:integral_short_hand}, Lemma \ref{lem:integral_eval} and 
\eqref{eq:bdm_integrals}, the entries of the first row are
\begin{align*}
I(g_{2}(1-\xi-\eta) \xi) &= \int_{\widehat{T}} g_{2}(1-\xi-\eta) \; \xi \; (r_{1}^{*} \xi + r_{2}^{*} \eta + 1) \; d \widehat{T} \\
&= g_{2} \left [ \int_{\widehat{T}} (\xi - \xi^{2} - \eta \xi ) \; (r_{1}^{*} \xi + r_{2}^{*} \eta + 1) \; d \widehat{T}  \right] \\
&= g_{2} \left( [ \dfrac{1}{4!} (2 r^{*}_{1} + r^{*}_{2} + 4) - \dfrac{2}{5!}(3 r_{1}^{*} 
+ r_{2}^{*} + 5) - \dfrac{1}{5!} (2 r_{1}^{*} + 2 r_{2}^{*} + 5) \right)\\
&= g_{2} \left [ \dfrac{1}{5!} (2r_{1}^{*} + r_{2}^{*} + 5) \right] \\
I(\xi (g_{2} \xi + \eta - g_{2}) ) &= \int_{\widehat{T}} (g_{2} \xi^{2} + \xi \eta - g_{2} \xi ) (r_{1}^{*} \xi + r_{2}^{*} \eta + 1) \; d \widehat{T} \\
&= \left( \dfrac{2 g_{2}}{5!} (3 r_{1}^{*} + r_{2}^{*} + 5) + \dfrac{1}{5!} (2r_{1}^{*} + 2 r_{2}^{*} + 5) 
- \dfrac{g_{2}}{4!} (2 r_{1}^{*} + r_{2}^{*} + 4)  \right ) \\
&= 2 ( 1 - 2 g_{2}) r^{*}_{1} + (2 - 3 g_{2}) r_{2}^{*} + 5(1- 2g_{2})\\
I((g_{2}-1) \eta \xi) &= \int_{\widehat{T}} (g_{2}-1) \eta \xi  (r_{1}^{*} \xi + r_{2}^{*} \eta + 1) \; d \widehat{T} \\
&= (g_{2} - 1) \left[ \dfrac{1}{5!} (2 r_{1}^{*} + 2 r_{2}^{*} + 5) \right]
\end{align*}
with the remaining columns equaling zero.  A similar procedure can be used to find the remaining terms in the system.  
The complete entires of the matrix expressed in terms of the coordinates of the triangle $T$ are shown 
in \eqref{eq:interior_projection_mat_eval} which we denote $M_{T}$.
 
Taking the determinate of \eqref{eq:interior_projection_mat_eval} yields 
\begin{align*}
| M_{T} | &= \dfrac{1}{36} (r^{*}_{1} + r^{*}_{2} + 3)(2 r^{*}_{1} + r^{*}_{2} + 5) ( r^{*}_{1} + 2 r^{*}_{2} + 5) \\
& \quad ( 2 r^{*}_{1} + 2 r^{*}_{2} + 5) ( (r^{*}_{1})^{2} + 4 r^{*}_{1} r^{*}_{2} + (r^{*}_{2})^{2} + 10 r^{*}_{1}  + 10 r^{*}_{2} + 15).
\end{align*}


Since $r^{*}_{1}, r^{*}_{2} \geq 0$, it follows that $|M_{T} | > 0$ implying that the matrix representation of the projection 
operator is full rank.  Therefore $\Pi_{h} \tautens = \{\mbf{0} \}$ is the unique solution. Hence $\Pi_{h}$ is well defined.

Finally, we verify that the spaces given in \eqref{eq:bdm2_discrete_function_spaces} satisfy 
\eqref{eq:rz_projection_1}-\eqref{eq:rz_projection_3}.  Observe that for $U_{h} = (P_{1})^{2}$ and $Q_{h} = P_{1}$ the 
test space of \eqref{eq:rz_projection_1} is the set
\begin{align*}
\left \{ \begin{pmatrix} \delta_{1} + z \delta_{5} & \delta_{2} + z \delta_{6}\\ \delta_{3} -r \delta_{5} & \delta_{4} -r \delta_{6} \end{pmatrix}
 | \quad \forall \; \delta_{1},\delta_{2},\delta_{3},\delta_{4},\delta_{5},\delta_{6} \in \mathbb{R} \right \},
\end{align*}
which is the same as the test space described in \eqref{eq:BDM2_interpolation_operator_element_1}.  Furthermore, 
Theorem \ref{lem:thm_41} requires that \eqref{eq:rz_projection_2} is satisfied on a subset of
\begin{align*}
\left \{ \begin{pmatrix} s_{1} + s_{2} s + s_{3} s^{2} \\ s_{4} + s_{5} s + s_{6} s^{2} \end{pmatrix} |
 \quad \forall \; s_{1}, s_{2}, s_{3}, s_{4}, s_{5}, s_{6} \in \mathbb{R} \right \} 
\end{align*}
for all $\ell$.  Since the boundary integral \eqref{eq:BDM2_interpolation_operator_edge_1} is satisfied for all 
quadratic polynomials on all $\ell$, this condition is also satisfied.  Lastly, for $U_{h}(T)  =  \left( P_{1}(T) \right)^{2}$, 
$Q_{h}(T) = P_{1}(T)$, the test functions in \eqref{eq:rz_projection_3} are a subset of the test functions
in  \eqref{eq:BDM2_interpolation_operator_p2}.
\end{proof}

\afterpage{%
    \clearpage
    \thispagestyle{empty}
    \begin{landscape}
        \centering 
{\scriptsize
\begin{align}
\label{eq:interior_projection_mat}
\begin{pmatrix}
I(g_{2} (1 - \xi - \eta)\; \xi)& I ( \xi (g_{2} \xi + \eta - g_{2}))& I((g_{2} -1) \eta \xi )&0&0&0 \\
I((g_{2}-1)(1 - \xi - \eta)\; \eta )& I( (g_{2}-1) \eta \xi )& I( \eta (\xi + g_{2} \eta - g_{2})&0&0&0 \\
0&0&0&I(g_{2} (1 - \xi - \eta)\; \xi)& I ( \xi (g_{2} \xi + \eta - g_{2}))& I((g_{2} -1) \eta \xi ) \\
0&0&0&I((g_{2}-1)(1 - \xi - \eta)\; \eta )& I( (g_{2}-1) \eta \xi )& I( \eta (\xi + g_{2} \eta - g_{2})) \\
I((1 - \xi - \eta) g_{2} \xi \eta )& I(\xi (g_{2} \xi + \eta - g_{2}) \eta)& I(\eta (g_{2} - 1) \xi \eta)&
I(-(1-\xi-\eta) g_{2} \xi^{2} )& I(- \xi (g_{2} \xi + \eta - g_{2}) \xi ))& I(- \eta (g_{2} - 1) \xi^{2} ) \\
I((1 - \xi - \eta) (g_{2} - 1) \eta^{2} )& I(\xi (g_{2} - 1) \eta^{2})& I(\eta (\xi + g_{2} \eta - g_{2} ) \eta)&
I(-(1 - \xi - \eta) (g_{2}-1) \eta \xi)& I(- \xi (g_{2} - 1) \eta \xi )& I(- \eta (\xi + g_{2} \eta - g_{2} ) \xi) \\
\end{pmatrix}
\end{align}
\vfill
\scriptsize
\begin{multline}
\label{eq:interior_projection_mat_eval}
\left(
\begin{matrix}
g_{2}(2 r^{*}_{1} + r^{*}_{2} + 5) &  2(1-2g_{2})r^{*}_{1} + (2 - 3g_{2}) r^{*}_{2} + 5(1-2g_{2}) & (g_{2}-1) (2 r^{*}_{1} + 2 r^{*}_{2} + 5) \\
(g_{2}-1) (r^{*}_{1} + 2 r^{*}_{2} + 5) & (g_{2}-1) (2 r^{*}_{1} + 2 r^{*}_{2} + 5) &  ((2-3g_{2})r^{*}_{1} + 2(1-2g_{2})r^{*}_{2} + 5(1-2g_{2}) \\
0&0&0 \\
0&0&0 \\
g_{2}(r^{*}_{1} + r^{*}_{2} +3) & (2-3g_{2})r^{*}_{1} + (3-4 g_{2})r^{*}_{2} + 3(2-3g_{2}) & (g_{2}-1) (2r^{*}_{1} + 3 r^{*}_{2} + 6) \\
(g_{2}-1)(r^{*}_{1} + 3 r^{*}_{2} + 6) & (g_{2}-1) (2r^{*}_{1} + 3 r^{*}_{2} + 6) & (2-3g_{2}) r^{*}_{1} + 3(1-2g_{2}) r^{*}_{2} + 6(1-2g_{2})
\end{matrix} \right.
\\
\left.
\begin{matrix}
0 &  0 & 0 \\
0 &  0 & 0 \\
g_{2}(2 r^{*}_{1} + r^{*}_{2} + 5) &  2(1-2g_{2})r^{*}_{1} + (2 - 3g_{2}) r^{*}_{2} + 5(1-2g_{2}) & (g_{2}-1) (2 r^{*}_{1} + 2 r^{*}_{2} + 5) \\
(g_{2}-1) (r^{*}_{1} + 2 r^{*}_{2} + 5) & (g_{2}-1) (2 r^{*}_{1} + 2 r^{*}_{2} + 5) &  (2-3g_{2})r^{*}_{1} + 2(1-2g_{2})r^{*}_{2} + 5(1-2g_{2}) \\
-g_{2}(3 r^{*}_{1} + r^{*}_{2} + 6) & -[(3-6g_{2})r^{*}_{1} + (2 - 3g_{2})r^{*}_{2} + 6(1-2g_{2})] & -(g_{2}-1) (3 r^{*}_{1} + 2 r^{*}_{2} + 6) \\
-(g_{2}-1) (r^{*}_{1}+r^{*}_{2}+3) & -(g_{2}-1)(3 r^{*}_{1} + 2 r^{*}_{2} + 6) & -(3-4g_{2})r^{*}_{1} + (2-3g_{2})r^{*}_{2} + 3(2-3g_{2})
\end{matrix} \right)
\end{multline}
\vfill
}
{
}
    \end{landscape}
    \clearpage
}


\newpage


 \setcounter{equation}{0}
\setcounter{figure}{0}
\setcounter{table}{0}
\setcounter{theorem}{0}
\setcounter{lemma}{0}
\setcounter{corollary}{0}
\setcounter{definition}{0}
\section{Error Analysis}
\label{sec:convergence_analysis}
In this section, for $\bSigma_{h} \times U_{h} \times Q_{h}$ satisfying the inf-sup condition
\begin{align}
\label{eq:convergence_inf_sup}
\inf_{\mbf{w}_{h} \in U_{h}, p_{h} \in Q_{h}} \sup_{(\sigtens_{h}, \sigma_{h}) \in \bSigma_{h}} 
\dfrac{b((\sigtens_{h}, \sigma_{h}), \mbf{w}_{h}) + c((\sigtens_{h}, \sigma_{h}), p_{h})}{ (\| (\sigtens_{h}, \sigma_{h})
 \|_{\bSigma}) (\| \mbf{w}_{h} \|_{U} + \|p_{h}\|_{Q}) } \geq \beta > 0,
\end{align} 
we present an error analysis for the solution to the discrete linear elasticity problem 
\eqref{eq:axi_saddle_point_1_mod_discrete}-\eqref{eq:axi_saddle_point_3_mod_discrete}.  
For notational compactness, we let
\begin{align}
\label{eq:b_operator_convergence_analysis}
B((\sigtens_{h}, \sigma_{h}), (\mbf{v}_{h}, p_{h} ) ) = b((\sigtens_{h}, \sigma_{h}), \mbf{v}_{h}) + c((\sigtens_{h}, \sigma_{h}), p_{h}).
\end{align}

Recall that operator $a(\cdot, \cdot): \bSigma_{h} \times \bSigma_{h} \rightarrow \mathbb{R}$ as defined in 
\eqref{eq:axi_a_grad_div} is continuous and coercive (see Lemmas \ref{lem:a_bounded} and \ref{lem:a_coercive}).  That is,
\begin{align}
\label{eq:a_coercivity}
a((\sigtens,\sigma),(\sigtens,\sigma)) &= \| (\sigtens, \sigma) \|^{2}_{\bSigma} \geq \gamma > 0 
\text{ for all } (\sigtens,\sigma) \in \bSigma_{h}, \\
\label{eq:a_continuous}
a((\sigtens, \sigma), (\tautens, \tau)) &\leq \alpha \| (\sigtens, \sigma) \|_{\bSigma} \| (\tautens, \tau) \|_{\bSigma} 
\end{align}
for some $\alpha > 0$ and all $(\sigtens, \sigma), (\tautens, \tau) \in \bSigma$.  We also note that $B((\cdot, \cdot), (\cdot, \cdot) )$ 
is continuous since
\begin{align}
\label{eq:b_continuous}
\begin{split}
B((\sigtens, \sigma), (\mbf{v}, q)) &= b( (\sigtens, \sigma), \mbf{v}) + c( (\sigtens, \sigma), q ) \\
&= (\mbf{v}, \axidiv \sigtens) - (v_{r}, \dfrac{\sigma}{r}) +  (\sigtens, \mathcal{S}^{2}(q)) + (\axidiv (\sigtens, \sigma) \wedge \mbf{x}, q)   \\
&\leq C_{1} \| \mbf{v} \|_{{}_{1}\mbf{L}^{2}(\Omega)} \| (\sigtens, \sigma) \|_{\bSigma} 
+ C_{2}  \| q \|_{{}_{1}L^{2}(\Omega)} \| (\sigtens, \sigma) \|_{\bSigma} \\ 
&\leq \beta \| (\sigtens, \sigma) \|_{\bSigma} ( \| \mbf{v} \|_{U} + \| q \|_{Q})
\end{split}
\end{align}
for all $(\sigtens, \sigma) \in \bSigma$, $\mbf{v} \in U$ and $q \in Q$ where $C_{1}, C_{2}, \beta > 0$.

The discrete null space of the operator $B((\cdot, \cdot), (\cdot, \cdot))$ is defined as
\begin{align}
\label{eq:zh_space}
Z_{h} = \{ (\tautens_{h}, \tau_{h}) \in \bSigma_{h} : B((\tautens_{h}, \tau_{h}) , (\mbf{v}_{h}, q_{h}) = 0 
\text{ for all } \mbf{v}_{h} \in U_{h} \text{ and } q_{h} \in Q_{h} \}.
\end{align}

Since $B((\tautens_{h}, \tau_{h}), (\mbf{v}_{h}, q_{h})) = 0$ only holds on the discrete subspaces 
$U_{h}$ and $Q_{h}$, $Z_{h} \not \subset Z$.  This observation motivates the following theorem which bounds 
the error $\sigtens_{h}$ in terms of the spaces $U_{h}$, $Q_{h}$ and $Z_{h}$. 
\begin{theorem}
\label{lem:bound_num_one}
Let $((\sigtens, \sigma), \mbf{w}, p)$ solve \eqref{eq:axi_saddle_point_1_mod}-\eqref{eq:axi_saddle_point_3_mod} 
and $(\sigtens_{h}, \sigma_{h})$ solve \eqref{eq:axi_saddle_point_1_mod_discrete}-\eqref{eq:axi_saddle_point_3_mod_discrete}.  
If $\bSigma_{h} \subset \bSigma$, $U_{h} \subset U$, $Q_{h} \subset Q$, and $Z_{h}$ is defined as in \eqref{eq:zh_space}, then
\begin{align*}
\begin{split}
\| (\sigtens - \sigtens_{h} , \sigma - \sigma_{h} ) \|_{\bSigma} &\leq 
C \Big( \inf_{(\tautens_{h}, \tau_{h}) \in \bSigma_{h}} \|  (\sigtens - \tautens_{h}, \sigma - \tau_{h}) \|_{\bSigma} \\
& \quad \quad \quad + \inf_{\mbf{v}_{h} \in U_{h} } \| \mbf{w} - \mbf{v}_{h} \|_{U} + \inf_{q_{h} \in Q_{h}} \|p - q_{h} \|_{Q} \Big),
\end{split}
\end{align*}
where $C >0$, is independent of $h$.
\end{theorem}
\begin{proof}
Let $(\sigtens_{h}, \sigma_{h}) \in Z_{h}$ be the unique solution to
\begin{align}
\label{eq:a_over_zh}
a((\sigtens_{h}, \sigma_{h}), (\tautens_{h}, \tau_{h}) ) = (\mbf{f}, \axidiv (\tautens_{h}, \tau_{h})) 
\text{ for all } (\tautens_{h}, \tau_{h}) \in Z_{h},
\end{align}
as ensured by the Lax-Milgram Theorem (provided that $\mbf{f}$ lives in the dual space of 
$_{1} \mbf{H} (\taxidiv, \Omega ;  \mathbb{R}^{2})$).  

To develop an error bound, for $(\sigtens_{h}, \sigma_{h})$, 
we must compare it with the true solution $(\sigtens, \sigma)$.  Noting again that $Z_{h} \not \subset Z$, from
 \eqref{eq:axi_saddle_point_1_mod}-\eqref{eq:axi_saddle_point_3_mod} the true solution $((\sigtens,\sigma),\mbf{w},p)$ 
 satisfies
\begin{align}
\label{eq:a_tilde_1}
a((\sigtens,\sigma),(\xitens_{h}, \xi_{h})) = (\mbf{f}, \axidiv(\xitens_{h}, \xi_{h}) ) - B((\xitens_{h}, \xi_{h}), (\mbf{w}, p)) 
\text{ for all } (\xitens_{h}, \xi) \in \bSigma_{h}.
\end{align}
Subtracting \eqref{eq:a_over_zh} from \eqref{eq:a_tilde_1}
\begin{align*}
a ((\sigtens - \sigtens_{h}, \sigma-\sigma_{h}), (\xitens_{h}, \xi_{h})) = - B((\xitens_{h}, \xi_{h}), (\mbf{w}, p)) 
\text{ for all } (\xitens_{h}, \xi_{h}) \in Z_{h}.
\end{align*}
From \eqref{eq:zh_space} it then follows that for all $(\xitens_{h}, \xi_{h}) \in Z_{h}, \mbf{v}_{h} \in U_{h}, q_{h} \in Q_{h}$
\begin{align}
\label{eq:a_operator_and_b_operator}
a((\sigtens-\sigtens_{h}, \sigma - \sigma_{h}), (\xitens_{h}, \xi_{h})) = 
- B ((\xitens_{h}, \xi_{h}), (\mbf{w}, p)) + B((\xitens_{h}, \xi_{h}), (\mbf{v}_{h}, q_{h})).
\end{align}
Next, adding and subtracting $(\tautens_{h}, \tau_{h}) \in Z_{h}$ in $a(\cdot, \cdot)$, \eqref{eq:a_operator_and_b_operator} becomes
\begin{align*}
\begin{split}
a((\tautens_{h} - \sigtens_{h},\tau_{h} - \sigma_{h}),(\xitens_{h}, \xi_{h})) = 
&- a((\sigtens - \tautens_{h}, \sigma - \tau_{h}), (\xitens_{h}, \xi_{h})) \\
& \quad \quad \quad \quad - B((\xitens_{h}, \xi_{h}), (\mbf{w} - \mbf{v}_{h}, p - q_{h})).
\end{split}
\end{align*}
Choosing $(\xitens_{h}, \xi_{h}) = (\tautens_{h} - \sigtens_{h}, \tau_{h} - \sigma_{h}) \in Z_{h}$, and using the 
continuity and coercivity of $a(\cdot, \cdot)$ (described in \eqref{eq:a_continuous}, \eqref{eq:a_coercivity}) and the 
continuity of $B((\cdot, \cdot),(\cdot,\cdot))$ (described in \eqref{eq:b_continuous}) we obtain
\begin{align*}
\begin{split}
0 &< \; \gamma \| ( \tautens_{h} - \sigtens_{h}, \tau_{h} - \sigma_{h}) \|^{2}_{\bSigma} \\
 &\leq \alpha \; \| ( \tautens_{h} - \sigtens_{h}, \tau_{h} - \sigma_{h}) \|_{\bSigma}  
    \|(\sigtens - \tautens_{h}, \sigma - \tau_{h}) \|_{\bSigma}  \\ 
& \quad + \beta \; \| ( \tautens_{h} - \sigtens_{h}, \tau_{h} - \sigma_{h}) \|_{\bSigma} 
\left( \| \mbf{w} - \mbf{v}_{h} \|_{U} + \| p - q_{h} \|_{Q}  \right).
\end{split}
\end{align*}
Dividing through by $\gamma \| (\tautens_{h} - \sigtens_{h}, \tau_{h} - \sigma_{h}) \|_{\bSigma}$ gives
\begin{align}
\label{eq:zero_inf_inequality}
\| ( \tautens_{h} - \sigtens_{h}, \tau_{h} - \sigma_{h}) \|_{\bSigma} \leq \dfrac{\alpha}{\gamma} 
\|(\sigtens - \tautens_{h}, \sigma - \tau_{h}) \|_{\bSigma} + \dfrac{\beta}{\gamma} 
\left( \| \mbf{w} - \mbf{v}_{h} \|_{U} + \| p - q_{h} \|_{Q} \right).
\end{align}
Next, applying the triangle inequality, for an arbitrary element $(\tautens_{h}, \tau_{h}) \in \bSigma_{h}$,
\begin{align}
\label{eq:first_inf_inequality}
\|(\sigtens - \sigtens_{h}, \sigma - \sigma_{h}) \|_{\bSigma} &\leq \| (\sigtens - \tautens_{h}, \sigma - \tau_{h}) \|_{\bSigma} 
+ \| (\tautens_{h} - \sigtens_{h}, \tau_{h} - \sigma_{h} ) \|_{\bSigma}.
\end{align}

Since $(\tautens_{h}, \tau_{h}) \in \bSigma_{h}$, $\mbf{v}_{h} \in U_{h}$ and $q_{h} \in Q_{h}$ are arbitrary, 
combining \eqref{eq:zero_inf_inequality} and \eqref{eq:first_inf_inequality} we get
\begin{align}
\label{eq:second_inf_inequality}
\begin{split}
\| (\sigtens - \sigtens_{h}, \sigma - \sigma_{h}) \|_{\bSigma} &\leq (1 + \dfrac{\alpha}{\gamma})
 \inf_{(\tautens_{h}, \tau_{h}) \in Z_{h}} \| (\sigtens - \tautens_{h}, \sigma - \tau_{h}) \|_{\bSigma} \\
& \quad + \dfrac{\beta}{\gamma}
 \left( \inf_{\mbf{v}_{h} \in U_{h}} \| \mbf{w} - \mbf{v}_{h} \|_{U} +  \inf_{q_{h} \in Q_{h}} \| p - q_{h} \|_{Q} \right).
\end{split}
\end{align}

In order to lift the approximation of $(\sigtens - \tautens_{h}, \sigma - \tau_{h})$ from the infinimum over $Z_{h}$ 
to the infinimum over $\bSigma_{h}$, we use the inf-sup condition \eqref{eq:convergence_inf_sup}.  
A equivalent property to the spaces $\bSigma_{h} \times U_{h} \times W_{h}$ satisfying \eqref{eq:convergence_inf_sup} 
is the existence of a projection $\Pi_{h} : \bSigma \rightarrow \bSigma_{h} $ satisfying
\begin{align*}
B(((\tautens, \tau) - \Pi_{h}(\tautens,\tau)), (\mbf{v}_{h}, q_{h}) ) &= 0 \text{ for all } (\mbf{v}_{h}, q_{h}) \in U_{h} \times Q_{h} \\
\mbox{and } \ \ 
 \| \Pi_{h}(\tautens, \tau) \|_{\bSigma} &\leq C_{\Pi} \| (\tautens, \tau) \|_{\bSigma},
\end{align*}
where $C_{\Pi} >0$ is a constant that is independent of $h$.

Let $(\xitens_{h}, \xi_{h}) \in \bSigma_{h}$, and introduce $(\rhotens_{h}, \rho_{h}) \in \bSigma_{h}$ satisfying
\begin{align*}
(\rhotens_{h}, \rho_{h}) = \Pi_{h} (\sigtens -  \xitens_{h}, \sigma - \xi_{h}) \text{ where } \| (\rhotens_{h}, \rho_{h}) \|_{\bSigma} 
\leq C_{\Pi} \| (\sigtens - \xitens_{h}, \sigma - \xi_{h}) \|_{\bSigma}.
\end{align*}
Taking $(\tautens_{h}, \tau_{h}) = (\xitens_{h} + \rhotens_{h}, \xi_{h} + \rho_{h})$
\begin{align*}
\begin{split}
&B((\tautens_{h}, \tau_{h}), (\mbf{w}_{h} , q_{h})) = B((\xitens_{h},\xi_{h}), (\mbf{v}_{h}, q_{h})) 
+ B((\rhotens_{h},\rho_{h}), (\mbf{v}_{h}, q_{h})) \\
&\quad = B((\xitens_{h},\xi_{h}), (\mbf{v}_{h},q_{h})) 
+ B((\sigtens,\sigma), (\mbf{v}_{h}, q_{h})) - B((\xitens_{h},\xi_{h}), (\mbf{v}_{h}, q_{h})) \\
&\quad =B((\sigtens,\sigma), (\mbf{v}_{h}, q_{h})) = 0,
\end{split}
\end{align*}
which implies that $(\tautens_{h}, \tau_{h}) \in Z_{h}$.  

 Next, using $(\tautens_{h}, \tau_{h}) = (\xitens_{h} + \rhotens_{h}, \xi_{h} + \rho_{h})$
\begin{align*}
\begin{split}
\| (\sigtens - \tautens_{h}, \sigma - \tau_{h}) \|_{\bSigma} &\leq \| (\sigtens - \xitens_{h}, \sigma - \xi_{h}) \|_{\bSigma} 
+ \| (\rhotens_{h}, \rho_{h}) \|_{\bSigma} 
\leq (1 + C_{\Pi}) \| (\sigtens - \xitens_{h}, \sigma - \xi_{h}) \|_{\bSigma} \, .
\end{split}
\end{align*}
Finally, taking infima over the appropriate spaces on the left and right sides gives the result
\begin{align}
\label{eq:lifting_arg}
\inf_{(\tautens_{h}, \tau_{h}) \in Z_{h}} \| (\sigtens - \tautens_{h}, \sigma - \tau_{h}) \|_{\bSigma} \leq 
(1 + C_{\Pi}) \inf_{(\xitens_{h}, \xi_{h}) \in \bSigma_{h} } \| (\sigtens - \xitens_{h}, \sigma - \xi_{h}) \|_{\bSigma}.
\end{align}
 
Combining \eqref{eq:second_inf_inequality} and \eqref{eq:lifting_arg} we obtain
\begin{align}
\label{eq:tau_bound_1}
\begin{split}
\| (\sigtens - \sigtens_{h}, \sigma- \sigma_{h}) \|_{\bSigma} 
&\leq C  ( \inf_{\tautens_{h}, \tau_{h} \in \bSigma_{h}} \| (\sigtens - \tautens_{h}, \sigma - \tau_{h}) \|_{\bSigma} \\
&\quad \quad \quad + \inf_{\mbf{v}_{h} \in U_{h} } \| \mbf{w} - \mbf{v}_{h} \|_{U} + \inf_{q_{h} \in Q_{h}} \|p - q_{h} \|_{Q} ).
\end{split}
\end{align}
\end{proof}

With error bounds for the stress space established, the following theorem establishes error bounds for the displacement 
and skew-symmetry approximations.  
\begin{theorem}
\label{lem:displacement_bound}
For $((\sigtens, \sigma), \mbf{w}, p)$ satisfying \eqref{eq:axi_saddle_point_1_mod}-\eqref{eq:axi_saddle_point_3_mod} 
and $((\sigtens_{h}, \sigma_{h}), \mbf{w}_{h}, p_{h})$ 
satisfying \eqref{eq:axi_saddle_point_1_mod_discrete}-\eqref{eq:axi_saddle_point_3_mod_discrete} 
there exists $C > 0$, independent of $h$, such that 
\begin{align}
\label{eq:displacment_bound_eq}
\begin{split}
&\|\mbf{w} - \mbf{w}_{h} \|_{U} + \| p - p_{h} \|_{Q} \\
&\leq C \left( \inf_{\tautens_{h}, \tau_{h} \in \bSigma_{h}} \| (\sigtens - \tautens_{h}, \sigma - \tau_{h}) \|_{\bSigma} 
+ \inf_{\mbf{v}_{h} \in U_{h}} \| \mbf{w} - \mbf{v}_{h} \|_{U} + \inf_{q_{h} \in Q_{h}} \| p-q_{h} \|_{Q} \right).
\end{split}
\end{align}
\end{theorem}
\begin{proof}
Subtracting equations \eqref{eq:axi_saddle_point_1_mod_discrete} from \eqref{eq:axi_saddle_point_1_mod} gives
\begin{align*}
B((\xitens_{h}, \xi_{h}), (\mbf{w} - \mbf{w}_{h}, p - p_{h} )) = - a( (\sigtens - \sigtens_{h}, \sigma - \sigma_{h}), (\xitens_{h}, \xi_{h})) 
\end{align*}
for all $(\xitens_{h}, \xi_{h}) \in \bSigma_{h}$.

For any $\mbf{v}_{h} \in U_{h}$ and $q_{h} \in Q_{h}$, the inf-sup condition \eqref{eq:convergence_inf_sup} gives
\begin{align}
\label{eq:inf_sup_bound}
\begin{split}
&\beta \left ( \| \mbf{w}_{h} - \mbf{v}_{h} \|_{U} + \| p_{h} - q_{h} \|_{Q} \right) \leq 
\sup_{(\xitens_{h}, \xi_{h}) \in \bSigma_{h}} 
\dfrac{| B((\xitens_{h}, \xi_{h}), (\mbf{w}_{h}-\mbf{v}_{h}, p_{h} - q_{h})) |}{\| (\xitens_{h}, \xi_{h}) \|_{\bSigma}} \\
&\leq \sup_{(\xitens_{h}, \xi_{h}) \in \bSigma_{h}} 
\left( \dfrac{| B((\xitens_{h}, \xi_{h}), (\mbf{w_{h}}-\mbf{w}, p_{h} - p)) |}{\| (\xitens_{h}, \xi_{h}) \|_{\bSigma}} 
+ \dfrac{| B((\xitens_{h}, \xi_{h}), (\mbf{w}-\mbf{v}_{h}, p - q_{h})) |}{\| (\xitens_{h}, \xi_{h}) \|_{\bSigma}} \right) \\
&\leq \sup_{(\xitens_{h}, \xi_{h}) \in \bSigma_{h}} 
\left(\dfrac{| -a((\sigtens - \sigtens_{h}, \sigma - \sigma_{h}), (\xitens_{h}, \xi_{h})) |}{\| (\xitens_{h}, \xi_{h}) \|_{\bSigma}} 
+ \dfrac{| B((\xitens_{h}, \xi_{h}), (\mbf{w}-\mbf{v}_{h}, p - q_{h})) |}{\| (\xitens_{h} , \xi_{h}) \|_{\bSigma}} \right) \\
&\leq \max \{\alpha, \beta\} ( \| (\sigtens - \sigtens_{h}, \sigma - \sigma_{h}) \|_{\bSigma}  
+ \| \mbf{w} - \mbf{v}_{h} \|_{U} + \|p - q_{h} \|_{Q}),
\end{split}
\end{align}
where in the last step we have used the continuity of $a(\cdot, \cdot)$ and $B(\cdot,\cdot)$.

Combining \eqref{eq:inf_sup_bound} with the triangle inequality gives
\begin{align}
\label{eq:non_inf_u_p_bound}
\begin{split}
 \| \mbf{w} - \mbf{w}_{h} \|_{U} + \| p - p_{h} \|_{Q}  
&\leq \| \mbf{w} - \mbf{v}_{h} \|_{U} + \|\mbf{v}_{h} - \mbf{w}_{h} \|_{U} + \| p - q_{h} \|_{Q} + \| q_{h} - p_{h} \|_{Q} \\
&\leq C ( \| (\sigtens - \sigtens_{h}, \sigma - \sigma_{h}) \|_{\bSigma} + \| \mbf{w} - \mbf{v}_{h} \|_{U} + \| p - q_{h} \|_{Q} ).
\end{split}
\end{align}
As $\mbf{v}_{h} \in U_{h}$ and $q_{h} \in Q_{h}$ are arbitrary, \eqref{eq:displacment_bound_eq} follows from \eqref{eq:non_inf_u_p_bound} and \eqref{eq:tau_bound_1}.
\end{proof}

Combining Theorems \ref{lem:bound_num_one} and \ref{lem:displacement_bound} we have the following.
\begin{corollary}
\label{thm:main_convergence_theorem}
Let $((\sigtens, \sigma), \mbf{w}, p) \in \bSigma \times U \times Q$ be the solution of 
\eqref{eq:axi_saddle_point_1_mod}-\eqref{eq:axi_saddle_point_3_mod} and 
$((\sigtens_{h}, \sigma_{h}), \mbf{w}_{h}, p_{h}) \in \bSigma_{h} \times U_{h} \times Q_{h}$ the solution 
of \eqref{eq:axi_saddle_point_1_mod_discrete}-\eqref{eq:axi_saddle_point_3_mod_discrete}, then  
\begin{align*}
\begin{split}
&\| (\sigtens - \sigtens_{h}, \sigma - \sigma_{h} )\|_{\bSigma} + \| \mbf{w} - \mbf{w}_{h} \|_{U} + \| p - p_{h} \|_{Q} \\
& \quad \quad \leq C  ( \inf_{(\tautens_{h}, \tau_{h}) \in \bSigma_{h} } \| (\sigtens - \tautens_{h}, \sigma - \tau_{h}) \|_{\bSigma}  + \inf_{\mbf{v}_{h} \in U_{h}} \| \mbf{w} - \mbf{v}_{h} \|_{U} + \inf_{q_{h} \in Q_{h}} \| p - q_{h} \|_{Q} ).
\end{split}
\end{align*}
\end{corollary}
\mbox{ } \hfill \qed

Using Corollary \ref{thm:main_convergence_theorem}, and additional smoothness assumptions, we can now form an 
error bound in terms of the mesh parameter $h$.  First observe that for the axisymmetric $\textbf{BDM}_{k}$ 
interpolation operator $\tilde{\rho}_{h}: {}_{1}\mbf{H}^{1}(\Omega) \rightarrow \textbf{BDM}_{k}(\mathcal{T}_{h})$ as 
defined in \cite{Ervin1}, if $\mbf{u} \in {}_{1}\mbf{H}^{k+1}(\Omega)$, then for some $C>0$,
\begin{align}
\label{eq:axi_bdm_h_bound_1}
\| \mbf{u} - \tilde{\rho}_{h}(\mbf{u}) \|_{{}_{1}L^{2}(\Omega)} \leq C \; h^{k+1} | \mbf{u} |_{{}_{1}\mbf{H}^{k+1}(\Omega)}.
\end{align}
In addition, if $\axidiv \mbf{u} \in {}_{1}H^{k}(\Omega)$ where 
$\left( \Sigma_{T \in \mathcal{T}_{h}} | \axidiv \tilde{\rho}_{h} (\mbf{u}) |^{2}_{{}_{1}H^{k+1}(T)} \right)^{2} 
< C_{1}$, then for some $C > 0$,
\begin{align}
\label{eq:axi_bdm_h_bound_2}
\| \axidiv \mbf{u} - \axidiv \tilde{\rho}_{h} (\mbf{u}) \|_{{}_{1}L^{2}(\Omega)} \leq C h^{k}.
\end{align}
Combining the results and assumptions of \eqref{eq:axi_bdm_h_bound_1} and 
\eqref{eq:axi_bdm_h_bound_2}, if $\mbf{u} \in {}_{1}\mbf{H}^{k+1}(\Omega)$ and $\axidiv \mbf{u} \in {}_{1}H^{k}(\Omega)$ 
where $\left( \Sigma_{T \in \mathcal{T}_{h}} | \axidiv \tilde{\rho}_{h} (\mbf{u}) |^{2}_{{}_{1}H^{k+1}(T)} \right) < C_{1}$, then 
there exists $C>0$ such that
\begin{align*}
\| \mbf{u} - \tilde{\rho}_{h} \; \mbf{u} \|_{{}_{1}\mbf{H}(\text{div},\Omega)} \leq C \; h^{k}.
\end{align*}
Under analogous assumptions, this result can be extended to the tensor case, 
where $\tilde{\pmb{\rho}}_{h} : {}_{1}\Htens^{1}(\Omega) \rightarrow (\textbf{BDM}_{k}(\mathcal{T}_{h}))^{2}$ 
represents the $\mbf{BDM}_{k}$ interpolation operator applied to the rows of a tensor so that
\begin{align}
\label{eq:bdm_tensor_interpolation_bound}
\| \sigtens - \tilde{\pmb{\rho}}_{h} \sigtens \|_{{}_{1}\Htens(\text{div},\Omega)} \leq C \; h^{k}.
\end{align}
Next we present a result from \cite{BelhachmiEtAl} which bounds the Cl\'{e}ment operator $\Lambda^{k}_{h}$.  
The Cl\'{e}ment operator $\Lambda^{k}_{h}$ maps ${}_{1}L^{2}(\Omega)$ into the space of degree $k$ Lagrangian 
finite elements on the mesh $\mathcal{T}_{h}$.  Indeed, as stated in Corollary 2 of Theorem 1 
in \cite{BelhachmiEtAl}, for $v \in {}_{1}H^{k+1}(\Omega)$, there exists a $C$ independent of $h$ such that 
\begin{align}
\label{eq:pressure_type_bound}
\| v - \Lambda^{k}_{h} v \|_{{}_{1}L^{2}(\Omega)} \leq C h^{k+1} | v |_{{}_{1}H^{k+1}(\Omega)}.
\end{align}
As with the \textbf{BDM} interpolation $\tilde{\rho}_{h}$, the bound for $\Lambda^{k}_{h}$ can be 
extended to vector and tensor functions.\\
The following corollary gives the error bound in terms of the mesh parameter $h$.  

\begin{corollary}
\label{cor:interpolation_error_bound_k12}
Assume that $\bPi_{h}$ of Lemma \ref{lem:bdm1_projection} or \ref{lem:BDM2_projection_stable} satisfies
 \eqref{eq:projection_bound_1}-\eqref{eq:projection_bound_2}.  
If $(\sigtens, \sigma, \mbf{w}, p) \in {}_{1}\Htens^{k}(\Omega) \times {}_{1}L^{2}(\Omega)  
\times {}_{1}\mbf{H}^{k}(\Omega) \times {}_{1}H^{k}(\Omega)$ solves 
\eqref{eq:axi_saddle_point_1_mod}-\eqref{eq:axi_saddle_point_3_mod} and 
$(\sigtens_{h}, \sigma_{h}, \mbf{w}_{h}, p_{h}) \in (\mbf{BDM}_{k})^{2}(\mathcal{T}_{h}) \times P_{k}(\mathcal{T}_{h}) 
\times (P_{k-1}(\mathcal{T}_{h}))^{2} \times P_{k-1}(\mathcal{T}_{h})$ 
solves \eqref{eq:axi_saddle_point_1_mod_discrete}-\eqref{eq:axi_saddle_point_3_mod_discrete} for $k = 1,2$, then  
\begin{align}
\label{eq:interp_error_bound}
\begin{split}
&\| (\sigtens - \sigtens_{h}, \sigma - \sigma_{h}) \|_{\bSigma} + \| \mbf{w} - \mbf{w}_{h} \|_{U} + \| p - p_{h} \|_{Q} \leq C \; h^{k}.
\end{split}
\end{align}
\end{corollary}
\begin{proof}
From Corollary \ref{thm:main_convergence_theorem},
\begin{align}
\label{eq:corollary_restate}
\begin{split}
&\| (\sigtens - \sigtens_{h}, \sigma - \sigma_{h}) \|_{\bSigma} + \| \mbf{u} - \mbf{u}_{h} \|_{U} + \| p - p_{h} \|_{Q} \\
&\leq C  \Big( \inf_{(\tautens_{h}, \tau_{h}) \in \bSigma_{h} } 
\| (\sigtens - \tautens_{h}, \sigma - \tau_{h}) \|_{\bSigma}  
+ \inf_{\mbf{v}_{h} \in U_{h}} \| \mbf{u} - \mbf{v_{h}} \|_{U} + \inf_{q_{h} \in Q_{h}} \|p - q_{h} \|_{Q} \Big) \, .
\end{split}
\end{align}

The \textbf{BDM} error bounds from \eqref{eq:bdm_tensor_interpolation_bound}, \eqref{eq:pressure_type_bound} gives 
\begin{align}
\label{eq:h_bound_sig}
\inf_{(\tautens_{h}, \tau_{h}) \in \Sigma_{h} \times S_{h}} \| (\sigtens - \tautens_{h}, \sigma - \tau_{h}) \|_{\bSigma} \leq C \;  h^{k}.
\end{align}

In addition, using a vector generalization of \eqref{eq:pressure_type_bound}
\begin{align}
\label{eq:h_bound_u}
\inf_{\mbf{v}_{h} \in U_{h}} \| \mbf{u} - \mbf{v}_{h} \|_{U} &\leq \| \mbf{u} - \Lambda^{k-1}_{h} \mbf{u} \|_{{}_{1}L^{2}(\Omega)} 
\leq C \; h^{k} | \mbf{u} |_{{}_{1}\mbf{H}^{k}(\Omega)}  \, ,  \\
\mbox{and } \ \ 
\label{eq:h_bound_skew}
 \inf_{q_{h} \in W_{h}} \| \mathcal{S}^{2} (p - q_{h}) \|_{Q} &\leq C_{1} \| p - \Lambda^{k}_{h} \; p \|_{{}_{1}L^{2}(\Omega)} 
 \leq C \; h^{k} |p |_{{}_{1}H^{k}(\Omega)}.
\end{align}

Combining \eqref{eq:corollary_restate}, \eqref{eq:h_bound_sig}, \eqref{eq:h_bound_u} and \eqref{eq:h_bound_skew} gives the result.
\end{proof}

To conclude this section, we establish an error bound for the true displacement $\mbf{u}$.  At this point, error bounds have 
been established in terms of the pseudo displacement variable $\mbf{w}$.  Recall from 
Section \ref{sec:axisymmetric_form}, however, that $\mbf{w} = \mbf{u} - \mbf{x}^{\perp} p$.
\begin{corollary}  \label{err4tru}
Let $((\sigtens, \sigma), \mbf{w}, p) \in \bSigma \times U \times Q$ be the solution 
of \eqref{eq:axi_saddle_point_1_mod}-\eqref{eq:axi_saddle_point_3_mod} and 
$((\sigtens_{h}, \sigma_{h}), \mbf{w}_{h}, p_{h}) \in \bSigma_{h} \times U_{h} \times Q_{h}$ 
the solution of \eqref{eq:axi_saddle_point_1_mod_discrete}-\eqref{eq:axi_saddle_point_3_mod_discrete}.  
Furthermore, let $\mbf{u} = \mbf{w} + \mbf{x}^{\perp}p$ denote the true displacement, and 
$\mbf{u}_{h} = \mbf{w}_{h} + \mbf{x}^{\perp} p_{h}$ denote the discrete approximation to the true displacement.  
There exists a $C>0$ independent of $h$, such that
\begin{align*}
\begin{split}
\| \mbf{u} - \mbf{u}_{h} \|_{U} &\leq C  \Big( \inf_{(\tautens_{h}, \tau_{h}) \in \bSigma_{h}} 
\| (\sigtens - \tautens_{h}, \sigma - \tau_{h}) \|_{\bSigma}  \\ \quad \quad
& \quad \quad \quad  \quad \quad \quad 
       + \inf_{\mbf{v}_{h} \in U_{h}} \| \mbf{w} - \mbf{v_{h}} \|_{U} + \inf_{q_{h} \in Q_{h}} \| p - q_{h} \|_{Q} \Big).
\end{split}
\end{align*}
\end{corollary}
\begin{proof}
For a bounded domain $\Omega$, observe that
\begin{align*}
\| \mbf{x}^{\perp} (p - p_{h}) \|_{U} \leq C_{\mbf{x}^{\perp}} \| p - p_{h} \|_{Q},
\end{align*}
where the constant $C_{\mbf{x}^{\perp}} > 0$ is independent of $h$.  
Therefore, using Theorem \ref{lem:displacement_bound} we have that  
\begin{align*}
\begin{split}
\| \mbf{u} - \mbf{u}_{h} \|_{U} &= \| (\mbf{w} - \mbf{w}_{h}) + \mbf{x}^{\perp} (p - p_{h}) \|_{U} \leq \|\mbf{w} - \mbf{w}_{h} \|_{U} + \| \mbf{x}^{\perp} (p - p_{h}) \|_{U} \\
& \leq C  \Big( \inf_{(\tautens_{h}, \tau_{h}) \in \bSigma_{h}} \| (\sigtens - \tautens_{h}, \sigma - \tau_{h}) \|_{\bSigma}  + \inf_{\mbf{v}_{h} \in U_{h}} \| \mbf{u} - \mbf{v_{h}} \|_{U} \\
&\quad \quad \quad \quad \quad + \inf_{q_{h} \in Q_{h}} \| p - q_{h} \|_{Q} \Big).
\end{split}
\end{align*}
\end{proof}

%
\setcounter{equation}{0}
\setcounter{figure}{0}
\setcounter{table}{0}
\setcounter{theorem}{0}
\setcounter{lemma}{0}
\setcounter{corollary}{0}
\setcounter{definition}{0}
\section{Numerical Experiments}
\label{sec:computational_results}
In this section we present two numerical experiments to investigate our theoretical results. 
For both experiments we consider $\Omega \, = \, (0 , 1) \times (0 , 1)$,
and compute approximations using the approximation elements 
$((\textbf{BDM}_{1}(\mcT_{h}))^{2} \times P_{1}(\mcT_{h})) \times (P_{0}(\mcT_{h}))^{2} \times P_{0}(\mcT_{h})$ 
(shown in Table \ref{fig:elasticity_bdm1_example2} and Table \ref{fig:elasticity_bdm1_exampler2}), and
$((\textbf{BDM}_{2}(\mcT_{h}))^{2} \times P_{2}(\mcT_{h})) 
\times (P_{1}(\mcT_{h}))^{2} \times P_{1}(\mcT_{h})$ 
(shown in Table \ref{fig:elasticity_bdm2_example2} and Table \ref{fig:elasticity_bdm2_exampler2}). 
For both experiments, the value for the grad-div parameter (see \eqref{eq:axi_a_grad_div}) used
was $\gamma = 1$, and the values for the Lam\'{e} constants were $\mu = 1/2$ and $\lambda = 1$. 

\subsubsection*{Experiment 1}
For Experiment 1 the displacement solution was taken to be
\begin{align}
\mbf{u}(r,z) = \begin{pmatrix} 4r^{3}(1-r)z(1-z) \\ -4r^{3}(1-r)z(1-z) \end{pmatrix}.
\end{align}

Correspondingly, the true symmetric stress tensor is
\begin{scriptsize}
\begin{align}
\sigtens &= \begin{pmatrix} 4r^{2}(-2r^{2}z+r^{2}+9rz^{2}-7rz-r-7z^{2}+7z) &  2r^{2}(2r^{2}z-r^{2}-4rz^{2}+2rz+r+3z^{2}-3z) \\ 
2r^{2}(2r^{2}z-r^{2}-4rz^{2}+2rz+r+3z^{2}-3z) & 4r^{2}(-4r^{2}z + 2r^{2} + 5rz^{2}-rz-2r-4z^{2}+4z) \end{pmatrix} \\
\sigma & = 4r^{2}(-2r^{2}z + r^{2} + 6 r z^{2} - 4rz -r - 5z^{2} + 5z)
\end{align}
\end{scriptsize}
and the divergence of the stress tensor is
\begin{align}
\axidiv (\sigtens, \sigma) = \begin{pmatrix} 2r(2r^{3}-24r^{2}z+10r^{2}+60rz^{2}-42rz-9r-32z^{2}+32z) \\ -2r(8r^{3}+r^{2}(7-30z) + 4r(4z^{2}+2z-3)-9(z-1)z) \end{pmatrix}.
\end{align}

The solution was chosen to be consistent with homogenous Dirichlet conditions while having a sufficiently
high order polynomial degree to investigate the orders of convergence.  

Presented in Table \ref{fig:elasticity_bdm1_example2}-\ref{fig:elasticity_bdm2_example2} are the results of the 
simulation.  We note that the convergence rate for the displacement reflects the true displacement,  $\| \mbf{u} - \mbf{u}_{h} \|_{U}$.

\begin{table}[t]
\centering
\caption{Experiment 1: Convergence rates for $(\textbf{BDM}_{1}(\mcT_{h}))^{2} \times P_{1}(\mcT_{h})) \times (P_{0}(\mcT_{h}))^{2} 
\times P_{0}(\mcT_{h})$
finite elements with grad-div 
stabilization parameter $\gamma = 1$.}
\label{fig:elasticity_bdm1_example2}
\tabulinesep=1.15mm
\small
\begin{tabu}{|c|c|c|c|c|c|c|}
\hline
$h$ & $\|(\sigtens , \sigma) - (\sigtens_{h} , \sigma_{h}) \|_{\Sigma}$ & Cvg. Rate & $\| \mbf{u} - \mbf{u}_{h} \|_{U}$ & Cvg. Rate & $\|\text{as}(\sigtens - \sigtens_{h}) \|_{Q}$ & Cvg. Rate \\
\hline
$\frac{1}{4}$	&	1.273E+00	&	1.0	&	2.908E-02	&	1.0	&	1.912E-01	&	1.1	\\	\hline
$\frac{1}{6}$	&	8.444E-01	&	1.0	&	1.911E-02	&	1.1	&	1.200E-01	&	1.1	\\	\hline
$\frac{1}{8}$	&	6.308E-01	&	1.0	&	1.410E-02	&	1.0	&	8.636E-02	&	1.1	\\	\hline
$\frac{1}{10}$	&	5.034E-01	&	1.0	&	1.115E-02	&	1.0	&	6.727E-02	&	1.1	\\	\hline
$\frac{1}{12}$	&	4.189E-01	&	--	&	9.227E-03	&	--	&	5.508E-02	&	--	\\	\hline
Pred.               &              &  1.0  &               &  1.0 &               &   1.0 \\ \hline
\end{tabu}
\end{table}

\begin{table}[t]
\centering
\caption{Experiment 1: Convergence rates for $(\textbf{BDM}_{2}(\mcT_{h}))^{2} \times P_{2}(\mcT_{h})) 
\times (P_{1}(\mcT_{h}))^{2} \times P_{1}(\mcT_{h})$  
finite elements with grad-div 
stabilization parameter $\gamma = 1$.}
\label{fig:elasticity_bdm2_example2}
\tabulinesep=1.15mm
\small
\begin{tabu}{|c|c|c|c|c|c|c|}
\hline
$h$ & $\| (\sigtens , \sigma) - (\sigtens_{h} , \sigma_{h}) \|_{\Sigma}$ & Cvg. Rate & $\| \mbf{u} - \mbf{u}_{h} \|_{U}$ & Cvg. Rate & $\|\text{as}(\sigtens - \sigtens_{h}) \|_{Q}$ & Cvg. Rate \\
\hline
$\frac{1}{4}$	&	6.797E-02	&	2.0	&	8.381E-03	&	1.9	&	1.602E-02	&	2.1	\\	\hline
$\frac{1}{6}$	&	3.061E-02	&	2.0	&	3.915E-03	&	1.9	&	6.753E-03	&	2.1	\\	\hline
$\frac{1}{8}$	&	1.730E-02	&	2.0	&	2.238E-03	&	2.0	&	3.647E-03	&	2.1	\\	\hline
$\frac{1}{10}$	&	1.109E-02	&	2.0	&	1.442E-03	&	2.0	&	2.264E-03	&	2.1	\\	\hline
$\frac{1}{12}$	&	7.711E-03	&	--	&	1.005E-03	&	--	&	1.536E-03	&	--	\\	\hline
Pred.               &              &  2.0  &               &  2.0 &               &   2.0 \\ \hline
\end{tabu}
\end{table}

\subsubsection*{Experiment 2}
For this numerical experiment, we considered the displacement solution
\begin{align}
\mbf{u}(r,z) = \begin{pmatrix} r^{3} \sin(r \pi) \cos((z-0.5) \pi) \\ -r^{3} \sin(r \pi) \cos((z-0.5) \pi) \end{pmatrix}.
\end{align}
This solution was selected to be consistent with homogenous Dirichlet conditions while also providing a non-polynomial validation example.  Based on $\mbf{u}$, the true solution for $\sigtens$ was determined from the relationship
\begin{align}
\label{eq:tensor_displacment_relationship}
\mathcal{A} \sigtens = \epsilon (\mbf{u}) \quad \text{   where   } \quad \mathcal{A} \sigtens = \dfrac{1}{2 \mu} \left( \sigtens - \dfrac{\lambda}{2 \mu + 3 \lambda} \text{tr}( \sigtens) \right) \, .
\end{align}
For brevity, the expressions for $(\sigtens, \sigma)$ and $\axidiv (\sigtens, \sigma)$ are omitted here.\\

The results of the simulations are presented in Tables \ref{fig:elasticity_bdm1_exampler2} and \ref{fig:elasticity_bdm2_exampler2}.

\begin{table}[t]
\centering
\caption{Experiment 2: Convergence Rates for $(\mathbf{BDM}_{1} (\mathcal{T}_{h}))^{2} \times P_{1}(\mathcal{T}_{h}) \times (P_{0}(\mathcal{T}_{h}))^{2} \times P_{0}(\mathcal{T}_{h})$ finite elements with grad-div stabilization parameter $\gamma = 1$.}
\label{fig:elasticity_bdm1_exampler2}
\tabulinesep=1.15mm
\small
\begin{tabu}{|c|c|c|c|c|c|c|}
\hline
$h$ & $\| (\sigtens, \sigma) - (\sigtens_{h}, \sigma_{h}) \|_{\Sigma}$ & Cvg. Rate & $\| \mbf{u} - \mbf{u}_{h} \|_{U}$ & Cvg. Rate & $\|\text{as}(\sigtens - \sigtens_{h}) \|_{Q}$ & Cvg. Rate \\
\hline
$\frac{1}{4}$	&	3.235E+00	&	1.0	&	8.675E-02	&	1.1	&	6.103E-01	&	1.1	\\	\hline
$\frac{1}{6}$	&	2.136E+00	&	1.0	&	5.619E-02	&	1.1	&	3.862E-01	&	1.1	\\	\hline
$\frac{1}{8}$	&	1.596E+00	&	1.0	&	4.111E-02	&	1.1	&	2.811E-01	&	1.1	\\	\hline
$\frac{1}{10}$	&	1.275E+00	&	1.0	&	3.239E-02	&	1.1	&	2.209E-01	&	1.1	\\	\hline
$\frac{1}{12}$	&	1.062E+00	&	--	&	2.674E-02	&	--	&	1.821E-01	&	--	\\	\hline
Pred.               &              &  1.0  &               &  1.0 &               &   1.0 \\ \hline
\end{tabu}
\end{table}

\begin{table}[t]
\centering
\caption{Experiment 2: Convergence Rates for $(\mathbf{BDM}_{2} (\mathcal{T}_{h}))^{2} \times P_{2}(\mathcal{T}_{h}) \times (P_{1}(\mathcal{T}_{h}))^{2} \times P_{1}(\mathcal{T}_{h})$ finite elements with grad-div stabilization parameter $\gamma = 1$.}
\label{fig:elasticity_bdm2_exampler2}
\tabulinesep=1.15mm
\small
\begin{tabu}{|c|c|c|c|c|c|c|}
\hline
$h$ & $\| (\sigtens, \sigma) - (\sigtens_{h}, \sigma_{h}) \|_{\Sigma}$ & Cvg. Rate & $\| \mbf{u} - \mbf{u}_{h} \|_{U}$ & Cvg. Rate & $\|\text{as}(\sigtens - \sigtens_{h}) \|_{Q}$ & Cvg. Rate \\
\hline
$\frac{1}{4}$	&	4.291E-01	&	1.9	&	2.720E-02	&	1.9	&	6.467E-02	&	2.1	\\	\hline
$\frac{1}{6}$	&	1.966E-01	&	2.0	&	1.243E-02	&	2.0	&	2.816E-02	&	2.1	\\	\hline
$\frac{1}{8}$	&	1.119E-01	&	2.0	&	7.036E-03	&	2.0	&	1.556E-02	&	2.1	\\	\hline
$\frac{1}{10}$	&	7.208E-02	&	2.0	&	4.514E-03	&	2.0	&	9.825E-03	&	2.1	\\	\hline
$\frac{1}{12}$	&	5.023E-02	&	--	&	3.138E-03	&	--	&	6.752E-03	&	--	\\	\hline
Pred.               &              &  2.0  &               &  2.0 &               &   2.0 \\ \hline
\end{tabu}
\end{table}

The computational results are consistent with the theoretically predicted results from 
Corollaries \ref{cor:interpolation_error_bound_k12} and \ref{err4tru}.


%
\section{Conclusion}
\label{sec:elasticity_discussion}

We have developed a computational framework for the axisymmetric linear elasticity problem with weak symmetry.  
Provided the projection bounds \eqref{eq:projection_bound_1}-\eqref{eq:projection_bound_2} are satisfied, Lemmas
\ref{lem:bdm1_projection} and  \ref{lem:BDM2_projection_stable} establish that the 
finite element spaces 
$(\textbf{BDM}_{1}(\mcT_{h}))^{2} \times P_{1}(\mcT_{h})) \times (P_{0}(\mcT_{h}))^{2} \times P_{0}(\mcT_{h})$ and
$(\textbf{BDM}_{2}(\mcT_{h}))^{2} \times P_{2}(\mcT_{h})) 
\times (P_{1}(\mcT_{h}))^{2} \times P_{1}(\mcT_{h})$ 
are inf-sup stable, resulting in approximations satisfying the error bounds stated in 
Corollary \ref{cor:interpolation_error_bound_k12}.  
Computational  presented in Section \ref{sec:computational_results} 
support these results.

It is an open question if for $k \ge 3$,  $(\textbf{BDM}_{k}(\mcT_{h}))^{2} \times P_{k}(\mcT_{h})) 
\times (P_{k-1}(\mcT_{h}))^{2} \times P_{k-1}(\mcT_{h})$ form 
an inf-sup stable set of  approximation spaces for this problem.  

In the Cartesian setting, the spaces 
$(\mbf{BDM}_{k}(\mcT_{h}))^{2} \times (P_{k-1}(\mcT_{h}))^{2} \times P_{k-1}(\mcT_{h})$ form an inf-sup stable
set of approximation spaces 
for the linear elasticity problem with weak symmetry \cite{BBF1}.  
Therefore, it is reasonable to conjecture that 
$(\mbf{BDM}_{k}(\mcT_{h}))^{2} \times P_{k}(\mcT_{h})) \times (P_{k-1}(\mcT_{h}))^{2} \times P_{k-1}(\mcT_{h})$ 
are inf-sup stable for the axisymmetric problem.
To test this conjecture, Tables \ref{fig:elasticity_bdm3_example2} and \ref{fig:elasticity_bdm3_exampler2} present convergence results for 
$(\mbf{BDM}_{3}(\mcT_{h}))^{2} \times P_{3}(\mcT_{h})) \times (P_{2}(\mcT_{h}) )^{2} \times P_{2}(\mcT_{h})$ for the 
numerical experiments described in 
Section \ref{sec:computational_results}. For these experiments the approximations
converge with convergence rate $O(h^{k}) \, = \, O(h^{3})$. \\

\begin{table}[t]
\centering
\caption{Experiment 1: Convergence rates for $(\mbf{BDM}_{3}(\mcT_{h}))^{2} \times P_{3}(\mcT_{h})) \times (P_{2}(\mcT_{h}) )^{2} 
\times P_{2}(\mcT_{h})$ 
finite elements with $\gamma = 1$.}
\label{fig:elasticity_bdm3_example2}
\tabulinesep=1.15mm
\small
\begin{tabu}{|c|c|c|c|c|c|c|}
\hline
$h$ & $\| (\sigtens , \sigma) - (\sigtens_{h} , \sigma_{h}) \|_{\Sigma}$ & Cvg. Rate & $\| \mbf{u} - \mbf{u}_{h} \|_{U}$ & Cvg. Rate & $\|\text{as}(\sigtens - \sigtens_{h}) \|_{Q}$ & Cvg. Rate \\
\hline
$\frac{1}{4}$	&	1.155E-02	&	3.0	&	1.359E-03	&	2.9	&	1.454E-03	&	3.1	\\	\hline
$\frac{1}{6}$	&	3.459E-03	&	3.0	&	4.233E-04	&	2.9	&	4.192E-04	&	3.1	\\	\hline
$\frac{1}{8}$	&	1.465E-03	&	3.0	&	1.816E-04	&	3.0	&	1.733E-04	&	3.1	\\	\hline
$\frac{1}{10}$	&	7.517E-04	&	3.0	&	9.370E-05	&	3.0	&	8.742E-05	&	3.1	\\	\hline
$\frac{1}{12}$	&	4.356E-04	&	--	&	5.445E-05	&	--	&	5.002E-05	&	--	\\	\hline
\end{tabu}
\end{table}

\begin{table}[t]
\centering
\caption{Example 2: Convergence Rates for $(\mathbf{BDM}_{3} (\mathcal{T}_{h}))^{2} \times P_{3}(\mathcal{T}_{h}) \times (P_{2}(\mathcal{T}_{h}))^{2} \times P_{2}(\mathcal{T}_{h})$ finite elements with grad-div stabilization parameter $\gamma = 1$.}
\label{fig:elasticity_bdm3_exampler2}
\tabulinesep=1.15mm
\small
\begin{tabu}{|c|c|c|c|c|c|c|}
\hline
$h$ & $\| (\sigtens, \sigma) - (\sigtens_{h}, \sigma_{h}) \|_{\Sigma}$ & Cvg. Rate & $\| \mbf{u} - \mbf{u}_{h} \|_{U}$ & Cvg. Rate & $\|\text{as}(\sigtens - \sigtens_{h}) \|_{Q}$ & Cvg. Rate \\

\hline
$\frac{1}{4}$	&	5.470E-02	&	2.9	&	4.610E-03	&	2.9	&	7.849E-03	&	3.0	\\	\hline
$\frac{1}{6}$	&	1.658E-02	&	3.0	&	1.422E-03	&	3.0	&	2.355E-03	&	3.0	\\	\hline
$\frac{1}{8}$	&	7.046E-03	&	3.0	&	6.085E-04	&	3.0	&	9.937E-04	&	3.0	\\	\hline
$\frac{1}{10}$	&	3.620E-03	&	3.0	&	3.136E-04	&	3.0	&	5.082E-04	&	3.0	\\	\hline
$\frac{1}{12}$	&	2.098E-03	&	--	&	1.821E-04	&	--	&	2.937E-04	&	--	\\	\hline
\end{tabu}
\end{table}

{\Large \textbf{Acknowledgement}}: The authors thankfully acknowledge helpful discussions with Professors Jason Howell and 
Hengguang Li.

%
 \setcounter{equation}{0}
\setcounter{figure}{0}
\setcounter{table}{0}
\setcounter{theorem}{0}
\setcounter{lemma}{0}
\setcounter{corollary}{0}
\setcounter{definition}{0}
\appendix

\section{Proof of Lemma \ref{lem:c_operator_result}}
\label{apxpLm1}
Let $\mbf{x} = (r,z)$.  Then
\begin{align*}
\axidiv (\tautens \wedge \mbf{x}) &= \axidiv \begin{pmatrix}  \tau_{11}  z-  \tau_{21} r \\
  \tau_{12} z-  \tau_{22} r  \end{pmatrix} \\
& = \dfrac{\partial}{\partial r} ( \tau_{11} z -  \tau_{21} r) + \dfrac{\partial}{\partial z} (   \tau_{12} z -  \tau_{22} r) + \dfrac{1}{r} (  \tau_{11} z -  \tau_{21} r)   \\
& = z \dfrac{\partial \tau_{11}}{\partial r} - \tau_{21}  - r \dfrac{\partial \tau_{21}}{\partial r} + \tau_{12} + z \dfrac{\partial \tau_{12}}{\partial z} - r \dfrac{\partial \tau_{22}}{\partial z}   + \dfrac{z}{r} \tau_{11}  - \tau_{21}\\
& = z(\dfrac{\partial \tau_{11}}{\partial r}  + \dfrac{\partial \tau_{12}}{\partial z} + \dfrac{1}{r} \tau_{11}) - r (\dfrac{\partial \tau_{21}}{\partial r} + \dfrac{\partial \tau_{22}}{\partial z} + \frac{1}{r} \tau_{21}) + \tau_{12} - \tau_{21}\\
& = (\axidiv \begin{pmatrix} \tau_{11} & \tau_{12} \\ \tau_{21} & \tau_{22} \end{pmatrix}) \wedge \begin{pmatrix} r \\ z\end{pmatrix} +  \begin{pmatrix} \tau_{11} & \tau_{12} \\ \tau_{21} & \tau_{22} \end{pmatrix} : \begin{pmatrix} 0 & 1 \\ -1 & 0 \end{pmatrix} \\
& = (\axidiv \tautens) \wedge \mbf{x} +  \tautens : \mathbb{P}, \stepcounter{equation}\tag{\theequation}\label{eq:exterior_result_1}
\end{align*}
where $\mathbb{P} = \begin{pmatrix} 0 & 1 \\ -1 & 0 \end{pmatrix}$. 
Therefore,
\begin{align*}
\axidiv(\tautens \wedge \mbf{x}) - \dfrac{z}{r} \tau = \axidiv(\tautens, \tau) \wedge \mbf{x} + \tautens : \mathbb{P}.
\end{align*}
Next we multiply the left and right hand sides of (\ref{eq:exterior_result_1}) by $p \; r$ and integrate over $T$ to yield
\begin{align}
\label{eq:projection_result_1}
\int_{T} \axidiv (\tautens \wedge \mbf{x} ) \; p \; r \; d T - \int_{T} z \; \tau \; p \; d T &= \int_{T} (\axidiv (\tautens, \tau) ) \wedge \mbf{x} \; p \; r \; dT + \int_{T} \tautens : \mathbb{P} \; p \; r \; dT \nonumber \\
&= ( (\axidiv (\tautens, \tau)) \wedge \mbf{x}, p)_{T} + ( as(\tautens) , \mathcal{S}^{2}(p))_{T} \nonumber \\
& = c( (\tautens, \tau), p)_{T} \ \ \mbox{(see \eqref{cres2T})}.
\end{align}
Note that we have used the relationship $\tautens : \mathbb{P} \; p  = as(\tautens) : \mathcal{S}^{2}(p)$.
Next, applying integration by parts to the first term on the left-hand side of \eqref{eq:projection_result_1} gives
\begin{align}
\label{eq:IBP}
\int_{T} \axidiv (\tautens \wedge \mbf{x} ) \; p \; r \; dT  &= \int_{T} \nabla \cdot (r \; \tautens \wedge \mbf{x}) \; p \; dT \nonumber \\
& = \int_{\partial T} (\tautens \wedge \mbf{x} ) \cdot \mbf{n} \; p \; r \; \partial s - \int_{T} (\tautens \wedge \mbf{x} ) \cdot \nabla  p \, \; r \; d T.
\end{align}
Then combining \eqref{eq:projection_result_1}, and \eqref{eq:IBP} yields
\begin{align*}
c(\tautens, p) = \int_{\partial T} (\tautens \wedge \mbf{x} ) \cdot \mbf{n} \; p \; r \;ds - \int_{T} (\tautens \wedge \mbf{x} ) \cdot \nabla p \, \; r \; dT - \int_{T} \tau \; z\; p \; dT. 
\end{align*}
Finally, since
\begin{align*}
(\tautens \wedge \mbf{x} ) \cdot \nabla p &= \tautens : (\mbf{x}^{\perp} \otimes \nabla p)  \text{ and }
(\tautens \wedge \mbf{x} ) \cdot \mbf{n} = (\tautens \cdot \mbf{n}) \cdot \mbf{x}^{\perp} ,
\end{align*}
we have
\begin{align*}
c((\tautens, \tau), p)_{T}= \int_{\partial T} (\tautens \cdot \mbf{n}) \cdot \mbf{x}^{\perp} \; p \; r \;ds - \int_{T} \tautens : (\mbf{x}^{\perp} \otimes \nabla p) \; r \; dT - \int_{T} \tau \; z \; p \; dT.
\end{align*}
\mbox{ } \hfill \qed


%
 \setcounter{equation}{0}
\setcounter{figure}{0}
\setcounter{table}{0}
\setcounter{theorem}{0}
\setcounter{lemma}{0}
\setcounter{corollary}{0}
\setcounter{definition}{0}
\section{Computations on the reference triangle $\widehat{T}$}
\label{sec:mappings}

The reference triangle $\widehat{T}$ is defined as the triangle with vertices $(0,0)$, $(1,0)$ and $(0,1)$. 
\begin{figure}[h]
\center
\includegraphics[scale=0.4]{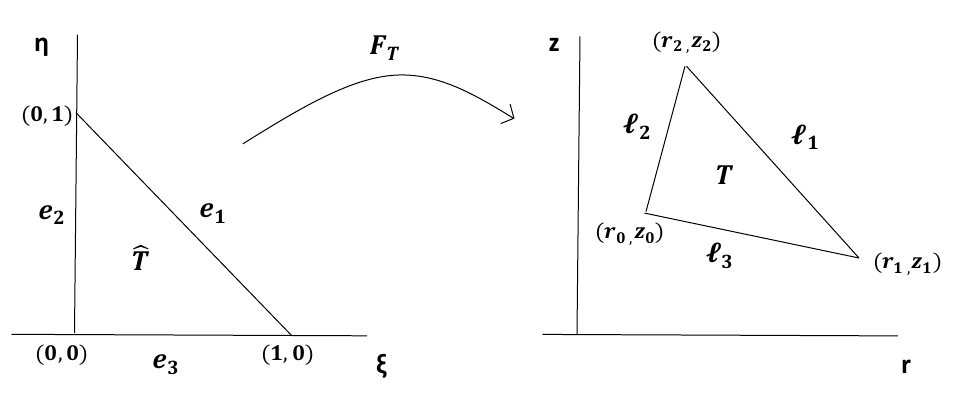}
\caption{Reference Triangle}
\label{fig:affine_map}
\end{figure}
 
Every triangle $T \in \mathcal{T}_{h}$ has three coordinates $(r_{0}, z_{0})$, $(r_{1}, z_{1})$ and $(r_{2},z_{2})$, which
we assume are always labeled in a counter-clockwise manner such that $r_{0} \leq r_{1}, r_{2}$.  
Further, an affine mapping $F_{T}$ from the reference triangle $\widehat{T}$ (see Figure \eqref{fig:affine_map}) to the 
physical domain $T \in \mcT_{h}$ exists and takes the form
\begin{align*}
\begin{pmatrix} r \\ z \end{pmatrix} = \begin{pmatrix} r_{1} - r_{0} & r_{2} - r_{0} \\ z_{1} - z_{0} & z_{2} - z_{0} \end{pmatrix} 
\begin{pmatrix} \xi \\ \eta \end{pmatrix} + \begin{pmatrix} r_{0} \\ z_{0} \end{pmatrix} = 
\begin{pmatrix} r_{10} & r_{20} \\ z_{10} & z_{20} \end{pmatrix} \begin{pmatrix} \xi \\ \eta \end{pmatrix} + 
\begin{pmatrix} r_{0} \\ z_{0} \end{pmatrix}. 
\end{align*}

Observe that we have used the notational short hand $r_{i} - r_{j} = r_{ij}$, and $z_{i} - z_{j} = z_{ij}$.  
Associated with each affine mapping $F_{T}$ is the determinant of the Jacobian matrix $|J_{T}| = |r_{10}z_{20} - z_{10} r_{20}|$. 

Provided that the triangulation $\mathcal{T}_{h}$ is regular, every affine map $F_{T}$ can be expressed as
\begin{align*}
\begin{pmatrix} r \\ z \end{pmatrix} = \begin{pmatrix} r_{10} & r_{20} \\ z_{10} & z_{20} \end{pmatrix} 
\begin{pmatrix} \xi \\ \eta \end{pmatrix} + \begin{pmatrix} r_{0} \\ z_{0} \end{pmatrix} = h 
\begin{pmatrix} c_{10} & c_{20} \\ d_{10} & d_{20} \end{pmatrix} \begin{pmatrix} \xi \\ \eta \end{pmatrix} + 
\begin{pmatrix} r_{0} \\ z_{0} \end{pmatrix},
\end{align*}
where $a_{min} \leq c_{10}, c_{20}, d_{10}, d_{20} \leq a_{max}$.  Furthermore, the determinant of the Jacobian 
is $|J_{T}| = h^{2} ( c_{10} d_{20} - d_{10} c_{20}) = h^{2} J_{D}$ where $0< jd_{min} \leq J_{D} \leq jd_{max}$, for
$jd_{min}, \,  jd_{max} \in \real^{+}$.

For every regular triangulation $\mathcal{T}_{h}$ of an axisymmetric domain $\Omega$ with symmetry axis $\Gamma_{0}$, 
each triangle $T \in \mathcal{T}_{h}$ can be categorized as one of three types:
\begin{itemize}
\item Type I: $\partial T \cap \Gamma_{0} = \mbf{e}^{*}$ where $\mbf{e}^{*}$ denotes an entire edge,
\item Type II: $\partial T \cap \Gamma_{0} = P_{0}$ where $P_{0}$ is a single point,
\item Type III: $\partial T \cap \Gamma_{0} = \emptyset$.
\end{itemize}
For each type of triangle, we can be more specific about the form of the affine mapping $F_{T}$.  In the following, 
$\hat{r}$ and $\hat{z}$ represent the mapping of the variables $r$ and $z$ on the physical element $T$ to the reference 
triangle $\widehat{T}$ as functions of $\xi$ and $\eta$.

If $T$ is Type I, then
\begin{align*}
\hat{r} &= h \; c_{10} \; \xi\\
\hat{z} &=  (z_{0} + h\;d_{10}  \xi + h\;d_{20} \eta)
\end{align*}
and
\begin{align*}
J_{T} = h \begin{pmatrix} c_{10} & 0 \\ d_{10} & d_{20} \end{pmatrix} = h \tilde{J}_{T}.
\end{align*}
Since $c_{20} = 0$, it must be the case that $c_{10} > 0$ to ensure that $T$ is well defined.

If $T$ is Type II, then
\begin{align*}
\hat{r} &= h \; (c_{10} \xi + c_{20} \eta) \\
\hat{z} &= (z_{0} + h\;d_{10}  \xi + h\;d_{20} \eta)
\end{align*}
and
\begin{align*}
J_{T} = h \begin{pmatrix}c_{10} & c_{20} \\ d_{10} & d_{20} \end{pmatrix} = h \tilde{J}_{T}.
\end{align*}
In addition, since only one node lies on the symmetry axis, $c_{10}, c_{20} > 0$.

Finally, if $T$ is Type III, then
\begin{align*}
\hat{r} &= (r_{0}  + h\; c_{10}  \xi + h\;c_{20}  \eta) \\
\hat{z} &= (z_{0} + h\;d_{10} \xi + h\;d_{20}  \eta)
\end{align*}
and
\begin{align*}
J_{T} = h \begin{pmatrix} c_{10} & c_{20} \\ d_{10} & d_{20} \end{pmatrix} = h \tilde{J}_{T}
\end{align*}
where $c_{10}, c_{20} \geq 0$ and $c_{10} + c_{20} > 0$.

In many cases, is it more convenient to work on the reference triangle $\widehat{T}$ than the physical domain $T$.  
However, it is important to recall that when mapping vector functions in ${}_{1}\Htens(\axidiv, \mathbb{M}^{2})$ between $T$ 
and $\widehat{T}$, it is necessary to preserve normal components.  Therefore, rather than using a standard affine mapping, 
we must use the contravariant Piola transformation \cite{Ervin, Bentley}.  Let $J_{T}$ be the Jacobian matrix associated with the 
affine mapping $F_{T}: \widehat{T} \rightarrow T$, then the Piola mapping of the function $\widehat{\mbf{q}}$ 
(defined on the reference triangle) is
\begin{align*}
\mathcal{P}(\widehat{\mbf{q}})(\mbf{x}) := \dfrac{1}{|J_{T}|} J_{T} \; \widehat{\mbf{q}}(\widehat{\mbf{x}}), 
\text{ where } \mbf{x} = F(\widehat{\mbf{x}}).
\end{align*}

The following lemma describes some useful properties of the Piola map as it relates to the integration of 
${}_{1}\Htens(\axidiv, \mathbb{M}^{2})$ functions.
\begin{lemma}
\label{lem:piola_mapping_properties}
Let $\widehat{\tautens}, \widehat{\sigtens} \in {}_{1} \Htens(\axidiv, \widehat{T}; \mathbb{M}^{2})$ and 
$\widehat{\mbf{v}} \in {}_{1} \mbf{L}^{2}(\widehat{T})$, and let $\tautens = \mathcal{P}(\widehat{\tautens})$, 
$\sigtens = \mathcal{P}(\widehat{\sigtens})$, and $\mbf{v} = \widehat{\mbf{v}} \circ F^{-1}$
\begin{align*}
\int_{T} \tautens : \sigtens \; r \; d T &= \int_{\widehat{T}} J_{T} \; \widehat{\tautens}^{t} : J_{T} \; \widehat{\sigtens}^{t} \; 
\dfrac{1}{|J_{T}|} \; \hat{r} \; d \widehat{T} \\
\int_{\partial T} (\tautens \cdot \mbf{n}) \cdot \mbf{v} \; r \; d s &= \int_{\partial \widehat{T}} (\widehat{\tautens} \cdot \mbf{n}) \cdot
 \widehat{\mbf{v}} \; \hat{r} \; d s 
\end{align*}
\end{lemma}
Additional details and proofs can be found in \cite{Ervin, BBF}.

As a result of using polynomials as the discrete finite element approximation spaces, many of the integrals that appear in the 
finite element formulation have a similar structure.  The next lemma introduces an analytical solution for a common class of 
integrals that appear in the discrete finite element formulation of the axisymmetric linear elasticity problem.

To begin, for convenience of notation, if $r_{0}>0$, let
\begin{align*}
\dfrac{\hat{r}}{r_{0}} = \dfrac{(r_{1}-r_{0})}{r_{0}} \xi + \dfrac{(r_{2}-r_{0})}{r_{0}} \eta + 1 = r_{1}^{*} \xi + r_{2}^{*} \eta + 1
\end{align*}
while if $r_{0} = 0$, then
\begin{align*}
\hat{r} = r_{1} \xi + r_{2} \eta.
\end{align*}

Since we assume that the coordinates of $T$ are labeled such that $r_{0} \leq r_{1},r_{2}$, it follows that $r_{1}^{*}, r_{2}^{*} \geq 0$.  
Thus, if we are calculating the integral of a function $f(r, z)$ on $T$ using the reference element $\widehat{T}$, 
\begin{align}
\label{eq:integral_short_hand}
\int_{T} f(r,z) \; r \; dT  = \begin{cases} r_{0} \int_{\widehat{T}} \hat{f}(\xi, \eta) \; (r_{1}^{*} \xi + r_{2}^{*} \eta + 1) \; d \widehat{T} = r_{0} \; I (\hat{f}(\xi, \eta)) \text{ if } r_{0} > 0 \\[3mm]
\int_{\widehat{T}} \hat{f}(\xi, \eta) \; (r_{1} \xi + r_{2} \eta) \; d \widehat{T} = I(\hat{f}(\xi, \eta)) \text{ if } r_{0} =0.
\end{cases}
\end{align}
where $d \widehat{T} = |J_{T}| \; d\xi \; d\eta$.

\begin{lemma}
\label{lem:integral_eval}
For integers $s \geq 0$ and $t \geq 0$, 
\begin{align}
\int_{0}^{1} \int_{0}^{1-\xi} \xi^{s} \eta^{t} (r_{1} \xi + r_{2} \eta + 1) \; d \eta \; d \xi &= \dfrac{s! \; t!}{(s+t+2)!} 
\left[ \dfrac{r_{1} (s+1) }{(s+t+3)} + \dfrac{r_{2}(t+1)}{(s+t+3)} + 1 \right] \nonumber \\
&= \dfrac{s! \; t!} {(s+t+3)!} \left[ r_{1} (s+1) + r_{2}(t+1) + (s+t+3) \right]   \label{eq:r_1_case}
\end{align}
and
\begin{align}
\int_{0}^{1} \int_{0}^{1-\xi} \xi^{s} \eta^{t} (r_{1} \xi + r_{2} \eta) \; d \eta \; d \xi &= \dfrac{s! \; t!}{(s+t+2)!} 
\left[ \dfrac{r_{1} (s+1) }{(s+t+3)} + \dfrac{r_{2}(t+1)}{(s+t+3)} \right] \nonumber \\
&= \dfrac{s! \; t!} {(s+t+3)!} \left[ r_{1} (s+1) + r_{2}(t+1)\right].  \label{eq:r_0_case}
\end{align}
\end{lemma}
\begin{proof}
First, for $\Gamma(\cdot)$ denoting the gamma function, note that
\begin{align*}
\int_{0}^{1} \int_{0}^{1 - \xi} \xi^{s} \eta^{t} \; d \eta \; d \xi = \int_{0}^{1} \dfrac{\xi^{s}(1 - \xi)^{t+1}}{t+1} \; d \xi = \dfrac{1}{t+1} 
\dfrac{\Gamma(s+1) \Gamma(t+2)}{\Gamma(s+t+3)} = \dfrac{s! \; t!}{(s+t+2)!}.
\end{align*}
Therefore
\begin{align*}
& \int_{0}^{1} \int_{0}^{1-\xi} \xi^{s} \eta^{t} (r_{1} \xi + r_{2} \eta + 1) \; d \eta \; d \xi \\
&= r_{1} \int_{0}^{1} \int_{0}^{1 - \xi} \xi^{s+1} \eta^{t} \; d \eta \; d \xi + r_{2} \int_{0}^{1} \int_{0}^{1 - \xi} \xi^{s} \eta^{t+1} \; d \eta \; d \xi 
+ \int_{0}^{1} \int_{0}^{1 - \xi} \xi^{s} \eta^{t} \; d \eta \; d \xi\\
& = r_{1} \dfrac{(s+1)! \; t! }{(s+ t+ 3)!} + r_{2} \dfrac{s! \; (t+1)!}{(s + t + 3)!} + \dfrac{s! \; t!}{(s+t+2)!} \\
&= \dfrac{s! \; t!}{(s+t+2)!} \left( r_{1} \dfrac{(s+1)}{(s+t+3)} + r_{2} \dfrac{(t+1)}{(s+t+3)} + 1 \right).
\end{align*}
which verifies \eqref{eq:r_1_case}.  Removing the $+1$ from $(r_{1} \xi + r_{2} \eta +1)$ yields \eqref{eq:r_0_case}. 
\end{proof}

Some useful integrals computed using Lemma \ref{lem:integral_eval} for $r_{0} > 0$ are given below
\begin{align}
\label{eq:bdm_integrals}
\begin{split}
&\int_{\widehat{T}} \eta \; \hat{r} \; d \widehat{T} = \dfrac{1}{4!} \left[ r^{*}_{1} + 2 r_{2}^{*} 
+ 4 \right] \quad  \int_{\widehat{T}} \xi \; \hat{r} \; d \widehat{T}  = \dfrac{1}{4!} \left[ 2r^{*}_{1} + r^{*}_{2} + 4 \right] \\
&\int_{\widehat{T}} \eta^{2} \; \hat{r} \; d \widehat{T} = \dfrac{2}{5!} \left[ r^{*}_{1} + 3 r^{*}_{2} + 5 \right] \quad 
\int_{\widehat{T}} \xi \eta \; \hat{r} \; d \widehat{T} = \dfrac{1}{5!} \left[2r_{1}^{*} + 2 r_{2} + 5 \right] \\
& \int_{\widehat{T}} \xi^{2} \; \hat{r} \; d \widehat{T} = \dfrac{2}{5!} \left[3 r^{*}_{1} 
+ r^{*}_{2} + 5 \right] \quad \int_{\widehat{T}} \eta^{3} \; \hat{r} \; d \widehat{T} = \dfrac{3!}{6!} [r^{*}_{1} + 4r^{*}_{2} + 6] \\
&\int_{\widehat{T}} \xi \eta^{2} \; \hat{r} \; d \widehat{T} = \dfrac{2}{6!} \left[ 2 r^{*}_{1} + 3 r^{*}_{2} 
+ 6 \right] \quad \int_{\widehat{T}} \xi^{2} \eta \; \hat{r} \; d \widehat{T} = \dfrac{2}{6!} \left[ 3 r^{*}_{1} + 2 r^{*}_{2} + 6 \right] \\
&\int_{\widehat{T}} \xi^{3} \; \hat{r} \; d \widehat{T} = \dfrac{3!}{6!} [4 r_{1} + r_{2} + 6].
\end{split}
\end{align}

%
 \setcounter{equation}{0}
\setcounter{figure}{0}
\setcounter{table}{0}
\setcounter{theorem}{0}
\setcounter{lemma}{0}
\setcounter{corollary}{0}
\setcounter{definition}{0}
\section{Modeling equations for axisymmetric linear elasticity}
\label{sec:axisymmetric_elasticity}
In this Appendix, we illustrate how using a change of variable from Cartesian to cylindrical coordinates, 
the axisymmetric linear elasticity problem can be expressed as the decoupled meridian and azimuthal problems.  
Recall that cylindrical coordinates form a triple $(r, \theta, z)$ where $r$  is the radial distance, $\theta$ is 
the azimuthal coordinate and $z$ is the vertical coordinate.  In this section, let $\brOmega$ denote a 
three dimensional axisymmetric domain, $\Omega$ represent an $(r,z)$ cross section of $\brOmega$ and 
$\Omega_{\theta}$ denotes the domain of the $\theta$ angle.

\subsection{Cylindrical Coordinate Operators and Function Spaces}
\label{sec:cylindrical_definitions}
First we define the differential forms and inner products that arise in cylindrical coordinates.  To begin, the 
cylindrical coordinate unit vectors are denoted $\mbf{e}_{r}, \mbf{e}_{\theta}$ and $\mbf{e}_{z}$.  Expressed in terms of 
Cartesian unit vectors,
\begin{align*}
\mbf{e}_{r} = \begin{pmatrix} \cos{\theta} \\ \sin{\theta} \\ 0 \end{pmatrix}, \quad \mbf{e}_{\theta} = \begin{pmatrix} -\sin{\theta} \\ \cos{\theta} \\ 0 \end{pmatrix} \quad \text{ and } \quad \mbf{e}_{z} = \begin{pmatrix} 0 \\ 0 \\ 1 \end{pmatrix}. 
\end{align*}
One can note from these equations that the cylindrical coordinate unit vectors vary in space.  Moreover, 
unless otherwise specified, we assume tensors and vectors are represented in terms of the cylindrical coordinates unit vectors.  
That is,
\begin{align*}
\begin{pmatrix} \phi_{1} \\ \phi_{2} \\ \phi_{3} \end{pmatrix} = \phi_{1} \mbf{e_{r}} + \phi_{2} \mbf{e_{\theta}} + \phi_{3} \mbf{e_{z}}
\end{align*}
and
\begin{align*}
\begin{split}
\begin{pmatrix} \phi_{rr} & \phi_{r \theta} & \phi_{rz} \\ \phi_{\theta r} & \phi_{\theta \theta} & \phi_{\theta z} \\ \phi_{zr} & \phi_{z \theta} & \phi_{zz} \end{pmatrix} = \phi_{rr} \mbf{e_{rr}} + \phi_{r \theta} \mbf{e_{r \theta}} &+ \phi_{zz} \mbf{e_{rz }} + 
\phi_{\theta r} \mbf{e_{\theta r}} + \phi_{\theta \theta} \mbf{e_{\theta \theta}} + \phi_{\theta z} \mbf{e}_{\theta z} \\
& + \phi_{z r} \mbf{e_{z r}} + \phi_{z \theta} \mbf{e_{z \theta}} + \phi_{z z} \mbf{e}_{z z}
\end{split}
\end{align*}
where $\mbf{e}_{ij} = \mbf{e}_{i} \otimes \mbf{e}_{j}$.

As a result of the spatially varying unit vectors, differential operators in cylindrical coordinates have a different 
algebraic form than in Cartesian coordinates.  These operators are not derived here, but details can be found in many sources 
including \cite{Quartapelle}.

We use two forms of notation for differential operators in cylindrical coordinates: $\nabla_{\text{cyl}}$ and $\nabla_{\text{axi}}$.  
The first denotes the complete cylindrical coordinate operator, while the second represents the cylindrical coordinate operator 
applied to an axisymmetric function (recall that $\dfrac{\partial u}{\partial \theta} = 0$ if $u$ is axisymmetric).

The cylindrical coordinate del operator is 
\begin{align*}
\nabla_{\text{cyl}} = \mbf{e}_{r} \dfrac{\partial }{\partial r} + \mbf{e}_{\theta} \dfrac{1}{r} \dfrac{\partial}{\partial \theta} + \mbf{e}_{z} \dfrac{\partial}{\partial z}.
\end{align*}

Applied to the scalar function $f$, this gives the gradient operators
\begin{align*}
\nabla_{\text{cyl}} f = \dfrac{\partial f}{\partial r} \mbf{e}_{r} + \dfrac{1}{r} \dfrac{\partial f}{\partial \theta} \mbf{e}_{\theta} + \dfrac{\partial f}{\partial z} \mbf{e}_{z} \quad \text{ and } \quad \nabla_{\text{axi}} f = \dfrac{\partial f}{\partial r} \mbf{e}_{r} + \dfrac{\partial f}{\partial z} \mbf{e}_{z}.
\end{align*}

For a vector function $\mbf{u} = (u_{r}, u_{\theta}, u_{z})^{t}$, the gradient tensor is 
\begin{align}
\label{eq:cylindrical_vector_gradient}
\nabla_{\text{cyl}} \mbf{u} = \begin{pmatrix} \dfrac{\partial u_{r}}{\partial r} & \dfrac{1}{r} \dfrac{\partial u_{r}}{\partial \theta} - \dfrac{u_{\theta}}{r} & \dfrac{\partial u_{r}}{\partial z} \\[6pt]
\dfrac{\partial u_{\theta}}{\partial r} & \dfrac{1}{r} \dfrac{\partial u_{\theta}}{\partial \theta} + \dfrac{u_{r}}{r} & \dfrac{\partial u_{\theta}}{\partial z} \\[6pt]
\dfrac{\partial u_{z}}{\partial r} & \dfrac{1}{r} \dfrac{\partial u_{z}}{\partial \theta} & \dfrac{\partial u_{z}}{\partial z} \end{pmatrix} \quad 
\text{ and }
\quad
\nabla_{\text{axi}} \mbf{u} = \begin{pmatrix} \dfrac{\partial u_{r}}{\partial r} &  \dfrac{- u_{\theta}}{r} & \dfrac{\partial u_{r}}{\partial z} \\[6pt]
\dfrac{\partial u_{\theta}}{\partial r} & \dfrac{u_{r}}{r} & \dfrac{\partial u_{\theta}}{\partial z} \\[6pt]
\dfrac{\partial u_{z}}{\partial r} & 0 & \dfrac{\partial u_{z}}{\partial z} \end{pmatrix}.
\end{align}

For a vector function $\mbf{u} = (u_{r}, u_{z})^{t}$, we also define the gradient operator $\nabla$ such that
\begin{align*}
\nabla \mbf{u} = \begin{pmatrix} \dfrac{\partial u_{r}}{\partial r} & \dfrac{\partial u_{r}}{\partial z} \\[2mm] \dfrac{\partial u_{z}}{\partial r} & \dfrac{\partial u_{z}}{\partial z} \end{pmatrix}.
\end{align*}

The divergence operator applied to $\mbf{u} = (u_{r},u_{\theta},u_{z})^{t}$ gives
\begin{align}
\label{eq:cylindrical_vector_divergence}
\nabla_{\text{cyl}} \cdot \mbf{u} = \dfrac{1}{r} \dfrac{\partial (r \; u_{r})}{\partial r} + \dfrac{1}{r} \dfrac{\partial u_{\theta}}{\partial \theta} + \dfrac{\partial u_{z}}{\partial z} \quad \text{ and } \quad \nabla_{\text{axi}} \cdot \mbf{u} = \dfrac{1}{r} \dfrac{\partial (r \; u_{r})}{\partial r} + \dfrac{\partial u_{z}}{\partial z}.
\end{align}

The divergence of an $\mathbb{M}^{3}$ tensor $\sigtens$ is,
\begin{align*}
\begin{split}
\nabla_{\text{cyl}} \cdot \sigtens &= \begin{pmatrix} \dfrac{\partial \sigma_{rr}}{\partial r} + \dfrac{1}{r} \dfrac{\partial \sigma_{r \theta}}{\partial \theta} + \dfrac{\partial \sigma_{rz}}{\partial z} + \dfrac{1}{r} (\sigma_{rr} - \sigma_{\theta \theta})  \\[2.5mm]
\dfrac{\partial \sigma_{\theta r}}{\partial r} + \dfrac{1}{r} \dfrac{\partial \sigma_{\theta \theta}}{\partial \theta} + \dfrac{\partial \sigma_{\theta z}}{\partial z} + \dfrac{1}{r} (\sigma_{\theta r} + \sigma_{r \theta})  \\[2.5mm]
\dfrac{\partial \sigma_{zr}}{\partial r} + \dfrac{1}{r} \dfrac{\partial \sigma_{z \theta}}{\partial \theta}  + \dfrac{\partial \sigma_{zz}}{\partial z} + \dfrac{1}{r} \sigma_{zr}
\end{pmatrix} \quad \text{ and } \\
\nabla_{\text{axi}} \cdot \sigtens &= \begin{pmatrix} \dfrac{\partial \sigma_{rr}}{\partial r} +  \dfrac{\partial \sigma_{rz}}{\partial z} + \dfrac{1}{r} (\sigma_{rr} - \sigma_{\theta \theta})  \\[2.5mm]
\dfrac{\partial \sigma_{\theta r}}{\partial r} + \dfrac{\partial \sigma_{\theta z}}{\partial z} + \dfrac{1}{r} (\sigma_{\theta r} + \sigma_{r \theta})  \\[2.5mm]
\dfrac{\partial \sigma_{zr}}{\partial r}  + \dfrac{\partial \sigma_{zz}}{\partial z} + \dfrac{1}{r} \sigma_{zr}
\end{pmatrix}.
\end{split}
\end{align*}

\subsection{Meridian and Azimuthal Subspaces}
\label{sec:meridian_azimuthal_space}
Next, we assume all functions are axisymmetric and define the meridian and azimuthal subspaces for tensor and vector functions.  
In addition, we specify the action of the differential operators introduced in Section \ref{sec:cylindrical_definitions} on the 
meridian and azimuthal subspaces.

The meridian and azimuthal subspaces of $_{\alpha}\Htens(\axidiv,\brOmega;\mathbb{M}^{3})$ are
\begin{align*}
_{\alpha}\Htens_{M}(\axidiv, \; \brOmega \; ; \mathbb{M}^{3}) = \left \{ \sigtens \in \; _{\alpha} \ub{H} (\axidiv, \brOmega ;  \mathbb{M}^{3}) \;  : 
\sigtens = \begin{pmatrix}
\sigma_{rr} & 0 & \sigma_{rz} \\
0 & \sigma_{\theta \theta} & 0 \\
\sigma_{zr} & 0 & \sigma_{zz} \\
\end{pmatrix}   \right \}, \\ 
_{\alpha}\Htens_{A}(\axidiv,\brOmega; \mathbb{M}^{3}) = \left \{  \sigtens \in 
\; _{\alpha} \ub{H} (\axidiv, \brOmega ;  \mathbb{M}^{3}) :
\sigtens = 
\begin{pmatrix}
0 & \sigma_{r \theta} & 0 \\
\sigma_{\theta r} & 0 & \sigma_{\theta z} \\
0 & \sigma_{z \theta} & 0 \\
\end{pmatrix}
 \right \}.
\end{align*}

Note that $_{\alpha}\Htens(\axidiv,\brOmega; \mathbb{M}^{3}) = {}_{\alpha}\Htens_{M}(\axidiv,\brOmega; \mathbb{M}^{3}) \oplus {}_{\alpha}\Htens_{A}(\axidiv,\brOmega; \mathbb{M}^{3})$.  
This decomposition extends to tensors in $_{\alpha}\Htens(\axidiv,\brOmega;\mathbb{K}^{3})$, 
\begin{align*}
_{\alpha}\Htens_{M}(\axidiv,\brOmega;\mathbb{K}^{3}) &= \left \{\sigtens \in {}_{\alpha} \ub{H} (\axidiv, \brOmega ;  \mathbb{K}^{3}) : \sigtens = \begin{pmatrix} 0 & 0 & \sigma_{rz} \\
0 & 0 & 0 \\
-\sigma_{rz} & 0 & 0 \end{pmatrix}  \right \},  \\
_{\alpha}\Htens_{A}(\axidiv,\brOmega;\mathbb{K}^{3}) &= \left \{\sigtens \in {}_{\alpha} \ub{H} (\axidiv, \brOmega ;  \mathbb{K}^{3}) : \sigtens = \begin{pmatrix} 0 & \sigma_{r \theta} & 0 \\
-\sigma_{r \theta} & 0 & \sigma_{\theta z} \\
0 & - \sigma_{\theta z} & 0 \end{pmatrix} \right \}. 
\end{align*}

For $\sigtens \in {}_{1}\Htens_{M} (\axidiv, \brOmega; \mathbb{M}^{3})$,
\begin{align}
\label{eq:Hm_div_axi}
\nabla_{\text{axi}} \cdot \sigtens &= \left(\dfrac{\partial \sigma_{rr}}{\partial r} + \frac{1}{r} (\sigma_{rr} - \sigma_{\theta \theta}) + \dfrac{\partial \sigma_{rz}}{\partial z} \right)\mbf{e}_{r} + \left(\dfrac{\partial \sigma_{zr}}{\partial r} + \frac{1}{r} \sigma_{zr} + \dfrac{\partial \sigma_{zz}}{\partial z}\right) \mbf{e}_{z} \, , 
\end{align}
and for $\sigtens\in {}_{1}\Htens_{A} (\axidiv, \brOmega; \mathbb{M}^{3})$ 
\begin{align*}
\begin{split}
\nabla_{\text{axi}} \cdot \sigtens = &  \left(\dfrac{\partial \sigma_{r \theta}}{\partial r}  + \dfrac{\partial \sigma_{z \theta}}{\partial z} + \frac{1}{r} (\sigma_{r \theta} + \sigma_{\theta r}) \right)\mbf{e}_{\theta}.
\end{split}
\end{align*}

The meridian and azimuthal subspaces for the displacement space $_{1} \mbf{L}^{2} (\brOmega)$ are
\begin{align*}
{}_{1}\mbf{L}^{2}_{M} (\brOmega) = \left \{ \mbf{u} : \begin{pmatrix} u_{r} \\ 0 \\ u_{z} \end{pmatrix} \in {}_{1} \mbf{L}^{2} (\brOmega) \right \}
\quad \text{    and    } \quad
_{1} \mbf{L}^{2}_{A} (\brOmega) = \left \{ \mbf{u} : \begin{pmatrix} 0 \\ u_{\theta} \\ 0 \end{pmatrix} \in {}_{1}\mbf{L}^{2} (\brOmega) \right \}.
\end{align*}

For $\mbf{u}_{M} \in {}_{1} \mbf{L}^{2}_{M}(\brOmega)$ and $\mbf{u}_{A} \in {}_{1} \mbf{L}^{2}_{A}(\brOmega)$, 
the cylindrical gradient operator (\ref{eq:cylindrical_vector_gradient}) has the form
\begin{align*}
\nabla_{\text{axi}} \mbf{u}_{M} = \begin{pmatrix} \dfrac{\partial u_{r} }{\partial r} & 0 & \dfrac{\partial u_{r}}{\partial z} \\ 0 & \dfrac{u_{r}}{r} & 0 \\ \dfrac{\partial u_{z}}{\partial r} & 0 & \dfrac{\partial u_{z}}{\partial z} \end{pmatrix} \quad \text{ and } \quad
\nabla_{\text{axi}} \mbf{u}_{A} = \begin{pmatrix} 0 & \dfrac{-u_{\theta}}{r} & 0 \\ \dfrac{\partial u_{\theta}}{\partial r} & 0 & \dfrac{\partial u_{\theta}}{\partial z} \\ 0 & 0 & 0 \end{pmatrix}.
\end{align*}

For $\mbf{u}_{M} \in {}_{1} \mbf{L}^{2}_{M}(\brOmega)$ and $\mbf{u}_{A} \in {}_{1} \mbf{L}^{2}_{A}(\brOmega)$, 
the divergence operator \eqref{eq:cylindrical_vector_divergence} has the form
\begin{align*}
\axidiv \mbf{u}_{M} = \dfrac{1}{r} \dfrac{\partial_{r} (r u_{r})}{\partial r} + \dfrac{\partial u_{z}}{\partial z}
\quad \text{ and } \quad
\nabla_{\text{axi}} \cdot \mbf{u}_{A} = 0.
\end{align*}

Because of axisymmetry, the $\theta$ variable does not appear in the meridian or azimuthal subspaces.  
Therefore, for functions $p, q \in {}_{1}L^{2}(\brOmega)$, we define the axisymmetric cylindrical coordinate inner product as
\begin{align*}
(p, q) = \dfrac{1}{2 \pi} \iint_{\Omega} \int_{\theta = 0}^{2 \pi} p \; q \; r \; d \theta \; d r \; d z & = \iint_{\Omega} p \; q \; r \; dr \; dz.
\end{align*}

When working with the meridian and azimuthal problems, it is helpful to use the following reduced dimensional representations 
of the meridian and azimuthal subspaces.  To begin, elements $\mbf{u} \in {}_{1} \mbf{L}^{2}_{M}(\brOmega; \mathbb{R}^{3})$, 
can be represented as $\mathbb{R}^{2}$ vectors
\begin{align*}
\begin{pmatrix}
u_{r} \\ 0 \\ u_{z}
\end{pmatrix}
\rightarrow
\begin{pmatrix}
u_{r} \\ u_{z}
\end{pmatrix} \in {}_{1} \mbf{L}^{2}(\Omega; \mathbb{R}^{2}).
\end{align*}

Elements of ${}_{1}\Htens_{M}(\axidiv, \brOmega, \mathbb{M}^{3})$ can be represented as an $\mathbb{M}^{2}$ tensor and 
a scalar function
\begin{align*}
\begin{pmatrix}
\sigma_{rr} & 0 & \sigma_{rz} \\
0 & \sigma_{\theta \theta} & 0 \\
\sigma_{zr} & 0 & \sigma_{zz} \\
\end{pmatrix} \rightarrow
\begin{pmatrix} \sigma_{rr} & \sigma_{rz} \\ \sigma_{zr} & \sigma_{zz} \end{pmatrix} \in {}_{1} \Ltens (\Omega, \mathbb{M}^{2}) \; \; \text{ and } \; \; 
\sigma_{\theta \theta} \in {}_{1} L^{2}(\Omega) \, , 
\end{align*}
where $\axidiv \left( \begin{pmatrix} \sigma_{rr} & \sigma_{rz} \\ \sigma_{zr} & \sigma_{zz} \end{pmatrix}, \; \sigma_{\theta \theta} \right) \in {}_{1}\mbf{L}^{2}(\Omega)$.

To specify that the reduced form notation is being used, elements $\mbf{u} \in {}_{1} \mbf{L}^{2}_{M}(\brOmega; \mathbb{R}^{3})$ 
are denoted $\mbf{u}_{M}$.  Further, the reduced form of $\sigtens \in {}_{1}\Htens_{M}(\nabla \cdot, \brOmega, \mathbb{M}^{3})$ 
is the pair $(\sigtens_{M}, \sigma_{\theta \theta})$ where $\sigtens_{M}$ is a tensor component and $\sigma_{\theta \theta}$ is a 
scalar component of $\sigtens$.  
Moreover, $\nabla_{\text{axi}} \cdot (\sigtens_{M}, \sigma_{\theta \theta}) = \nabla_{\text{axi}} \cdot \sigtens$ as defined in \eqref{eq:Hm_div_axi}.

Elements of $\mbf{u} \in {}_{1}\mbf{L}^{2}_{A}(\brOmega; \mathbb{R}^{3})$ can be identified with scalar functions
\begin{align*}
\begin{pmatrix}
0 \\ u_{\theta} \\ 0
\end{pmatrix}
\rightarrow
u_{\theta} \in {}_{1} L^{2}(\Omega)
\end{align*}
and elements of $_{1}\Htens_{A}(\axidiv, \brOmega; \mathbb{M}^{3})$, can written as a $\mathbb{M}^{2}$ tensors
\begin{align*}
\begin{pmatrix} 0 & \sigma_{r \theta} & 0 \\ \sigma_{\theta r} & 0 & \sigma_{\theta z} \\ 0 & \sigma_{z \theta} & 0 \end{pmatrix} \rightarrow
\begin{pmatrix} \sigma_{r \theta} &  \sigma_{\theta r} \\
\sigma_{\theta z} & \sigma_{z \theta} \end{pmatrix} \in {}_{1} \Htens (\axidiv, \Omega; \mathbb{M}^{2}).
\end{align*}

To indicate the reduced form is being used, for $\mbf{u} \in {}_{1}\mbf{L}^{2}_{A}(\brOmega; \mathbb{R}^{3})$, the reduced form 
will be expressed simply as the scalar function $u_{\theta}$.  Further, the reduced form 
of $\sigtens \in {}_{1}\Htens_{A}(\axidiv, \brOmega; \mathbb{M}^{3})$, is denoted $\sigtens_{A}$.

Norms in reduced form are inherited from the norms of the original space.  For example, taking $\sigtens \in  {}_{1}\Htens_{M}(\axidiv, \brOmega, \mathbb{M}^{3})$,
\begin{align*}
\begin{split}
\| \sigtens \|^{2}_{_{1} \Htens_{M}(\axidiv , \brOmega; \mathbb{M}^{3})} &= \| (\sigtens_{M},\sigma_{\theta \theta}) \|^{2}_{_{1} \Htens_{M}(\axidiv , \Omega)} \\
& = \| \nabla_{\text{axi}} \cdot (\sigtens_{M},\sigma_{\theta \theta}) \|^{2}_{_{1} \Ltens(\Omega)} + \| (\sigtens_{M},\sigma_{\theta \theta}) \|^{2}_{_{1} \mbf{L}^{2} (\Omega)}. 
\end{split}
\end{align*}

In the following, we take
\begin{align*}
\bSigma &= \{ (\sigtens, \sigma) \in {}_{1}\Ltens(\Omega, \mathbb{M}^{2}) \times {}_{1}L^{2}(\Omega) : \axidiv (\sigtens, \sigma) \in {}_{1}L^{2}(\Omega) \}, \\
\Sigma &= {}_{1}\ub{H} (\nabla _{\text{axi}} \cdot, \Omega, \mathbb{M}^{2}), \\  
U & = {}_{1} \mbf{L}^{2} (\Omega), \text{ and } Q  = {}_{1} {L}^{2} (\Omega).
\end{align*}

\subsection{Axisymmetric Weak Form}
At this point, we are ready to define the weak form of the meridian and azimuthal problems.  
First we note that the strong form of the axisymmetric linear elasticity problem \eqref{eq:elasticity_strong} is
\begin{align}
\label{eq:problem_one_strong_1}
\mathcal{A} \sigtens - \frac{1}{2}(\nabla_{\text{axi}}\mbf{u} + (\nabla_{\text{axi}}\mbf{u})^{t})  = \mbf{0} \text{ in } \brOmega \\
\label{eq:problem_one_strong_2}
\axidiv \sigtens = \mbf{f} \text{ in } \brOmega
\end{align}
where we assume the clamped boundary condition, $\mbf{u} \, = \, \mbf{0}$ on $\partial \brOmega$.  

An axisymmetric solution to (\ref{eq:problem_one_strong_1}) and (\ref{eq:problem_one_strong_2}) can be expressed in 
terms of the orthogonal subspaces $\Htens_{A}(\nabla_{\text{axi}} \cdot, \brOmega, \mathbb{M}^{3})$ 
and $\Htens_{M}(\nabla_{\text{axi}} \cdot, \brOmega, \mathbb{M}^{3})$, and $_{1} \mbf{L}^{2}_{A}(\brOmega; \mathbb{R}^{3})$ 
and $_{1} \mbf{L}^{2}_{M}(\brOmega; \mathbb{R}^{3})$.

\subsubsection{Meridian problem} 
The first step to derive the meridian problem is to multiply (\ref{eq:problem_one_strong_1}) with a 
test function $\tautens \in \Htens_{M}(\axidiv, \brOmega; \mathbb{M}^{3})$ and integrate.  
For $\sigtens \in \Htens_{M}(\axidiv , \brOmega; \mathbb{M}^{3})$, $\mathcal{A} \sigtens$ has the form 
(recall the operator $\mathcal{A}$ (\ref{eq:complianceTensor}) for $m = 3$)
\begin{align*}
\mathcal{A} \sigtens = \dfrac{1}{2 \mu} \begin{pmatrix} \sigma_{rr} - \dfrac{\lambda}{2 \mu + 3 \lambda} \text{tr}(\sigtens) & 0 & \sigma_{rz} \\ 0 & \sigma_{\theta \theta} - \dfrac{\lambda}{2 \mu + 3 \lambda} \text{tr}(\sigtens) & 0 \\ \sigma_{zr} & 0 & \sigma_{zz} - \dfrac{\lambda}{2 \mu + 3 \lambda} \text{tr}(\sigtens) \end{pmatrix}.
\end{align*}

Therefore, for $\tautens \in \Htens_{M}(\axidiv, \brOmega; \mathbb{M}^{3})$, 
\begin{align*}
\mathcal{A} \; \sigtens : \tautens &= \dfrac{1}{2 \mu} (\sigma_{rr} \tau_{rr} + \sigma_{\theta \theta} \tau_{\theta \theta} + \sigma_{zz} \tau_{zz} +  \sigma_{rz} \tau_{rz} + \sigma_{zr} \tau_{zr} - \dfrac{\lambda}{2 \mu + 3 \lambda} \text{tr} (\sigtens) \text{tr} (\tautens) ).
\end{align*}

Using reduced form notation,
\begin{align}
\label{eq:problem_one_derivation_A_reduced_form}
\begin{split}
\mathcal{A} \; \sigtens : \tautens & = \dfrac{1}{2 \mu}  ( \sigtens_{M} : \tautens_{M} - \dfrac{\lambda}{2 \mu + 3 \lambda}  ( \text{tr}( \sigtens_{M} ) + \sigma_{\theta \theta} ) \text{tr} (\tautens_{M})   \\
& \quad \quad \quad \quad + \sigma_{\theta \theta} \tau_{\theta \theta} - \dfrac{\lambda}{2 \mu + 3 \lambda}  ( \text{tr}(\sigtens_{M}) + \sigma_{\theta \theta} ) \tau_{\theta \theta} ) \\
& = \mathcal{A} \sigtens_{M} : \tautens_{M} + \mathcal{A} \sigma_{\theta \theta} \; \tau_{\theta \theta} - \dfrac{1}{2 \mu} \dfrac{\lambda}{2 \mu + 3 \lambda} ( \sigma_{\theta \theta} \; \text{tr} (\tautens_{M})) +  \text{tr} (\sigtens_{M}) \tau_{\theta \theta} ),
\end{split}
\end{align} 
where the operator $\mathcal{A}$ applied to the scalar function $\sigma_{\theta \theta}$ is given by
(\ref{eq:complianceTensor}) for $m = 1$.

Integrating (\ref{eq:problem_one_derivation_A_reduced_form}) over $\Omega$ gives the bilinear 
form $a_{M}(\cdot,\cdot) : \bSigma \times \bSigma \rightarrow \mathbb{R}$,
\begin{align*}
\begin{split}
a_{M}((\sigtens_{M},\sigma_{\theta \theta}),(\tautens_{M},\tau_{\theta \theta})) &= (\mathcal{A} \sigtens_{M}, \tautens_{M}) + (\mathcal{A} \sigma_{\theta \theta}, \tau_{\theta \theta}) \\ & \quad \quad - \frac{1}{2 \mu} \frac{\lambda}{2\mu + 3 \lambda} [ (\sigma_{\theta \theta}, \Tr{\tautens_{M}})   + (\Tr{\sigtens_{M}}, \tau_{\theta \theta})].
\end{split}
\end{align*}

Observe that for $\mbf{u} \in {}_{1}\mbf{L}^{2}_{M}(\Omega)$ and $(\tautens, \tau_{\theta \theta}) \in \bSigma$,
\begin{align}
\label{eq:append_a_ibp_p1}
\begin{split}
-\int_{\Omega} \nabla_{\text{axi}} \mbf{u} : \tautens \; r \; d \Omega &= - \int_{\Omega} \begin{pmatrix} \dfrac{\partial u_{r}}{\partial r} & \dfrac{\partial u_{r}}{\partial z} \\[2.5mm] \dfrac{\partial u_{z}}{\partial r} & \dfrac{\partial u_{z}}{\partial z} \end{pmatrix} : \begin{pmatrix} \tau_{rr} & \tau_{rz} \\ \tau_{zr} & \tau_{zz} \end{pmatrix} \; r \; d \Omega - \int_{\Omega} \dfrac{u_{r}}{r} \tau_{\theta \theta} \; r \; d \Omega \\
& = -\int_{\Omega} \nabla \mbf{u}_{M} : \tautens_{M} \; r \; d \Omega - \int_{\Omega} \dfrac{u_{r}}{r} \tau_{\theta \theta} \; r \; d \Omega.
\end{split}
\end{align}

Next, we apply integration by parts to the expression
\begin{align}
\label{eq:b_operator_ibp}
\begin{split}
- \int_{\Omega} \nabla \mbf{u}_{M} : \tautens_{M} \; r \; d\Omega &= - \int_{\partial \Omega} u_{r} \; (\tautens_{M})_{1} \cdot \mbf{n} \; r \; d \Omega  + \int_{\Omega} u_{r} \; \regdiv (r \; (\tautens_{M})_{1}) \; d \Omega \\
& - \int_{\partial \Omega} u_{z} (\tautens_{M})_{2} \cdot \mbf{n} \; r \; d \Omega + \int_{\Omega} u_{z} \; \regdiv (r (\tautens_{M})_{2}) d \Omega. 
\end{split}
\end{align} 

As we are integrating over the domain $\Omega$, the boundary $\partial \Omega$ is comprised of two parts.  
The first corresponds to the boundary of the entire three dimensional domain $\partial \Omega$ upon which clamped 
displacement condition $\mbf{u}_{M} = \mbf{0}$ is enforced.  The second part of the boundary 
$\Gamma_{0}$ corresponds to the symmetry axis along which $u_{r} = 0$, and we assume that
$\tautens_{M} \cdot \mbf{n} = \mbf{0}$.  Therefore, all of the boundary integrals in \eqref{eq:b_operator_ibp} vanish so that
\begin{align}
\label{eq:append_a_ibp_p2}
\begin{split}
- \int_{\Omega} \nabla \mbf{u}_{M} : \tautens_{M} \; r \; d\Omega &= \int_{\Omega} u_{r} \; \regdiv (r \; (\tautens_{M})_{1}) \; d \Omega  + \int_{\Omega} u_{z} \; \regdiv (r (\tautens_{M})_{2}) d \Omega \\
&= \int_{\Omega} \mbf{u}_{M} \cdot \axidiv (\tautens_{M}) \; r \; d \Omega. 
\end{split}
\end{align}

Thus from \eqref{eq:append_a_ibp_p1} and \eqref{eq:append_a_ibp_p2} from we define the bilinear form 
$b_{M}(\cdot,\cdot) : \bSigma \times U \rightarrow \mathbb{R}$ as
\begin{align*}
b_{M}((\tautens_{M}, \tau_{\theta \theta}), \mbf{u}_{M}) = (\mbf{u}_{M}, \nabla_{\text{axi}} \cdot \tautens_{M}) - (u_{r}, \dfrac{\tau_{\theta \theta}}{r}).
\end{align*}

For $(\sigtens, \sigma_{\theta \theta}) \in \bSigma$, multiplying the left hand side of (\ref{eq:problem_one_strong_2}) with a 
test function $\mbf{v} \in {}_{1}\mbf{L}^{2}_{M}(\Omega)$ gives
\begin{align*}
\begin{split}
(\nabla_{\text{axi}} \cdot \sigtens) \cdot \mbf{v} & = (\partial_{r} \sigma_{rr} + \dfrac{1}{r} (\sigma_{rr} - \sigma_{\theta \theta}) + \partial_{z} \sigma_{rz} ) v_{r}
+ (\partial_{r} \sigma_{z r}  + \dfrac{1}{r} \sigma_{zr} + \partial_{z} \sigma_{zz} ) v_{z} \\
& = (\nabla_{\text{axi}} \cdot \sigtens_{M}) \cdot \mbf{v}_{M} - \dfrac{1}{r} \sigma_{\theta \theta} v_{r}.
\end{split}
\end{align*}

From integrating this expression we define the bilinear form
\begin{align*}
b_{M}((\sigtens_{M},\sigma_{\theta \theta}),\mbf{v}_{M}) = ((\nabla_{\text{axi}} \cdot \sigtens_{M}), \mbf{v}_{M} ) - (v_{r}, \dfrac{\sigma_{\theta \theta}}{r}).
\end{align*}

Finally, multiplying the right hand side of (\ref{eq:problem_one_strong_2}) with a test function 
$\mbf{v} \in {}_{1}L^{2}_{M}(\Omega)$ and integrating, defines the linear functional $(\mbf{f}_{M}, \mbf{v}_{M})$.

The meridian problem can now be defined as: 
\textit{Given $\mbf{f}_{M} \in {}_{1}\mbf{L}^{2}_{M}(\Omega)$, find
$((\sigtens_{M}, \sigma_{\theta \theta}), \mbf{u}_{M}) \in \bSigma \times U$ such that for all 
$((\tautens_{M}, \tau_{\theta \theta}), \mbf{v}_{M}) \in \bSigma \times U$}
\begin{align*}
a_{M}((\sigtens_{M},\sigma_{\theta \theta}),(\tautens_{M},\tau_{\theta \theta})) + b_{M}((\tautens_{M},\tau_{\theta \theta}), \mbf{u}_{M} ) &= 0 \\
b_{M}((\sigtens_{M},\sigma_{\theta \theta}), \mbf{v}_{M} ) &= (\mbf{f}_{M},\mbf{v}_{M}) \, .
\end{align*}

For the weak symmetry constraint (recall \eqref{eq:weak_symmetry_elasticity_3}), we define the bilinear form $c_{M}(.,.): \bSigma \rightarrow \mathbb{R}$
\begin{align*}
c_{M}((\sigtens_{M},\sigma_{\theta \theta}), p ) &= (\rhotens_{M}, p).
\end{align*}

The meridian problem with weak symmetry is: 
\textit{Given $\mbf{f}_{M} \in {}_{1}\mbf{L}^{2}_{M}(\Omega)$, find
 $((\sigtens_{M}, \sigma_{\theta \theta}), \mbf{u}_{M}, p) \in \bSigma \times U \times Q$ such that
 for all $((\tautens_{M}, \tau_{\theta \theta}), \mbf{v}_{M}, q) \in \bSigma \times U \times Q$}
\begin{align*}
a_{M}((\sigtens_{M},\sigma_{\theta \theta}),(\tautens_{M},\tau_{\theta \theta})) + b_{M}((\tautens_{M},\tau_{\theta \theta}), \mbf{u}_{M} ) + c_{M}((\tautens_{M},\tau_{\theta \theta}), p ) &= 0 \\
b_{M}((\sigtens_{M},\sigma_{\theta \theta}), \mbf{v}_{M} ) &= (\mbf{f},\mbf{v}_{M}) \\
c_{M}((\sigtens_{M},\sigma_{\theta \theta}), q ) & = 0 \, .
\end{align*}

\subsubsection{Azimuthal Problem}
Finally we consider the azimuthal problem.  
Recall that for $\sigtens \in \Htens_{A}(\axidiv , \brOmega ; \mathbb{M}^{3})$, $\mathcal{A} \sigtens$ has the form
\begin{align*}
\mathcal{A} \sigtens = \dfrac{1}{2 \mu} \begin{pmatrix} 0 & \sigma_{r \theta} & 0 \\ \sigma_{\theta r} & 0 & \sigma_{\theta z} \\ 0 & \sigma_{z \theta} & 0 \end{pmatrix}.
\end{align*}

Thus, for all $\tautens \in \Htens_{A}(\axidiv , \brOmega; \mathbb{M}^{3})$, the first term of \eqref{eq:problem_one_strong_1} is  
\begin{align}
\label{eq:azimuthal_derivation_A}
\mathcal{A} \; \sigtens : \tautens &= \dfrac{1}{2 \mu} (\sigma_{r \theta} \tau_{r \theta} + \sigma_{\theta r} \tau_{\theta r} + \sigma_{\theta z} \tau_{\theta z} +  \sigma_{z \theta} \tau_{z \theta}),
\end{align}
and using reduced form notation, $ \mathcal{A} \sigtens : \tautens = \frac{1}{2 \mu} \sigtens_{A} : \tautens_{A}.$
Integrating (\ref{eq:azimuthal_derivation_A}) defines the bilinear form 
$a_{A}(\cdot,\cdot) : \Sigma \times \Sigma \rightarrow \mathbb{R}$,
\begin{align*}
a_{A}(\sigtens_{A},\tautens_{A}) &= \frac{1}{2 \mu} (\sigtens_{A}, \tautens_{A})_{A}.
\end{align*}

For the second term in \eqref{eq:problem_one_strong_1}, taking $\mbf{u} \in {}_{1} \mbf{L}^{2}_{A}(\brOmega)$ and $\tautens \in \Htens_{A}(\axidiv , \brOmega; \mathbb{M}^{3})$, 
\begin{align*}
\begin{split}
-\int_{\Omega} \nabla_{\text{axi}} \mbf{u} :  \tautens \; r\; d \Omega &= -\int_{\Omega} \begin{pmatrix} 0 & \dfrac{- u_{\theta}}{r} & 0 \\ \dfrac{\partial u_{\theta}}{\partial r} & 0 & \dfrac{\partial u_{\theta}}{\partial z} \\ 0 & 0 & 0 \end{pmatrix} \begin{pmatrix} 0 & \tau_{r \theta} & 0 \\ \tau_{\theta r} & 0 & \tau_{\theta z} \\ 0 & \tau_{z \theta} & 0 \end{pmatrix} \; r \; d \Omega\\ 
&= -\int_{\Omega} \nabla u_{\theta} \cdot \begin{pmatrix} \tau_{\theta r} \\ \tau_{\theta z} \end{pmatrix} \; r \; d \Omega + \int_{\Omega} \dfrac{\tau_{r \theta} \; u_{\theta}}{r} \; r \; d \Omega.
\end{split}
\end{align*}

Integrating the first term by parts
\begin{align*}
-\int_{\Omega} \nabla u_{\theta} \cdot \begin{pmatrix} \tau_{\theta r} \\ \tau_{\theta z} \end{pmatrix} \; r \; d \Omega = -\int_{\partial \Omega} u_{\theta} \begin{pmatrix} \tau_{\theta r} \\ \tau_{\theta z} \end{pmatrix} \cdot \mbf{n} \; r \; \partial \Omega + \int_{\Omega} u_{\theta} \; \nabla \cdot \left(r\; \begin{pmatrix} \tau_{\theta r} \\ \tau_{\theta z}\end{pmatrix} \right) \; d \Omega.
\end{align*}

Since $\tautens \cdot \mbf{n} = \mbf{0}$ on $\partial \Omega$, it follows that
\begin{align*}
\begin{split}
-\int_{\Omega} \nabla_{\text{axi}} \mbf{u} :  \tautens \; r\; d \Omega &= \int_{\Omega} \mbf{u} \cdot (\axidiv \tautens) \; r \; d \Omega.
\end{split}
\end{align*}

Using the reduced form notation,
\begin{align*}
\begin{split}
\int_{\Omega} \mbf{u} \cdot (\nabla_{\text{axi}} \cdot \tautens) \; r \; d \Omega &= \int_{\Omega} u_{\theta} \; (\partial_{r} \tau_{r \theta} + \partial_{z} \tau_{z \theta} + \frac{1}{r} (\tau_{r \theta} + \tau_{\theta r} ) ) \; r \; d \Omega \\
& = \int_{\Omega} u_{\theta} \; \nabla_{\text{axi}} \cdot \tautens_{A} \; r \; d \Omega.
\end{split}
\end{align*}

This defines the bilinear form $b_{A}(.,.) : Q \times \Sigma \rightarrow \mathbb{R}$ as
\begin{align*}
b_{A}(u_{\theta}, \tautens_{A}) = (u_{\theta}, \nabla_{\text{axi}} \cdot \tautens_{A}).
\end{align*}

For $\sigtens \in \Htens_{A}(\nabla \cdot , \brOmega, \mathbb{M}^{3})$, multiplying the left hand side 
of (\ref{eq:problem_one_strong_2}) with a test function $\mbf{v} \in _{1} \mbf{L}^{2}_{A}(\brOmega)$ and integrating 
over $\Omega$ gives
\begin{align*}
\begin{split}
\int_{\Omega} (\nabla_{\text{axi}} \cdot \sigtens) \cdot \mbf{v} \; r \; d \Omega &= \int_{\Omega} (\partial_{r} \sigma_{r \theta} + \partial_{z} \sigma_{z \theta} + \dfrac{1}{r} (\sigma_{r \theta} + \sigma_{\theta r} ) ) v_{\theta} \; r \; d \Omega \\
&= b_{A}(v_{\theta}, \sigtens_{A}).
\end{split}
\end{align*}

Therefore, the weak form of the azimuthal problem can be defined as: 
\textit{Given $f_{\theta} \in {}_{1}L^{2}(\Omega)$, 
find $(\sigtens_{A},u_{A}) \in \Sigma \times Q$ such that for all $(\tautens_{A},v) \in \Sigma \times Q$}
\begin{align*}
a_{A}(\sigtens_{A}, \tautens_{A}) + b_{A}( u ,  \tautens_{A} ) & = 0 \\
b_{A}( v, \sigtens_{A}) & = (f_{\theta}, v) \, .
\end{align*}



\end{document}